\documentclass[oneside,10pt,reqno]{amsart}
\makeatletter
\@namedef{subjclassname@2020}{%
  \textup{2020} Mathematics Subject Classification}
\makeatother

\usepackage{amsmath,amsthm,amsfonts,amssymb,bm}
\usepackage[pagebackref=true]{hyperref}

\usepackage{color,graphicx,graphics}
\usepackage{mathrsfs,amssymb,bm,enumerate}
\usepackage[bbgreekl]{mathbbol}
\DeclareSymbolFontAlphabet{\mathbb}{AMSb}
\DeclareSymbolFontAlphabet{\mathbbl}{bbold}

\usepackage{geometry}
\numberwithin{equation}{section}
\newtheorem{theorem}{\qquad Theorem}[section]
\newtheorem{lemma}[theorem]{\qquad Lemma}
\newtheorem{corollary}[theorem]{\qquad Corollary}

\newtheorem{remx}[theorem]{Remark}
\newenvironment{remark}
{\pushQED{\qed}\remx}
{\popQED\endremx}

\newtheorem{proposition}{Proposition}[section]

\newcommand{\dd}{\;{\rm d}}
\newcommand{\ff}{\varphi}
\newcommand{\ee}{{\rm e}}
\newcommand{\ii}{{\rm i}}
\newcommand{\na}{\mathscr{A}}
\newcommand{\nb}{\mathcal{B}}
\newcommand{\nbb}{\mathscr{B}}

\newcommand{\nii}{\mathcal{I}}
\newcommand{\niii}{\mathscr{I}}

\newcommand{\nl}{\mathscr{L}_\omega}
\newcommand{\nll}{\mathscr{L}_\lambda}

\newcommand{\nq}{\mathscr{Q}}

\newcommand{\scal}[1]{\left\langle #1 \right\rangle}
\newcommand{\sett}[1]{\left\{   #1   \right\}}
\newcommand{\norm}[1]{\left\|   #1   \right\|}
\newcommand{\abso}[1]{\left|   #1   \right|}
\newcommand{\la}{\langle}
\newcommand{\ra}{\rangle}

\def\R{\mathbb{R}}
\def\N{\mathbb{N}}
\def\T{\mathbb{T}}
\def\M{\mathbb{M}}
\def\Z{\mathbb{Z}}
\def\C{\mathbb{C}}

\def\E{\mathbb{E}}

\def\A{\mathcal{A}}

\allowdisplaybreaks

\newcommand{\x}{{H_{x,y}^\sigma}}
\newcommand{\des}{(-\Delta)^\sigma}

\newcommand{\dess}{(-\Delta)^{\frac \sigma2}}
\newcommand{\paxs}{(-\Delta_x)^\sigma}
\newcommand{\pax}{-\Delta_x}

\newcommand{\al}{\alpha}
\newcommand{\rr}{\mathbb{R}}
\newcommand{\lt}{{L_{x,y}^2}}

\newcommand{\rn}{{\mathbb{R}^d}}

\newcommand{\fm}{\mathfrak{M}}
\newcommand{\fj}{\mathfrak{J}}
\newcommand{\fs}{\mathfrak{S}}

\newcommand{\vare}{\varepsilon}
\newcommand{\bg}{\Big}
\newcommand{\pt}{\partial}
\newcommand{\mK}{{Q}}
\newcommand{\mI}{{I}}
\newcommand{\mH}{{E}}
\newcommand{\mM}{{M}}
\newcommand{\ba}{\mathbf{a}}
\newcommand{\bb}{\mathbf{b}}
\newcommand{\br}{\mathbf{r}}
\newcommand{\wlim}{{\text{w-lim}}\,}
\newcommand{\bs}{\mathbf{s}}
\newcommand{\diag}{L_t^\ba L_x^\br}
\newcommand{\mD}{{\mathcal{D}}}
\newcommand{\tm}{{\tilde{m}}}
\newcommand{\mA}{{A}}
\newcommand{\tdu}{{\phi}}

\begin{document}
\address{Amin Esfahani
\newline \indent Department of Mathematics, Nazarbayev University\indent
\newline \indent  Astana 010000, Kazakhstan.\indent
}
\email{saesfahani@gmail.com,  amin.esfahani@nu.edu.kz}

\address{Hichem Hajaiej
\newline \indent Department of Mathematics, Cal State LA\indent
\newline \indent  Los Angeles CA 90032, USA.\indent }
\email{hhajaie@calstatela.edu}

\address{Yongming Luo
\newline \indent Faculty of Computational Mathematics and Cybernetics \indent
\newline \indent Shenzhen MSU-BIT University, China\indent
\newline \indent  International University Park Road 1, Shenzhen, Guangdong, China.\indent
}
\email{luo.yongming@smbu.edu.cn}

\address{Linjie Song
	\newline \indent Institute of Mathematics, AMSS, Chinese Academy of Science, Beijing 100190, China,\indent
	\newline \indent  University of Chinese Academy of Science, Beijing 100049, China.\indent }
\email{songlinjie18@mails.ucas.edu.cn}

\title[The FNLS on the waveguide manifolds]{On the focusing fractional nonlinear Schr\"{o}dinger equation   on the waveguide manifolds}
\author{Amin Esfahani, Hichem Hajaiej, Yongming Luo and Linjie Song}

	
\maketitle

	\begin{abstract}
In this paper, we consider the focusing fractional nonlinear Schr\"{o}dinger equation (FNLS) on the waveguide manifolds $\rr^d\times\T^m$ both in the isotropic and anisotropic case. Under different conditions, we establish the existence and periodic dependence of the ground states of the focusing FNLS. In the intercritical regime, we also establish the large data scattering for the anisotropic focusing FNLS by appealing to the framework of semivirial vanishing geometry.
	\end{abstract}

	\tableofcontents

%
%
%
%
%

	\section{Introduction}

	The paper is devoted to the study of the focusing fractional nonlinear Schr\"odinger equation (FNLS) on waveguide manifolds which arises in the study of nonlinear optics and Bose-Einstein condensates. Historically, the FNLS was introduced in \cite{laskin1,Laskin2002} by an expansion of the Feynman path integral from Brownian-like to L\'{e}vy-like quantum mechanical paths by applying  the theory of functionals over functional measure generated by the L\'{e}vy stochastic process. It also appears in the continuum limit of discrete models with long-range interactions, the description of Boson stars as well as in water wave dynamics. For more details on its physical background, see for instance \cite{Froehlich2007,Kirkpatrick2013}.

Mathematically, the presence of the fractional Laplacian introduces nonlocal effects which lead to the long-range interactions and memory effects affecting different physical properties of the system such as the dispersion, propagation and stability. Different problems for the FNLS on the purely Euclidean spaces $\R^d$ have been nowadays well-studied, we refer e.g. to \cite{dinh,Hong2015} for the well- and ill-posedness results of the FNLS in $H^\sigma(\rn)$, and to \cite{Felmer2012,Secchi2013,Cho2014,Guo2012,Feng2020,Feng2018,Peng2018,dinh2019,Li2022,Pava2018} for the existence and (in-)stability results of the standing wave solutions of the FNLS and other types of fractional dispersive equations such as the fractional KdV. We also underline the ground-breaking works \cite{Frank2013,Frank2016} by Frank, Lenzmann and Silvestre, where the authors employed various mathematical techniques and tools such as the variational methods, energy estimates and potential theory, in order to prove that the FNLS on $\R^d$ has a unique radial ground state solution (up to symmetries).

Recently, there has been an increasing interest in investigating the dispersive equations on the waveguide manifolds $\R^d\times\T^m$, whose mixed type nature leads to various new mathematical challenges. Different results for the NLS with integral Laplacian on the waveguide manifolds have been nowadays abundantly established, see e.g. \cite{TNCommPDE,TzvetkovVisciglia2016,TTVproduct2014,Ionescu2,HaniPausader,CubicR2T1Scattering,R1T1Scattering,Cheng_JMAA,ZhaoZheng2021,RmT1,YuYueZhao2021,Luo_Waveguide_MassCritical,Luo_inter,Luo_energy_crit,luo2022sharp}. In this paper, we aim to adopt the previous points of view and establish some first results concerning the properties of the standing wave solutions of the focusing FNLS, such as existence, stability and large data scattering results, on waveguide manifolds, which seem not to exist in literature to the date. We believe that such new results shall provide a more precise description for the current models applied in different physical applications. As we shall also see, the nonlocal and mixed nature of the model also produces new challenges which can not be solved by classical methods. We will hence invoke several new ideas to overcome these difficulties.

\subsection{The isotropic model}

In the first part, we study the isotropic focusing FNLS
	\begin{equation}\label{frac-nls}
		\ii u_t+	\des u =|u|^{\al}u,
	\end{equation}
	where $u=u(x,y)$ and $(x,y)\in\Omega=\rn\times\T^m$ with $\T=\R/2\pi\Z$, $m,n\geq1$, $\sigma\in(0,1)$ and $\al>0$. Formally, \eqref{frac-nls} possesses two conservation laws (energy and mass invariants)
	\begin{align*}
	\E(u)=\frac12 \|\dess u\|_{2} -\frac{1}{\al+2}\|u\|_{\alpha+2}^{\al+2}\quad\text{and}\quad
\M(u)=\|u\|_{2}^2.
	\end{align*}
Inspired by these quantities, we are particularly interested in the so-called \textit{standing wave} solutions which play an important role in the physical studies, as they might be the only observable quantities in physical applications. Here, we refer a function $\varphi$ to as a standing wave solution of \eqref{frac-nls} if and only if $\varphi$ solves the stationary FNLS
\begin{equation}\label{frac-c}
		\des \varphi+\omega \varphi=|\varphi|^{\al}\varphi,\qquad(x,y)\in \Omega=\rn\times\T^m.
	\end{equation}
It is easy to verify that if $\ff$ is a solution of \eqref{frac-c}, then $u(t,x,y)=\exp(-\ii\omega t)\ff(x,y)$ also solves \eqref{frac-nls}. To study the existence of standing wave solutions of \eqref{frac-c}, we will instead study a variational problem associated to \eqref{frac-c} which is formulated in terms of  $\E(u)$ and $\M(u)$. Such solutions will also be called the \textit{ground state} solutions, as they possess the least energy among all nontrivial solutions to \eqref{frac-c} in $H_{x,y}^\sigma$. Moreover, the regime for the parameter $\alpha$ that we are interested in is determined by $0<\al<2_\sigma^\ast:=\frac{4\sigma }{d+m-2\sigma }$ when $d+m>2\sigma $, and $0<\al<\infty$ when $d+m\leq2\sigma $. Here, the number $2_\sigma^\ast$ plays a special role, as it stands for the usual Sobolev exponent in the Sobolev embedding $H_{x,y}^\sigma\hookrightarrow \alpha+2$.

We briefly summarize the strategy for establishing the existence of solutions of \eqref{frac-c}. To determine the ground state solutions of equation \eqref{frac-c}, we will consider a minimization problem on a Nehari-type manifold (refer to \eqref{minimization-nehari}) or on a manifold with prescribed mass (refer to \eqref{mc def}). A key ingredient here will be a fractional Gagliardo-Nirenberg inequality of semiperiodic type that was originally proved in \cite{TTVproduct2014} in the integral case. This, in combination of suitable concentration compactness arguments, will enable us to infer the existence of a non-vanishing weak limit of the minimizing sequence (up to translations). Finally, by invoking standard variational arguments we then prove that the obtained non-vanishing weak limit is our sought of the solution of \eqref{frac-c}, see Theorem \ref{general-ground-state} and different existence results in Section \ref{sec 2.4}. Additionally, we also prove the higher regularity of the obtained standing wave solutions by appealing to the Poisson summation formula and a modification of the Mikhlin-Hörmander theorem in the context of $\mathbb{R}^d\times\mathbb{T}^m$, see Proposition \ref{regularity}.

Notice that by assuming that \eqref{frac-c} is periodically trivial, it reduces to the ordinary FNLS on the Euclidean space $\R^d$. It is therefore a natural and interesting question whether the previously obtained ground state solutions will coincide with the ones on $\R^d$. To answer this question, we follow closely the strategy in \cite{TTVproduct2014} and rescale \eqref{frac-c} into an equation where the rate of change of $\omega$ becomes more traceable and controllable, see \eqref{scaled-equation}. By analyzing the variation of $\omega$ associated with \eqref{scaled-equation}, we then determine a threshold $\omega^\ast$ at which the ground states obtained from the variational problem \eqref{minimization-nehari} become semitrivial for $\omega< \omega^\ast$, and nontrivial for $\omega> \omega^\ast$.
	
Inspired by the fact that semitrivial solutions take place for small $\omega$, we are naturally interested in the behaviors of the ground state solutions in another extremal situation $\omega\to\infty$, in which case we expect that the ground state solutions will recover the ones on $\R^{d+m}$. In order to do so, we introduce the $L$-Sobolev space and rescale \eqref{frac-c} to transfer the effects of $\omega$ onto this Sobolev space, with $L=\omega^{\frac{1}{2\sigma}}$. By employing a suitable extension operator associated with the scale $L$, the limit behavior of solutions of \eqref{frac-c} can be traced within the topology $H^\sigma(\mathbb{R}^{d+m})$, which in turn enables us to demonstrate that the ground states of \eqref{frac-c} converge to a ground state of \eqref{frac-c} on $\mathbb{R}^{d+m}$, see Theorem \ref{limit-g-theo} and Corollary \ref{strong-conv-coroll}. 

At the end of this subsection, we would also like to underline that, in comparison to the FNLS studied on the Euclidean spaces $\R^d$, the major difficulty by studying the isotropic FNLS on waveguide manifolds is attributed to the fact that the fractional Laplacian of the scaling function $t\mapsto t^{\frac{d}{2}}u(tx,y)$ becomes nonlocal, which being a main obstacle to derive the standard analytical tools such as the Pohozaev identity and monotonicity formulas in a classical way. Such difficulties will be overcome by certain subtle and novel applications of the fractional calculus which also lead to the rather technical proofs given in the present paper.

\subsection{The anisotropic model}

Having the ground states of \eqref{frac-nls} characterized, it is natural to follow the ideas of \cite{weinstein,Glassey1977}  to see whether a  sharp threshold	for the bifurcation of global well-posedness and finite time blow-up solutions is achievable. A major barrier in this direction is the isotropy of the dispersion of \eqref{frac-nls}. Indeed, the operator $\dess $
	corresponds to the multiplier $(|\xi|^2+|k|^2)^{\sigma}$, where $\xi$ is the Fourier variable corresponding to $x\in\rn$ and $k$ the one to $y\in\T^m$. On strong contrary to the integral analogue of \eqref{frac-nls} studied in \cite{TTVproduct2014,TNCommPDE,TzvetkovVisciglia2016}, the ideas of \cite{TTVproduct2014,TNCommPDE,TzvetkovVisciglia2016}, which are based on the independence of the Laplacians along the $x$- and $y$-direction respectively, seem not to be applicable to find  the Strichartz estimates of \eqref{frac-nls}. This motivates us to turn our attention to another version of the FNLS that rather admits an anisotropic nature. More precisely, in the second part of this paper, we consider the anisotropic FNLS
	\begin{equation}\label{nls}
		\ii u_t-\bg((-\Delta_{x})^\sigma+(-\Delta_{y})^\sigma \bg)u=-|u|^\alpha u,\qquad \quad(x,y)\in  \rn\times\T^m.
	\end{equation}
The model arises from physical applications where anomalous diffusion takes place, and the symbol of the anisotropic Laplacian given by \eqref{nls} corresponds to a Levy process with a singular Levy measure, see for instance \cite{Sire} and the references therein. From a mathematical point of view, the practice of studying \eqref{nls} is that the anisotropic symbol of the Laplacian will enable us to find the Strichartz estimates in the study of the Cauchy problem associated with \eqref{nls}. On the other hand, both \eqref{frac-nls} and \eqref{nls} share the same topological spaces (in which the well-posed result are established) and criticality indices due to the equivalence of the symbols $(-\Delta_{x})^\sigma+(-\Delta_{y})^\sigma$ and $(-\Delta )^\sigma$. Hence many techniques for studying the isotropic model \eqref{frac-nls} will indeed remain valid for the anisotropic FNLS \eqref{nls} with slight modifications.

Despite its physical significance, there are only few mathematical papers devoted to the study on \eqref{nls}. To our best knowledge, the first work in this direction was given by Sire et al. in their very recent paper \cite{Sire}, where the defocusing analogue of \eqref{nls} on the special manifold $\R^d\times\T$ was studied. Therein, the authors have first established the local well-posedness of \eqref{nls} in the energy space based on the contraction principle (which in fact holds both in the defocusing and focusing case). Moreover, by making use of suitable Morawetz-type estimates, the authors were also able to prove the global well-posedness and large data scattering for the defocusing \eqref{nls} under some additional symmetry conditions and conditions on the exponents.

	Inspired by \cite{Sire}, we consider in this paper the focusing \eqref{nls} on the waveguide manifold $\rr^d\times\T$ with $\alpha$ lying in the intercritical regime $\alpha\in(\frac{4\sigma}{d},\frac{4\sigma}{d+1-2\sigma})$. In comparison to the defocusing model, the new challenge arising in the focusing case is to set up a suitable framework that determines the bifurcation of the global scattering and finite time blowing-up solutions. For this purpose, the so-called \textit{semivirial-vanishing geometry} has been recently established by the third author in his papers \cite{Luo_inter,Luo_energy_crit,luo2022sharp} to study the focusing integral NLS on waveguide manifolds in the intercritical regime. In the fractional context, we shall also follow closely the strategy in \cite{Luo_inter,Luo_energy_crit,luo2022sharp} and invoke the theory of the semivirial-vanishing geometry to establish our results. We refer to Theorem \ref{thm norm ground} and \ref{thm large data scattering} for details.

\begin{remark}
\normalfont
During the preparation of the paper, the authors have learned that a similar framework of the semivirial-vanishing geometry has also been exploited in the recent paper \cite{ArdilaCarles2021} for the study of the NLS with partial confinement. There are however several major differences between \cite{ArdilaCarles2021} and the works \cite{Luo_inter,Luo_energy_crit,luo2022sharp}. For example, the uncertainty principle applied in \cite{ArdilaCarles2021} was accordingly replaced by a scale-invariant Gagliardo-Nirenberg inequality in the waveguide setting (see also Lemma \ref{GN}). On the other hand, the semivirial-vanishing geometry applied in \cite{ArdilaCarles2021} was mainly for the purpose of establishing the large data scattering and blow-up results, while in \cite{Luo_inter,Luo_energy_crit,luo2022sharp} it was also used to establish the existence of the normalized ground states. For more differences, we refer to \cite{ArdilaCarles2021} and \cite{Luo_inter,Luo_energy_crit,luo2022sharp} respectively.
\end{remark}

The rest of the paper is organized as follows: In Sections \ref{Sec2} and \ref{sec 2.4}, we study the existence of the ground state solutions with fixed frequencies and normalized ground states of \eqref{frac-c} and their properties. In Section \ref{Sec4} we prove the large data scattering for \eqref{nls}.

\subsection{Notations and definitions}
For simplicity, we often omit the explicit mention of the underlying domains for function spaces and incorporate this information in their indices. For example, $L_x^2$ refers to $L^2(\mathbb{R}^d)$ and $H_{x,y}^\sigma$ refers to $H^\sigma(\mathbb{R}^d \times \T^m)$, and so on. However, when the space is involved with time, we still display the underlying temporal interval such as $L_t^pL_x^q(I)$, $L_t^\infty L_{x,y}^2(\R)$, etc. The norm $\|\cdot\|_p$ is defined as $\|\cdot\|_p:=\|\cdot\|_{L_{x,y}^p}$.

We also recall the following equivalent definitions of the fractional Laplacians (see \cite{Ambrosio2017,sobolev_torus,DiNezza2012}) which will be used throughout the paper:
	\[
	\begin{split}
		\des u(x,y)&=
		C_{d,m,\sigma}\int_{\rr^d\times\rr^m}
		\frac{2u(x,y)-u(x+\tau,y+z)-u(x-\tau,y-z)}{|(\tau,z)|^{d+m+2\sigma }}\dd\tau\dd z
		\\&
		=C_{d,m,\sigma}\int_{\rr^d\times\rr^m}\frac{u(x,y)-u(\tau,z)}{|(x,y)-(\tau,z)|^{d+m+2\sigma }}\dd\tau\dd z
		\\&
		=
		C_{d,m,\sigma}\sum_{k\in\Z^m}\int_{\R^d\times\T^m}
		\frac{u(x,y)-u(\tau,z)}{(|x-\tau|^2+|y-z-2\pi k|^2) ^{\frac{d+m+2\sigma }{2}}}\dd\tau\dd z,
	\end{split}
	\]
	where $C_{d,m,\sigma}=\frac{2^{2\sigma }\Gamma(s+\frac{d+m}{2})}{\pi^\frac{d+m}2|\Gamma(-\sigma)|}$.
	We can also  express $	\des u(x,y)$  in terms of heat kernel, i.e.
	\[
	\des u(x,y)=
	C_{d,m,\sigma}\int_{\Omega}\int_0^\infty
	(u(x,y)-u(\tau,z))W_t(x-\tau,y-z)t^{-1-\sigma}\dd t\dd\tau\dd z,
	\]
	where
	\[
	W_t(x,t)=t^{-\frac{d+m}{2}}\sum_{k\in\Z^m}\ee^{-\frac{|x|^2+|y-2\pi k|^2}{t}}.
	\]

\section{Isotropic case: Ground states with fixed frequencies}\label{Sec2}

	\subsection{Existence of ground states with fixed frequencies}
We begin our study with the stationary isotropic model \eqref{frac-c}. To formulate the results, we define the following natural minimization problem

	\begin{equation}\label{minimization-nehari}
		c_\omega=\inf\sett{\A_\omega(u),\;u\in\x\setminus\{0\},\;\nb_\omega(u)=0 },
	\end{equation}
	
	where
	\begin{align*}
	\A_\omega(u)=
	\frac12 \|\dess u\|_{2}+  \frac\omega2\|u\|_{2}-\frac{1}{\al+2}\|u\|_{\alpha+2}^{\al+2},\quad
\nb_\omega(u)=\la\A_\omega'(u),u\ra.
	\end{align*}
%
Our main result on the existence of ground state solutions with fixed frequencies is given as follows:
	\begin{theorem}\label{general-ground-state}
Let $\sigma\in 0,1$ and $\alpha\in(0,2_\sigma^*)$, where
\begin{align*}
2_\sigma^*=
\left\{
\begin{array}{rl}
\frac{4\sigma }{d+m-2\sigma },&d+m>2\sigma,\\
\infty,&d+m\leq 2\sigma
\end{array}
\right.
\end{align*}
Then for any  $\omega>0$, there exists a ground state $\ff\in H_{x,y}^\sigma\setminus\{0\}$ of \eqref{frac-c} satisfying $\A_\omega(\ff)=c_\omega$. Moreover,
		\begin{equation}\label{minimization-leq}
			c_\omega=\sett{\nii_\omega(u),\;u\in\x\setminus\{0\},\;\nb_\omega(u)\leq0 },
		\end{equation}
		where $\nii_\omega=\A_\omega-\frac1{\al+2}\nb_\omega$. Furthermore, the mapping $\omega\mapsto c_\omega$ is continuous and strictly increasing on $(0,\infty)$.
	\end{theorem}

\begin{remark}
\normalfont
We underline that in comparison to the functional $\A_\omega$, the functional $\nii_\omega$ has the advantage that it is non-negative, which is more practical from a variational point of view. We will also utilize this property at many places in the upcoming proofs.
\end{remark}

\begin{remark}\label{remark-semitri}
\normalfont
		Notice that if $u$ is independent of $y$, then by  a change of variable we observe after some calculations that
		\[
		\begin{split}
			\des u(x)&=
			C_{d,m,\sigma}\int_{\rr^d\times\rr^m}\frac{u(x)-u(\tau)}{|(x,y)-(\tau,z)|^{d+m+2\sigma }}\dd\tau\dd z
			\\&
			= C_{d,m,\sigma}\int_{\rr^d\times\rr^m}\frac{u(x)-u(\tau)}{(|x -\tau|^2+|z|^2)^{\frac{d+m+2\sigma }{2}}}\dd\tau\dd z
			\\&
			=C_{d,m,\sigma}\int_{\rr^d\times\rr^m}\frac{u(x)-u(\tau)}{|x -\tau|^{d+2\sigma }(1+|z|^2)^{\frac{d+m+2\sigma }{2}}}\dd\tau\dd z
			\\&
			= C_{d,m,\sigma}\int_{\rr^m}\frac{\dd z}{(1+|z|^2)^{\frac{d+m+2\sigma }{2}}}  \int_{\rr^d}\frac{u(x)-u(\tau)}{|x -\tau|^{d+2\sigma }}\dd\tau
			\\&
			= \frac{C_{d,m,\sigma}}{C_{n,\sigma}}(-\Delta_x)^{\sigma}u(x)
			\int_{\rr^m}\frac{ \dd z}{(1+|z|^2)^{\frac{d+m+2\sigma }{2}}}
			=(-\Delta_x)^{\sigma}u(x).
		\end{split}
		\]
		Hence, any solutions of 	
		\begin{equation}\label{y-free-gs}
			\paxs Q_\omega+\omega Q_\omega=|Q_\omega|^{\al}Q_\omega,\qquad x\in\rn
		\end{equation}
		satisfying \eqref{frac-c} are called the \textit{semitrivial} solutions of \eqref{frac-c}. It is well-known (see \cite{Frank2016}) that \eqref{y-free-gs} possesses   the unique (up to translation) ground state $Q_\omega$ extracting from the minimization problem
		\[
		\bbnu_\omega=\inf\sett{\tilde\A_\omega(u),\;u\in H^\sigma(\rn)\setminus\{0\},\;\tilde {\mathcal{B}}_\omega(u)=0 },
		\]
		where
		\[
		\tilde	\A_\omega(u)=
		\frac12 \|(-\Delta_x)^{\frac \sigma 2} u\|_{L^2(\rn)}^2+  \frac\omega2\|u\|_{L^2(\rn)}-\frac{1}{\al+2}\|u\|_{L^{\al+2}(\rn)}^{\al+2}
		\]
		and
		$\tilde\nb_\omega(u)=\la\tilde\A_\omega'(u),u\ra$. Moreover, we have (analogous to Theorem \ref{general-ground-state})   that
		\[
		\bbnu_\omega=\sett{\tilde\nii_\omega(u),\;u\in H^\sigma(\rn)\setminus\{0\},\;\tilde\nb_\omega(u)\leq0},
		\]
		where $\tilde\nii_\omega=\tilde\A_\omega-\frac1{\al+2}\tilde\nb_\omega$. It is straightforward to verify that
		$Q_\omega(x)=\omega^{\frac1{\al}}Q_1(\omega^{\frac1{2\sigma }}x )$, $ \bbnu_\omega=\omega^{\frac{\al+2}{\al}-\frac1{2\sigma }} \bbnu_1$ and $\A_\omega(Q_1)=(2\pi)^m \bbnu_1$. In Section \ref{sec 2.2} we will consider the periodic dependence problem of the ground state solutions of \eqref{frac-c}. We shall see that the previous simple scaling identities of the semitrivial solutions will play a fundamental role in the underlying analysis.
	\end{remark}
To prove Theorem \ref{general-ground-state}, we begin with stating the following useful Gagliardo-Nirenberg inequalities on product spaces. While the first one is a well-known inequality (for which we omit the proof), we shall borrow an idea from \cite{TTVproduct2014} to prove the second one.


\begin{lemma}[Gagliardo-Nirenberg inequality on $\R^d\times\T^m$, I]\label{localized-GN1}
For $\alpha\in (0,\frac{4\sigma}{d+m-2\sigma}]$ the following inequality holds:
		\begin{align}
\|u\|_{\alpha+2}
			&\leq C\|u\|^{1-\frac{\alpha(d+m)}{2\sigma(\alpha+2)}}_{2}\|u\|^{\frac{\alpha(d+m)}{2\sigma(\alpha+2)}}_{H_{x,y}^\sigma}.\label{gn1}
		\end{align}
\end{lemma}

\begin{lemma}[Gagliardo-Nirenberg inequality on $\R^d\times\T^m$, II]\label{localized-GN}
The following inequality holds:
		\begin{align}
\|u\|_{ \frac{2\sigma}{d+m}+2}
			&\leq C\left(\sup_{x\in\rr^d}
			\|u\|_{L^{2}(\rr^d_x\times\T^m)}^{1-\frac{d+m}{d+m+2\sigma}}\right)
			\|u\|_{H_{x,y}^\sigma}^{\frac{d+m}{d+m+2\sigma}},\label{GN-type}
		\end{align}
		where $\rr_x^d=x+[0,1]^d$.
	\end{lemma}
	\begin{proof}
The proof essentially follows the same lines of \cite{TTVproduct2014}. For the readers' convenience we present the complete proof here. Let $\Omega_x:=\rr_x^d\times \T^m$. Fix $x_k\in\rn$ in a way such that $\rn=\bigcup_{k\in\N}\rr^d_{x_k}$, and $ \rr^d_{x_i}\cap \rr^d_{x_j}$ are the null subsets of $\rn$ for $i\neq j$. The classical Gagliardo–Nirenberg inequality implies that
		\[
		\begin{split}
			\|u\|_{L^{2+\frac{2\sigma}{d+m}}(\Omega_{x_k})}
			&\leq C
			\|u\|_{L^{2}(\Omega_{x_k})}^{1-\frac{d+m}{d+m+2\sigma}}
			\|u\|_{H^\sigma(\Omega_{x_k})}^{\frac{d+m}{d+m+2\sigma}}.
		\end{split}
		\]
Now we take $2+\frac{2\sigma}{d+m}$-powers on both side, estimate $\|u\|_{L^{2}(\Omega_{x_k})}$ by $\|u\|_{L^{2}(\rr^d_x\times\T^m)}$ and summing then $k$ up to infer that
\[
		\begin{split}
			\|u\|_{2+\frac{2\sigma}{d+m}}^{2+\frac{2\sigma}{d+m}}
			&\leq C
\left(\sup_{x\in\rr^d}
			\|u\|_{L^{2}(\rr^d_x\times\T^m)}^{(2+\frac{2\sigma}{d+m})(1-\frac{d+m}{d+m+2\sigma})}\right)
			\|u\|_{H_{x,y}^\sigma}^{2}.
		\end{split}
		\]
from which \eqref{GN-type} follows.
	\end{proof}

We will also need the following Weinstein-characterization (see \cite{weinstein}) of the minimizers of \eqref{minimization-leq}.

\begin{lemma}[Weinstein-characterization of the minimizers of \eqref{minimization-leq}]\label{weinmin-lemma}
		If $\ff$ is a minimizer of \eqref{minimization-leq}, then it is a minimizer of
		\begin{equation}\label{wein-min}
			\fj_\omega:=\inf_{u\in H_{x,y}^\sigma\setminus\{0\}}
			\frac{
				\|\dess u\|_{2}+    \omega \|u\|_{2} }
			{\|u\|_{\alpha+2}^{2}}.
		\end{equation}
Moreover, we have $\fj_\omega^{\frac{\al+2}{\al}}=\frac{2(\al+2)}{\al}  c_\omega$.
	\end{lemma}	
\begin{proof}
	Let $u\in H_{x,y}^\sigma\setminus\{0\}$. Then it is easy to see that  there is $\ell=\ell(u)>0$ such that $\nb_\omega(\ell u)=0$. In fact, this can be explicitly calculated:
	\[
	\ell=
	\frac{
		\|\dess u\|_{2}+    \omega \|u\|_{2} }{\|u\|_{\alpha+2}^{\al+2}  	}
	\]
	Then
	\[
	c_\omega\leq\A_\omega(\ell u)
	=\frac{ \al }{2(\al+2)}\left(
	\frac{	\left(\|\dess u\|_{2}+    \omega \|u\|_{2}\right)^{\al+2}}{\|u\|_{\alpha+2}^{2(\alpha+2)}}
	\right)^{\frac1{\al}} =:	\tilde\fj_\omega(u),
	\]
whence $$c_\omega\leq\inf_{u\in H_{x,y}^\sigma\setminus\{0\}}	\tilde\fj_\omega(u).$$
Next, we anticipate the statement $\mathcal{B}_\omega(\phi)=0$, a fact that will be independently proved in the proof of the existence part of Theorem \ref{general-ground-state} given later. Hence
	\[
	\inf_{u\in H_{x,y}^\sigma\setminus\{0\}}\tilde\fj_\omega(u)
	\leq
	\tilde\fj_\omega(\ff)
	=
	\frac{ \al }{2(\al+2)}\left(\|\dess\ff\|_{2}+    \omega \|\ff\|_{2}\right)
	=c_\omega
	,
	\]
	where we used $\nb_\omega(\ff)=0$ in the last equality. Therefore,
	\[
	c_\omega=\inf_{u\in H_{x,y}^\sigma\setminus\{0\}}	\tilde\fj_\omega(u).
	\]
That $\ff$ is a minimizer of \eqref{wein-min} and $\fj_\omega^{\frac{\al+2}{\al}}=\frac{2(\al+2)}{\al}  c_\omega$ then follow from a fundamental rescaling argument. This completes the proof.
\end{proof}

%
	%
	%
	%

We are now ready to give the proof of Theorem \ref{general-ground-state}.

	\begin{proof}[Proof of Theorem \ref{general-ground-state}]
By the definition of $c_\omega$, it is easily seen that \eqref{minimization-leq} holds.
		
Next, let $\{u_k\}$ be a minimizing sequence of \eqref{minimization-nehari}. Then  $\{u_k\}$ is bounded in $\x$. The Sobolev embedding combining with the fact $\nb_\omega(u_k)=0$ shows that $\{u_k\}$ is uniformly bounded in $k$ from below in $L_{x,y}^{\alpha+2}$. By Lemma \ref{localized-GN1} and \ref{localized-GN}, it follows that
		\[
		\sup_{x\in\rr^d}\|u_k\|_{L^2(\rr_x^d\times\T^m)}\gtrsim1.
		\]
		Hence there exist $z_k\in\rr^d $, a subsequence of $\{u_k\}$ (not relabeled) and $\ff\in\x\setminus\{0\}$ such that $w_k:=u_k(\cdot+z_k,\cdot)\rightharpoonup\ff$ in $\x$ as $k\to\infty$. The   Brezis–Lieb lemma implies that
		\begin{equation}\label{BL-est}
			\begin{split}
				\nii_\omega(w_k)-\nii_\omega(w_k-\ff)-\nii_\omega(\ff)\to0  \quad\text{and}\quad
				\nb_\omega(w_k)-\nb_\omega(w_k-\ff)-\nb_\omega(\ff)\to0
			\end{split}
		\end{equation}
		as $k\to\infty$.
		If $\nb_\omega(\ff)>0$, then it is obtained from  \eqref{BL-est} and the fact $\nb_\omega(w_k)=0$ for sufficiently large $k$ that $\nb_\omega(w_k-\ff)<0$. This implies that there is $\tau_k\in(0,1)$ such that $\nb_\omega(\tau_k(w_k-\ff))=0$. Thereby
		\[
		c_\omega\leq\nii_\omega(\tau_k(w_k-\ff))< \nii_\omega(w_k-\ff).
		\]
		Since $\nii_\omega(w_k)\to c_\omega$, we get to the contradiction  $\nii_\omega(\ff)<0$ and thus $\nb_{\omega}(\ff)\leq0$. The lower semicontinuity in $\x$ implies that $\nii_\omega(w_k)\to\nii_\omega(\ff)=c_\omega$. We therefore deduce from \eqref{BL-est} that
		$\nii_\omega(w_k-\ff)\to0$ adducing that $w_k\to\ff$ in $\x$, and $\nb_{\omega}(\ff)=\lim_{k\to\infty}\nb_{\omega}(w_k)=0$.
		Thus $\ff$ is a minimizer of \eqref{minimization-nehari}.
		
		Now by identity \eqref{minimization-leq} and the Lagrange multiplier theorem we know that there is some $\theta\in\rr$ such that $\A_\omega'(\ff)=\theta\nb_{\omega}'(\ff)$. Multiplying the latter by $\ff$ and integrating on $\R^d\times\T^m$, we get $\theta \al \|\ff\|_x^2=\nb_{\omega}(\ff)=0$, whence $\theta=0$ and in turn implies that $\ff$ is a solution of \eqref{frac-c}. Moreover, if $\psi\in\x$ is a nontrivial solution of \eqref{frac-c}, then    $\nb_{\omega}(\psi)=0$ and by definition also $\A_\omega(\ff)\leq \A_\omega(\psi)$, as desired
		
		We finally show the monotonicity and continuity of the mapping $\omega\mapsto c_\omega$. Let $u_{\omega_j}$ be minimizers of $c_{\omega_j}$ with $j=1,2$ such that $\omega_1<\omega_2$.
		By Lemma \ref{weinmin-lemma} it suffices to prove $\fj_{\omega_1}<\fj_{\omega_1}$. Using the definition of $\fj_\omega$ in	\eqref{wein-min} we have
			\[
			\begin{split}
				\fj_{\omega_1}
				&\leq
				\frac{
					\|\dess u_{\omega_2}\|_{2}+    \omega_1 \|u_{\omega_2}\|_{2} }{\|u_{\omega_2}\|_{\alpha+2}^2}
				\\&=
				\frac{
					\|\dess u_{\omega_2}\|_{2}+   \omega_2 \|u_{\omega_2}\|_{2}}{\|u_{\omega_2}\|_{\alpha+2}^2}
				-(\omega_2-\omega_1) \frac{\|u_{\omega_2}\|_{2} }{\|u_{\omega_2}\|_{\alpha+2}^2}	\\&<
				\fj_{\omega_2},
			\end{split}
			\]
			thus $\fj_\omega$ is strictly increasing in $\omega$. In a similar way, we deduce that
			\[
			\begin{split}
				\fj_{\omega_2}
				&\leq
				\frac{
					\|\dess u_{\omega_1}\|_{2}+    \omega_2 \|u_{\omega_1}\|_{2} }{\|u_{\omega_1}\|_{\alpha+2}^2}
				\\&
				\leq \fj_{\omega_1}
				+(\omega_2-\omega_1)
				\frac{\|u_{\omega_1}\|_{2} }{\|u_{\omega_1}\|_{\alpha+2}^2},
			\end{split}
			\]	
			so that
			\[
			\begin{split}
				0\leq  \fj_{\omega_2} - \fj_{\omega_1}
				&\leq
				(\omega_2-\omega_1)
				\frac{\|u_{\omega_1}\|_{2} }{\|u_{\omega_1}\|_{\alpha+2}^2}\\&
				\leq (\omega_2-\omega_1)
				\frac{\|\dess u_{\omega_1}\|_{2}+    \omega_1 \|u_{\omega_1}\|_{2} }{\|u_{\omega_1}\|_{\alpha+2}^2}\\&
				\leq (\omega_2-\omega_1)\fj_{\omega_1}
				.
			\end{split}
			\]
			This implies that
			\[
			|\fj_{\omega_2} - \fj_{\omega_1} |\leq \fj_{\omega_1} |\omega_2-\omega_1|  .
			\]
			We hence infer that $\fj_\omega$ is locally Lipschitz continuous in $\omega$, from which the continuity of the mapping $\omega\mapsto c_\omega$ follows.

		%
		%
		%
	\end{proof}
	%
	%
	%
	%
	%
	%
	%
	%
	%
	%
	%
	Having established the existence results, we also give the decay and regularity properties of the solutions of \eqref{frac-c} in the following proposition.
	
	\begin{proposition}[Regularity of the ground state solutions]\label{regularity}
Any solution of $u\in H_{x,y}^\sigma$ of \eqref{frac-c} with $\omega>0$ belongs to $H_{x,y}^{2\sigma +1}\cap L_{x,y}^\infty$. Moreover, $u$  tends to zero as $|x|\to\infty$.
	\end{proposition}
	\begin{proof}
		The proof follows the same lines as the one for Proposition 3.1 in \cite{Frank2016}, namely, it suffices to show that
		\begin{equation}\label{mikhlin}
			\left\|\frac{(-\Delta )^{\sigma}}{\omega+(-\Delta)^{\sigma}}u\right\|_{q}
			\lesssim \left\| u\right\|_{q}
		\end{equation}
		holds for any $1<q<\infty$.		For this purpose, we first point out that
		$\frac{(-\Delta )^{\sigma}}{\omega+(-\Delta)^{\sigma}}u=(m_\circ(\xi,k)\hat u(\xi,k))^\vee$, where $m_\circ(\xi,k)=\frac{(|\xi|^2+|k|^2)^{\sigma}}{\omega+ (|\xi|^2+|k|^2)^{\sigma}}$. Moreover, for any fixed $k\in\Z^m$, the function $M(k)=m_\circ(\cdot,k)$ is a multiplier on $\mathcal{L}(L^q(\rn))$. On the other hand, since $m_\circ(\xi,\eta)$ with $(\xi,\eta)\in\rn\times\rr^m$ is a multiplier on $L^q(\rn\times\rr^m)$, $M(k)$ is actually in the space $L^q(L^q(\rn),\rr^m)=L^q(\rn\times\rr^m)$. It is easily observed that $M(k)$ depends continuously on $k$ with respect to the norm of $\mathcal{L}(L^q(\rn))$. Thus from the restriction theorem \cite[Theorem 3.8, Chapter VII]{Stein1971} we infer that $M(k)$ is also a multiplier on $L^q(L^q(\rn),\T^m)$.
		
		
		Having \eqref{mikhlin} proved, we show that any solution $u\in H_{x,y}^\sigma$ of \eqref{frac-c} belongs to $L^\infty_{x,y}$.
		By the Sobolev embedding it suffices to consider the case $2\sigma \leq d+m$. We define the kernel $K_\omega$ through its symbol
		\begin{equation}\label{kernel}
			\hat{K}_\omega(\xi,\eta)=\frac{1}{\omega+(|\xi|^2+|k|^2)^{\sigma}},\qquad(\xi, k)\in\rn\times\Z^m.
		\end{equation}
		The function $K_\omega$ can be written as
		\[
		K_\omega(x,y)=\int_0^\infty\ee^{-\omega t}H(x,y,\sigma)\dd t,
		\]
		where $H(x,y,t)$ is the fractional heat kernel  in $\R^d\times\T^m$ given by the inverse Fourier transform of $\exp(-t(|\xi|^2+|k|^2)^{\sigma})$. It is known (see e.g. \cite{Bogdan2007}) that $H(x,y,t)$ in $\rr^{d}\times\rr^m$ can be interpreted as
		\[
		H(x,y,t)\sim\frac{t}{(t^{\frac1\sigma}+|x|^2+|y|^2)^{\frac{d+m+2\sigma }{2}}}.
		\]
		Hence, by using the Poisson summation formula, we obtain in    $\R^d\times\T^m$ that
		\[
		H(x,y,t)\sim \sum_{k\in\Z^m}
		\frac{t}{(t^{\frac1\sigma}+|x|^2+|y+2\pi k|^2)^{\frac{m+d+2\sigma }{2}}}
		\]
		It is straightforward to see that $K_\omega\in L^q(\T^m,L^p(\rn))$ for any $1\leq p,q\leq\infty$ with
		\[
		\frac1p+\frac1q<1+\frac{d+2\sigma }{m}.
		\]
		An application of the Young's inequality then implies that $u\in L^\infty_{x,y}$.
		In conjunction with \eqref{mikhlin} we know that $u\in H^{2\sigma }_{x,y}\cap L^\infty_{x,y}$. By choosing $\mu\in(0,\min\sett{1,2\sigma -(d+m)/q_j})$, where $q_j>(d+m)/2\sigma $, and mimicking the argument of \cite[Theorem 3.4]{Felmer2012}, we deduce that $u\in C^{0,\mu}_{x,y}$. By appealing to the same arguments in \cite[Lemma B.2]{Frank2013}, we may also derive from \eqref{mikhlin} that $u\in H_{x,y}^{2\sigma +1}$. Finally, the fact $u\in H_{x,y}^{2\sigma +1}\cap C_{x,y}^{0,\mu}$ also implies that $u\to0$ as $|x|\to\infty$, as desired.
	\end{proof}
	
		
	\begin{remark}
\normalfont
		The proof of \cite[Theorem 3.1]{Felmer2012} actually shows for $\sigma<1$ that
		\begin{enumerate}[(i)]
			\item $(-\Delta)^{\sigma}:W^{\beta,p}\to W^{\beta-2\sigma ,p}$ for $2\sigma <\beta$;
			\item if $2\sigma <\gamma$, $(-\Delta)^{\sigma}:C^{0,\gamma}\to C^{0,\mu}$ for $\mu\leq\gamma-2\sigma $;
			\item	if $2\sigma >\gamma$, $(-\Delta)^{\sigma}:C^{1,\gamma}\to C^{0,\mu}$ for $\mu\leq1+\gamma-2\sigma $.
		\end{enumerate}
	These thus imply (see \cite{Cabre2014}) that any ground state of \eqref{frac-c} will also belong to $C^{1,\mu}$ for some $\mu\in(0,1)$.
	\end{remark}
	%
	%

	\subsection{Periodic dependence of the ground state solutions}\label{sec 2.2}
	
	As already mentioned in Remark \ref{remark-semitri}, we aim to find conditions under which the minimizers of \eqref{minimization-nehari} are semitrivial or nontrivial. For this purpose, we define
\begin{equation}\label{scaled-func}
		\nq_\omega(x,y)=\omega^{-\frac1{\al}}u_\omega(\omega^{-\frac1{2\sigma }}x,y)
	\end{equation}
	as a scaled ground state of \eqref{frac-c} satisfying
	\begin{equation}\label{scaled-equation}
		\nl ^{\sigma}\nq_\omega+\nq_\omega=|\nq_\omega|^{\al}\nq_\omega,
	\end{equation}
	where $\nl=\pax-\omega^{-\frac{1}{\sigma}}\Delta_y$. Notice that
	\[
	\begin{split}
		\nl^{\sigma}u&=
		\omega^{\frac{m}{2\sigma }}\int_{\rr^d\times\rr^m}\frac{f(x,y)-f(z_1,z_2)}{((x-z_1)^2+\omega^{-\frac{1}{\sigma}}(y-z_2)^2)^{\frac{d+m+2\sigma }{2}}}\dd z_1\dd z_2\\
		&=
		\int_{\rr^d\times\rr^m}\frac{f(x,y)-f(x+z_1,y+\omega^{-\frac{1}{2\sigma }}z_2)}{(z_1^2+z_2^2)^{\frac{d+m+2\sigma }{2}}}\dd z_1\dd z_2.
	\end{split}
	\]
	Similarly to \eqref{frac-c}, the ground states of \eqref{scaled-equation} are derived from the following minimization problem
	\[
	\begin{split}
		\bar c_\omega&=\inf\sett{\na_\omega(u):\;u\in\x\setminus\{0\},
			\;\nbb_\omega(u)=0  }\\&=\inf\sett{\niii_\omega(u):\;u\in\x\setminus\{0\},\;\nbb_\omega(u)\leq 0 },
	\end{split}
	\]
	where
	\[
	\begin{split}
		\na_\omega(u)
		&=
		\frac12\int_{\R^d\times\T^m} u \nl^{\sigma} u\dd x\dd y+  \frac12\|u\|_{2}-\frac{1}{\al+2}\|u\|_{\alpha+2}^{\al+2}\\& =
		\frac12\int_{\R^d\times\T^m} \left| \nl^\frac \sigma 2 u\right|^2\dd x\dd y+  \frac12\|u\|_{2}-\frac{1}{\al+2}\|u\|_{\alpha+2}^{\al+2}\\&
		= \frac{\omega^{\frac{m}{2\sigma }}}2\int_{\R^d\times\T^m}
		\int_{\rr^d\times\rr^m}
		\frac{|u(x,y)-u(z_1,z_2)|^2}{((x-z_1)^2+\omega^{-\frac{1}{\sigma}}(y-z_2)^2)^{\frac{d+m+2\sigma }{2}}}\dd z_1\dd z_2\dd x\dd y+
		\frac12\|u\|_{2}-\frac{1}{\al+2}\|u\|_{\alpha+2}^{\al+2}\\&
		=
		\frac{\omega^{\frac{m}{2\sigma }}}2\int_{\R^d\times\T^m}\int_{\rr^d\times\rr^m}\frac{|u(x+z_1,y+z_2)-u(x,y)|^2}{(z_1^2+\omega^{-\frac{1}{\sigma}}z_2^2)^{\frac{d+m+2\sigma }{2}}}\dd z_1\dd z_2\dd x\dd y+  \frac12\|u\|_{2}-\frac{1}{\al+2}\|u\|_{\alpha+2}^{\al+2}\\&
		=
		\frac{1}2\int_{\R^d\times\T^m}\int_{\rr^d\times\rr^m}\frac{|u(x+z_1,y+\omega^{-\frac1{2\sigma }}z_2)-u(x,y)|^2}{(z_1^2+ z_2^2)^{\frac{d+m+2\sigma }{2}}}\dd z_1\dd z_2\dd x\dd y+  \frac12\|u\|_{2}-\frac{1}{\al+2}\|u\|_{\alpha+2}^{\al+2},
	\end{split}
	\]
$\nbb_\omega(u)=\la\na_\omega'(u),u\ra$ and $\niii_\omega=\na_\omega-\frac1{\al+2}\nbb_\omega$.
	
The following lemma shows that the minimizers of $c_\omega$ are trivial for all sufficiently small $\omega$ (see also \cite{Liouville_type}).

	\begin{lemma}[$y$-independence of ground states with small $\omega$]\label{small-omega}
		Let $u_\omega$ be an optimizer of $c_\omega$ (whose existence is deduced from Theorem \ref{general-ground-state}). Then there exists some $\omega_0\in(0,\infty]$ such that $c_{\omega}=(2\pi)^m\bbnu_{\omega}$ and $u_\omega=Q_\omega$ (up to symmetries) for all $\omega\in(0,\omega_0)$.
	\end{lemma}
\begin{proof}
		We divide the proof into four steps.
		
		\subsubsection*{Step 1: $\bar{c}_\omega\to (2\pi)^m\bbnu_1$ as $\omega\to0$}
		
Since $\nbb_\omega(Q_1)=0$, it follows that $\bar c_\omega\leq (2\pi)^m\bbnu_1$. Moreover, we also have that $\|(-\Delta_y)^{\frac\sigma2}\nq_\omega\|_{2}\to0$ as $\omega\to0$. Indeed, if there exists a sequence $\{w_j\}_j\subset\rr^+$ converging to zero such that  $\liminf_{j\to\infty}\|(-\Delta_y)^{\frac\sigma2}\nq_{\omega_j}\|_{2}\gtrsim1$, then using the Minkowski's inequality we know that $\|\mathscr{L}_{\omega_j}^\frac \sigma 2\nq_{\omega_j}\|_{2}$ blows up as $j\to\infty$.
		This is however a contradiction to
		\begin{equation}\label{bound-equ}
			\|\mathscr{L}_{\omega_j}^\frac \sigma 2\nq_{\omega_j}\|_{2}
			\lesssim
			\niii_{\omega_j}(\nq_{\omega_j})
			\lesssim \bar c_{\omega_j}\leq(2\pi)^m\bbnu_1.
		\end{equation}
		
		Next, we show the stronger statement that $\omega^{-1}\|(-\Delta_y)^{\frac\sigma2}\nq_{\omega}\|_{2}\to0$ as $\omega\to0$.
		To proceed, we multiply \eqref{scaled-equation} by $(-\Delta_y)^{\frac \sigma 2} \bar\nq_\omega$ and integrate the result on $\R^d\times\T^m$ to get
		\[
		\begin{split}
			\left|   \int_{\R^d\times\T^m} (-\Delta_y)^{\frac \sigma 2}\nq_\omega (\nl^{\sigma}\nq_\omega+\nq_\omega)\dd x\dd y\right|
			&=
			\left| \int_{\R^d\times\T^m}(-\Delta_y)^{\frac \sigma 2}\nq_\omega|\nq_\omega|^{\al}\nq_\omega\dd x\dd y\right|
			\\&
			\leq \|(-\Delta_y)^{\frac \sigma 2}\nq_\omega\|_{2}\|\nq_\omega\|_{{2p}}^p\\
			&
			\leq
			\|(-\Delta_y)^{\frac \sigma 2}\nq_\omega\|_{2}\|\nq_\omega\|_{\infty}^{\frac p2}\|\nq_\omega\|_{p}^{\frac p2}\\
			&\lesssim
			\|\nq_\omega\|_{\infty}^{\frac p2}\|\nq_{\omega}\|_\x^{\frac p2}\|(-\Delta_y)^{\frac\sigma2}\nq_\omega\|_{2}\\&
			\lesssim
			\|\nq_{\omega}\|_\x^{\frac p2}\|(-\Delta_y)^{\frac \sigma 2}\nq_\omega\|_{2},
		\end{split}
		\]
		where in the above inequality we used Proposition \ref{regularity} and the uniform boundedness of $\nq_\omega$ in $\x$ and $L_{x,y}^\infty$ with respect to $\omega$  (see \eqref{bound-equ}). Therefore, we obtain that
		\[
		\omega^{-1}\|(-\Delta_y)^{\frac {3\sigma}4}\nq_\omega\|_{2}
		\lesssim \|(-\Delta_y)^{\frac \sigma4}\nl^{\frac \sigma 2}\nq_\omega\|_{2}
		+
		\omega^{-1}\|(-\Delta_y)^{\frac \sigma4}\nq_\omega\|_{2}\to0.
		\]
		This in turn implies that $\omega^{-1}\|(-\Delta_y)^{\frac {\sigma}2}\nq_\omega\|_{2}$ converges to zero as $\omega\to0$. We therefore obtain that $\|\nl^\frac \sigma 2\nq_\omega\|_{2}\to\|(\pax)^{\frac \sigma 2}\nq_\omega\|_{2}$ as $\omega\to0$. Multiply then \eqref{scaled-equation} by $\bar\nq_{\omega}$ and integrate on $\rr^d$ to obtain that for any $y\in\T^m$
		\[
		g_\omega(y):=\|(\pax)^{\frac \sigma 2}\nq_{\omega}(y)\|_{L^2(\rn)}^2-\|\nl^{\frac \sigma 2}\nq_{\omega}(y)\|_{L^2(\rn)}^2
		= \tilde\nb_\omega(\nq_\omega(\cdot,y)).
		\]
		The regularity of $\nq_\omega$ and the above considerations show that $g_\omega(y)\to0$, and consequently $\tilde\nb_\omega(\nq_\omega(\cdot,y))\to0$ for all $y\in\T^m$ as $\omega\to0$. Define
		\[
		\ell_\omega(y)=\left(\frac{\|\nq_\omega(\cdot, y)\|_{H^\sigma(\rn)}^2}{\|\nq_\omega(\cdot, y)\|_{L^{\al+2}(\rn)}^{\al+2}}\right)^{\frac{1}{\al}}.
		\]
		The facts $\tilde\nb_\omega(\nq_\omega(\cdot,y))\to0$ 	as $\omega\to0$ and the uniform boundedness of $\nq_\omega$ in $\x$ with respect to $\omega$ show that $\ell_\omega(y)\to1$ as $\omega\to0$. Moreover, $\tilde\nb_\omega(\ell_\omega(y)\nq_\omega(\cdot,y))=0$. The definition of $\bbnu_1$ then implies  $$\bbnu_1\leq\tilde\nii_1(\ell_\omega\nq_\omega(\cdot,y))=\ell_\omega^2\tilde\nii_1(\nq_{\omega}(\cdot,y))$$
		and consequently
		\[
		(2\pi)^m\bbnu_1\leq\liminf_{\omega\to0}\int_{\T^m}
		\tilde\nii_1(\ell_\omega\nq_\omega(\cdot,y))\dd y
		\leq \limsup_{\omega\to0}\niii_\omega(\nq_\omega)=\limsup_{\omega\to0}\bar{c}_\omega\leq(2\pi)^m\bbnu_1.
		\]

		\vspace{1mm}
		\subsubsection*{Step 2: Convergence of $\nq_\omega$ in $\x$ as $\omega\to0$.}
		
		Since $\nq_\omega$ is a minimizer of $\bar{c}_\omega$, it follows from {\color{black}Step 1} and the fact $\bar c_\omega\leq (2\pi)^m\bbnu_1$ that
		\[
		\|\nq_\omega\|_{2}+
		\|\nl ^\frac \sigma 2\nq_\omega\|_{2}\lesssim (2\pi)^m\bbnu_1.
		\]
		This implies that $\nq_\omega$ is uniformly bounded in $\x$. Similarly to the proof of Theorem \ref{general-ground-state}, one can find  $\{\omega_j\}_j\subset\rr^+$ converging to zero,  a function $  Q_0\in\x\setminus\{0\}$ and $\zeta=\zeta(\omega)\in\rn$ such that $\tau_\zeta\nq_\omega\rightharpoonup Q_0$   in $\x$. Combining the weakly lower semicontinuity of norms and the proof of Step 1, it follows that $(-\Delta_y)^{\frac\sigma2}  Q_0=0$  in $L_{x,y}^2$. Moreover, for all $y\in\T^m$, $Q_0$ satisfies \eqref{y-free-gs} with $\omega=1$ in the distributional sense. It follows from the uniqueness result of \eqref{y-free-gs} in \cite{Frank2016} that $Q_0\equiv  Q_1$  up to a translation. By redefining the translations, this implies that
		\begin{equation}\label{BL-inequ}
			\niii_\omega(\tau_\zeta\nq_\omega)-\niii_\omega(\tau_\zeta\nq_\omega-Q_1)-\niii_\omega(Q_1)\to0
		\end{equation}
		as $\omega\to0$. Since $\niii_\omega(\tau_\zeta\nq_\omega)=\bar c_\omega\to(2\pi)^m\bbnu_1$, we infer from \eqref{BL-est} that $\niii_\omega(\tau_\zeta\nq_\omega-Q_1)\to 0$ 	as $\omega\to0$, whence $\niii_\omega(Q_1)=(2\pi)^m\bbnu_1$ and $\tau_\zeta\nq_{\omega}$ converges strongly to $Q_1$ in $H^\sigma(\rn)$ as $\omega\to0$.

		In the rest of the proof, we drop $\tau_\zeta$ from $\tau_\zeta\nq_{\omega}$ for simplicity.

\subsubsection*{		Step 3: $\nl^{\frac \sigma 2} \nq_\omega = (-\Delta_x)^{\frac \sigma 2}\nq_\omega$ for any sufficiently small $\omega>0$.}

		Apply the operator $(-\Delta_y)^{\frac \sigma 2}$ on \eqref{scaled-equation} to get
		\[
		\nl^{\sigma}\tilde\nq_{\omega}
		+ \tilde\nq_{\omega}
		=(-\Delta_y)^{\frac \sigma 2}(|\nq_\omega|^{\al}\nq_\omega),
		\]
		where $\tilde\nq_\omega=(-\Delta_y)^{\frac \sigma 2}\nq_\omega\in\x$ (see Proposition \ref{regularity}). Multiplying the above equation by the conjugate of $\tilde \nq_\omega$ and integrating it on $\R^d\times\T^m$, we obtain
		\begin{equation}\label{estimate-1}
			\begin{split}
				\|\nl^{\frac \sigma 2}\tilde\nq_\omega\|_{2}&+\| \tilde\nq_\omega\|_{2}-
				\la  \tilde \nq_\omega  , |Q_1|^{\al}\tilde\nq_\omega \ra \\&\qquad
				=
				\scal{  \tilde \nq_\omega ,(-\Delta_y)^{\frac \sigma 2}(|\nq_\omega|^{\al}\nq_\omega)}-   \scal{   \tilde \nq_\omega ,  |Q_1|^{\al}\tilde\nq_\omega }.
			\end{split}
		\end{equation}
		To proceed, we first deduce that
		\begin{equation}\label{lhs-1}
			\begin{split}
				\text{LHS of } \eqref{estimate-1}&=
				\|\nl^{\frac \sigma 2}\tilde\nq_\omega\|_{2} +\| \tilde\nq_\omega\|_{2}-
				\la\tilde\nq_{\omega},|Q_1|^{\al}\tilde\nq_{\omega}   \ra
				\\  &
				\geq
				\int_\rn\sum_{k\in\Z^m}
				\left((|\xi|^2+\omega^{-\frac 1\sigma}|k|^2)^{\sigma}+1- \|Q_1\|_{L^\infty(\rn)}^{\al}\right)\abso{(\tilde\nq_{\omega})^\wedge}^2\dd\xi
				\\&
				\geq
				\int_\rn\sum_{k\in\Z^m}
				\left((|\xi|^2+\omega^{-\frac 1\sigma}|k|^2)^{\sigma}+1
				- \|Q_1\|_{L^\infty(\rn)}^{\al}\right)\abso{(\tilde\nq_{\omega})^\wedge}^2\dd\xi
				\\&
				\geq
				\|\tilde\nq_{\omega}\|_\x^2
			\end{split}
		\end{equation}
		for all sufficiently small $\omega$.
		On the other hand, the right-hand side of \eqref{estimate-1} can be re-written as
		\[
		\scal{\tilde \nq_\omega ,(-\Delta_y)^{\frac \sigma 2}\left((|\nq_\omega|^{\al} - |Q_1|^{\al})\nq_\omega\right)}
		=\scal{\tilde \nq_\omega , \left[(-\Delta_y)^{\frac \sigma 2},g\right]\nq_\omega+g\tilde \nq_\omega },
		\]
		where $g=|\nq_\omega|^{\al} - |Q_1|^{\al}$ and $[A,B]=AB-BA$. It follows from the Kato-Ponce commutator estimate (see \cite{Cardona2019,Li2019}) and the Cauchy-Schwarz inequality that
		\[
		\begin{split}
			\text{RHS of \eqref{estimate-1}} &\leq
			\|\tilde\nq_{\omega}\|_{2}  	\|[(-\Delta_y)^{\frac \sigma 2},g]\nq_\omega\|_{\lt}
			+\|g\|_{L_{x,y}^\infty}	\|\tilde\nq_{\omega}\|_{2}
			\\&
			\lesssim \|g\|_{L_{x,y}^\infty}	\|\tilde\nq_{\omega}\|_{2}
			\\&\lesssim
			\||\nq_{\omega}|^{\al} - |Q_1|^{\al} \|_{L_{x,y}^\infty}\|\tilde\nq_{\omega}\|_{2}
			\\  &
			\lesssim
			\| \nq_{\omega} - Q_1 \|_{L_{x,y}^\infty}^{\al}\|\tilde\nq_{\omega}\|_{2}.
		\end{split}
		\]
		The regularity of $Q_1$ (see \cite{Frank2016}) and Proposition \ref{regularity} together with Step 2 imply that
		\begin{equation}\label{rhs-1}
			\text{RHS of \eqref{estimate-1}}\leq\frac12\|\tilde\nq_{\omega}\|_\x^2
		\end{equation}
		for all sufficiently small $\omega$. Combining \eqref{lhs-1} and \eqref{rhs-1}, we conclude that $\|\tilde\nq_\omega\|_\x=0$ for all sufficiently small $\omega$. This implies in turn that $\|\nl^{\frac \sigma 2} \nq_\omega\|_{2}=\|(-\Delta_x)^{\frac \sigma 2}\nq_\omega\|_{2}$ if $\omega$ is small enough. The fact $\nl^{\sigma}=\nl^{\frac \sigma 2}\nl^{\frac \sigma 2}$ then completes the proof of Step 3.

\subsubsection*{Step 4: Conclusion}
	
		It follows from Step 3 that there exists $\omega_0>0$ such that for any $y\in\T^m$, the solution $\nq_\omega$ of \eqref{frac-c} satisfies \eqref{y-free-gs} for $\omega\in(0,\omega_0)$. We conclude from the regularity of $\nq_\omega$ and the uniqueness result in \cite{Frank2016} for \eqref{y-free-gs} that $\nq_\omega=Q_1$. By \eqref{scaled-func} this implies $u_\omega=Q_\omega$ and consequently $c_{\omega}=(2\pi)^m\bbnu_{\omega}$.
	\end{proof}

Next, we show the existence of nontrivial minimizers of $c_\omega$ for all sufficiently large $\omega$.
	
	\begin{lemma}[$y$-dependence of ground states with large $\omega$]\label{large-omega}
Let $u_\omega$ be an optimizer of $c_\omega$ (whose existence is deduced from Theorem \ref{general-ground-state}). Then there exists some $\omega_1\in[0,\infty)$ such that $c_{\omega}<(2\pi)^m\bbnu_{\omega}$ and $\nabla_y u_\omega\neq 0$ for all $\omega>\omega_1$. Moreover, if there exists $\tilde\omega_1>0$ such that $c_{\tilde\omega_1}<(2\pi)^m\bbnu_{\tilde\omega_1}$, then $c_{\omega}<(2\pi)^m\bbnu_{\omega}$ for all $\omega>\tilde\omega_1.$
\end{lemma}

	\begin{proof}

		By rescaling it is clear that  $c_\omega=\omega^{ \frac{\al+2}{\al}-\frac1{2\sigma }}\bar c_\omega$, $\na_\omega(\nq_\omega)=c_\omega$ and $\na(Q_1)=(2\pi)^m \bbnu_1$. Thus it suffices to show that $\bar{c}_\omega<(2\pi)^m \bbnu_1$ provided that $\omega$ is large enough. Let $g\not\equiv1$ be a positive function in $ C^\infty_0(\T^m)$ such that
		\[
		\int_{\T^m}g^{2a+2b\sigma-1}(y)\dd y=(2\pi)^m,
		\]
		where $a$ and $b$ are two positive constants that will be determined later.
		Define $v(x,y)=g^a(y)Q_1(g^b(y)x)$.	Then by using the Parseval  identity and the Minkowski  inequality, we obtain
		\[
		\begin{split}
			\nbb_\omega(v)
			&=
			\frac12\int_{\R^d\times\T^m}
			\left| \nl^\frac \sigma 2 v \right|^2\dd x\dd y
			+  \frac12\|v\|_{2}
			-\frac{1}{\al+2}\|v\|_{\alpha+2}^{\al+2}\\&
			\leq
			\int_\rn Q_1^2\dd x\int_{\T^m}g^{2a-b}(y)\dd y
			+
			\omega^{-1}\int_{\R^d\times\T^m}|(-\Delta_y)^{\frac \sigma 2}v|^2\dd x\dd y\\&\qquad
			+\int_\rn|(-\Delta_x)^{\frac \sigma 2}Q_1|^2\dd x\int_{\T^m}g^{2a+2b\sigma-1}(y)\dd y
			-\int_\rn|Q_1|^{\al+2}\dd x\int_{\T^m}g^{a(\al+2)-b}(y)\dd y.
		\end{split}
		\]
		Now suppose that $a=\frac{b(1+2\sigma )-1}{\al}$ with $b>\frac1{1+2\sigma }$. Using ${\mathcal{B}}_1(Q_1)=0$ we obtain
		\[
		\begin{split}
			\nbb_\omega(v)
			&\leq
			-\|Q_1\|_{L^2(\rn)}^2\left((2\pi)^m- \|g\|_{L_y^{2a-b}}^{2a-b}\right)
			+\omega^{-1}\| (-\Delta_y)^{\frac \sigma 2}v\|_{2}^2.
		\end{split}
		\]
		We note that
		\[
		\|g\|_{L_y^{2a-b}}^{2a-b}
		< (2\pi)^{\frac{(b(2\sigma +1)-1)m}{2bs+2a-1}  }
		\|g\|_{L_y^{2a+2b\sigma-1}}^{2a-b}=(2\pi)^m,
		\]
		which in turn implies $\nbb_\omega(v)<0$. By the definition of $\bar{c}_\omega$ it follows that
		\[
		\begin{split}
			\bar{c}_\omega\leq\niii_\omega(v)&=
			\frac{\al}{2(\al+2)}\int_{\R^d\times\T^m}\left(|v|^2+ \left| \nl^\frac \sigma 2 v \right|^2\right)\dd x\dd y\\
			&\leq \frac{\al}{2(\al+2)}
			\left(
			\|g\|_{L_y^{2a-b}}^{2a-b}
			\|Q_1\|_{L^2(\rn)}^2+
			\omega^{-1}\| (-\Delta_y)^{\frac \sigma 2}v\|_{2}^2
			+(2\pi)^m\| (-\Delta_x)^{\frac \sigma 2}Q_1\|_{2}^2
			\right)
			\\&=(2\pi)^m \bbnu_1+\frac{\al}{2(\al+2)}\left(-\|Q_1\|_{L^2(\rn)}^2\left((2\pi)^m- \|g\|_{L_y^{2a-b}}^{2a-b}\right)
			+\omega^{-1}\| (-\Delta_y)^{\frac \sigma 2}v\|_{2}^2\right)\\&<(2\pi)^m
			\bbnu_1,
		\end{split}
		\]
		whence $\bar c_\omega<(2\pi)^m \bbnu_1$ if $\omega$ is large enough. By rescaling we immediately infer that $c_{\omega}<(2\pi)^m\bbnu_{\omega}$ and consequently $\nabla_y u_\omega\neq 0$ for all $\omega>\omega_1$ with some $\omega_1\in[0,\infty)$.		
Now let $\tilde\omega_1>0$ be  such that $c_{\tilde\omega_1}<(2\pi)^m \bbnu_{\tilde\omega_1}$. Then there exists a ground state $u_{\tilde\omega_1}$ of \eqref{frac-c} such that $u_{\tilde\omega_1}\neq Q_{\tilde\omega_1}$. Setting $U_\omega(x,y)=\sigma^{\frac{1}{p-1}}Q_{\tilde\omega_1}\left(\sigma^{\frac1{2s}}x,y\right)$ with $\sigma=\omega/\tilde\omega_1>1$, we obtain from $\nb_\omega(Q_{\omega_1})=0$ that
		\[
		0=\nb_{\omega_1}(Q_{\omega_1})=\nb_{\omega}(U_\omega)+
		\int_{\rn}\sum_{k\in\Z^m}\left((|\xi|^2+\sigma^{\frac1s}|m|^2)^{s}-(|\xi|^2+|m|^2)^s\right)|\hat{U}_\omega|^2\dd \xi.
		\]
		Hence
		\[
		\nb_{\omega}(U_\omega)=
		\int_{\rn}\sum_{k\in\Z^m}\left((|\xi|^2+|m|^2)^s-(|\xi|^2+\sigma^{\frac1s}|m|^2)^{s}\right)|\hat{U}_\omega|^2\dd \xi\leq0
		\]
		and consequently
		\[
		c_\omega\leq
		\nii_\omega(U_\omega)
		\leq
		\sigma^{\frac{p+1}{p-1}-\frac{n}{2s}}\nii_{\tilde\omega_1}(Q_{\tilde\omega_1})
		<
		(2\pi)^m\sigma^{\frac{p+1}{p-1}-\frac{n}{2s}} \bbnu_{\tilde\omega_1}=(2\pi)^m \bbnu_{ \omega}.
		\]
\end{proof}

	Having all the preliminaries, we are in a position to classify the periodic dependence of the ground states of \eqref{frac-c} in term of the frequency $\omega$.

	\begin{theorem}[$y$-dependence of the ground states]\label{thm y depend}
Let $u_\omega$ be an optimizer of $c_\omega$ (whose existence is deduced from Theorem \ref{general-ground-state}). Then there exists $\omega^*\in(0,\infty)$ such that
\begin{itemize}
\item For all $\omega\in(0,\omega^*]$ we have $c_{\omega}=(2\pi)^m\bbnu_{\omega}$. Moreover, for all $\omega\in(0,\omega^*)$ we have $\nabla_y u_\omega=0$.

\item For all $\omega\in(\omega^*,\infty)$ we have $c_{\omega}<(2\pi)^m\bbnu_{\omega}$. Moreover, for all $\omega\in(\omega^*,\infty)$ we have $\nabla_y u_\omega\neq 0$.
\end{itemize}
	\end{theorem}

	\begin{proof}
Define
$$\omega^\ast:=\sup\{\omega>0:\;c_k=(2\pi)^m\bbnu_{k}  \;\text{for}\;k\in(0,\omega)\}.$$
It follows from Lemma \ref{small-omega} and \ref{large-omega} that $0<\omega^\ast<\infty$. Moreover, up to the point $\omega^*$ the claims about the relation between $c_\omega$ and $(2\pi)^m\bbnu_{\omega}$ also already follow from Lemma \ref{small-omega} and \ref{large-omega}. That $c_{\omega^*}=(2\pi)^m\bbnu_{\omega^*}$ holds follows then from the continuity of the mappings $\omega\mapsto c_\omega$ and also $\omega\mapsto \bbnu_\omega$, where the former is deduced from Theorem \ref{general-ground-state} and the latter can be proved in a similar way. Notice that $c_{\omega}<(2\pi)^m\bbnu_{\omega}$ necessarily implies $\nabla_y u_\omega\neq 0$. It thus remains to prove $\nabla_y u_\omega=0$ for all $\omega\in(0,\omega^*)$. We borrow an idea from \cite{GrossPitaevskiR1T1} to prove the claim. Define
$$\beta^\ast:=\sup\{\beta>0:\;\bar{c}_k=(2\pi)^m\bbnu_{1}  \;\text{for}\;k\in(0,\beta)\}.$$
By rescaling, showing the original claim is equivalent to showing that $\nabla_y \nq_\beta=0$ for all $\beta\in(0,\beta^*)$. Assume the contrary that there exists some $\lambda\in(0,\beta^*)$ such that there exists an optimizer
$$\nq_\lambda=\nq_\lambda(x,y)=\sum_{k\in\Z^d}\nq_{\lambda,k}(x)e^{iky}$$
of $\bar{c}_\lambda$ such that $\nq_{\lambda,k}\neq 0$ for some $k\in\Z^d\setminus\{0\}$. Let now $\mu\in(\lambda,\beta^*)$. By Fourier expanding the kinetic energy we infer that $\|\nl^{\frac{\sigma}{2}}\nq_\lambda\|_2^2-
\|\nl^{\frac{\sigma}{2}}\nq_\mu\|_2^2>0$. Hence by the definition of $\bar{c}_{\mu}$ we infer that $\na_\mu(\nq_\lambda)\geq \bar{c}_\mu$. This in turn implies
\begin{align*}
(2\pi)^m\bbnu_{1}&=\na_\lambda(\nq_\lambda)=\na_\mu(\nq_\lambda)+\frac{1}{2}(\|\nl^{\frac{\sigma}{2}}\nq_\lambda\|_2^2-
\|\nl^{\frac{\sigma}{2}}\nq_\mu\|_2^2)\nonumber\\
&\geq \bar{c}_{\mu}+\frac{1}{2}(\|\nl^{\frac{\sigma}{2}}\nq_\lambda\|_2^2-
\|\nl^{\frac{\sigma}{2}}\nq_\mu\|_2^2)\nonumber\\
&=(2\pi)^m\bbnu_{1}+\frac{1}{2}(\|\nl^{\frac{\sigma}{2}}\nq_\lambda\|_2^2-
\|\nl^{\frac{\sigma}{2}}\nq_\mu\|_2^2)>(2\pi)^m\bbnu_{1},
\end{align*}
which yields a contradiction and thus completes the desired proof.
\end{proof}

	\subsection{$\omega$-limit of the ground state solutions}
	The goal of this section is to investigate the behavior of the standing wave solutions of \eqref{frac-c} as $\omega$ tends to infinity, in which case we expect that the standing wave solutions on $\R^{d}\times\T^m$ shall resemble the ones on $\R^{d+m}$. For this reason, we need to amend the used notations appropriately to trace out the effects of $\omega$ on the standing waves.

	By scaling, we may transform \eqref{frac-c} into
	\begin{equation}\label{frac-c-scaled}
		\des u+  u=|u|^{\al}u,\qquad(x,y)\in \Omega=\rn\times\T^m_\omega,
	\end{equation}
	where $\T^m_\omega=(\omega^{\frac{1}{2\sigma }}\T)^m$. Associated with \eqref{frac-c-scaled}, we define for any $L>0$ the $L$-Sobolev space
	as
	\[
	H^\sigma_L=\sett{u(x,y)=\sum_{k\in\Z^m}u_k(x)\ee^{\ii  \frac\pi Lk\cdot y},\;  \|u\|_{H^\sigma_L}<\infty},
	\]
	where
	\[
	\|u\|_{H^\sigma_L}^2=(2L)^m\int_{\rr^d}\sum_{k\in\Z^m}
	\left(1+\frac\pi L|k| +2\pi|\xi|\right)^{2\sigma }|\hat{u}_k(\xi)|^2\dd\xi,
	\]
	\[
	u_k(x)=\frac{1}{(2L)^m}\int_{(-L,L)^m}u(x,y)\ee^{-\ii    \frac\pi Lk\cdot y}\dd y
	\]
	and the symbol $\wedge$ stands for the Fourier transform in $x$. Notice that
	\[
	\|u\|_{H^\sigma_L}^2\cong\int_{\Omega_L}|u(x,y)|^2\dd x\dd y
	+
	\sum_{k\in\Z^m}\int_{\Omega_L}
	\frac{|u(x,y)-u(\tau,z)|^2}{(|x-\tau|^2+|y-z- \frac{\pi}{L} k|^2)^{\frac{d+m+2\sigma }{2}}}\dd\tau\dd z,
	\]
	where $\Omega_L=\rr^d\times (-L,L)^m$.

In order to describe the $\omega$-behavior of ground states of \eqref{frac-c} as $\omega\to 0$, we need to make a connection between the Sobolev space $H^\sigma(\rr^{d+m})$ and the $L$-Sobolev space $H_L^{\sigma}$. In the rest of this section, we assume that $L=\omega^{\frac{1}{2\sigma}}$ (see \eqref{frac-c-scaled}. As we aim to push $\omega\to \infty$, we may also, without loss of generality, assume that $L\geq 1$.
	\begin{lemma}[Extension operator]\label{bounded-oper}
There exists an extension operator $\Upsilon_L:H_L^{\sigma}\to H^\sigma(\rr^{d+m})$, satisfying $\Upsilon_L(u)|_{\R^d\times\T^m}=u$ for all $u\in H_L^{\sigma}$, whose operator norm is uniformly bounded in $L$.
\end{lemma}

	\begin{proof}
In view of interpolation it suffices to show the existence of uniformly bounded in $L$ extension operators $\Upsilon_L:H^\sigma_L\to H^\sigma(\R^{d+m})$ for $\sigma=\{0,1\}$. Let $\chi\in C^\infty_c(\R^{m};[0,1])$ such that $\chi(y)\equiv 1$ for $y\in  [-L,L]^m$, $\chi_L(y)\equiv 0$ for $y\in  \R^m\setminus([-L-1,L+1]^m)$ and $\sup_{L\geq 1}\sup_{y\in\R^m}|\nabla_y\chi_L(y)|<\infty$. Define $\Upsilon_L (u)$ then by $\Upsilon_L(u)(x,y):=\chi(y)u(x,y)$. It is clear that $\Upsilon_L $ defines a uniformly bounded in $L$ extension operator for $\sigma=0$. On the other hand, using product rule we obtain
\[\|\Upsilon_L(u)\|_{H^1(\R^{d+m})}\lesssim \|\nabla_y \chi_L\|_{L^\infty(\R^m)}\|u\|_{L^2_L}+
\|\chi_L\|_{L^\infty(\R^m)}\|\nabla u\|_{L^2_L}\lesssim \|u\|_{H_L^1},
\]
as desired.
	\end{proof}
	
	\begin{corollary}\label{cons-embed}
		
		The embedding constant of $H^\sigma_L\hookrightarrow L^{\al+2}(\rr^d\times \T^m_L)$ is uniformly bounded in $L$.
	
	\end{corollary}
	
	Associated with \eqref{frac-c-scaled}, we define the functionals
	\begin{gather*}
	\A_L(u)=
	\frac12 \|u\|_{H_L^{\sigma}}^2    -\frac{1}{\al+2}\|u\|_{L^{\al+2}(\Omega_L)}^{\al+2},\quad
\nb_L(u)=\la\A_L'(u),u\ra
	\end{gather*}
	and the minimization problem
	\begin{equation}\label{L-nehari}
		c_L=\inf\sett{\A_L(u):\;u\in H^\sigma_L\setminus\{0\},\;\nb_L(u)=0 }.
	\end{equation}
Arguing as in the proof of Theorem \ref{general-ground-state}, for each $L\geq 1$ the variational problem $c_L$ has a minimizer $u_L\in H^\sigma_L$. We then establish the following convergence of the minimizers $u_L$ as $L\to \infty$. See also \cite{Esfahani2018}.
\begin{theorem}[Convergence of $u_L$ as $L\to\infty$]\label{limit-g-theo}
	There exists a sequence $\{\zeta_L\}\subset \rr^d\times\rr^m$ and a (not relabeled) subsequence $\{u_L\}$ such that
	$$\lim_{L\in\N,\,L\to\infty}\Upsilon_L u_L(\cdot+\zeta_L)=u$$
	weakly in $H^\sigma(\R^d\times\R^m)$, where $u\in H^\sigma(\rr^d\times\rr^m)\setminus\{0\}$ is the unique solution (up to translation, see Remark \ref{remark-semitri}) of
	\begin{equation}\label{whole-space}
		\des u+u=|u|^{\al}u,\qquad(x,y)\in \rn\times\rr^m.
	\end{equation}
\end{theorem}
\begin{proof}
	The proof is split  into six steps.
	\begin{enumerate}[(i)]
		\item \label{itemi} Let $u_k\in H^\sigma_L$ be a bounded sequence such that
		\begin{equation}\label{vanish-i}
			\lim_{k\to\infty}\sup_{\zeta\in\Omega_L}\int_{K_r(\zeta)}|u_k|^2\dd x\dd y=0
		\end{equation}
		for some $r>0$, where $K_r(\zeta)$ is the cube with the side length $r$ and centered at $\zeta$.	Then $\|u_k\|_{L^q(\Omega_L)}\to0$ as $k\to\infty$ for $q\in(2,2^\ast_\sigma)$.

		The proof of (\ref{itemi})  is a modification of one of  Lemma 2.4 in \cite{Secchi2013}.

		\item \label{itemii}
		For any $L>0$, there exists a minimizer $u_L$ of \eqref{L-nehari} which is a critical point of $\A_L$. Moreover, $\A_L(u_L)=c_L$ is uniformly bounded in $L$.

		To prove (\ref{itemii}), we first define the mountain-pass level
		\[
		\fm_L=\inf_{\gamma\in\Gamma_L}\max_{t\in[0,1]}\A_L(\gamma(t)),
		\]
		where $\Gamma_L=\sett{\gamma\in C^1([0,1],H_L^{\sigma}),\;\gamma(0)=1,\A_L(\gamma(1))<0}$.
		Note that from Corollary \ref{cons-embed} we have for any $u\in H^\sigma_L$
		\[
		\A_L(u)\geq \frac12\norm{u}_{H^\sigma_L}^2-C\norm{u}^{\al+2}_{H_L^{\sigma}},
		\]
		where $C>0$ is independent of $L$. This implies that $\A_L(u)\geq C>0$ if $u$ is small enough in $H^\sigma_L$. On the other hand, for a fixed  $v\in H_1^{\sigma}$, we choose $t_0>0$ such that $\A_1(t_0v)<0$. Let $v_1=t_0v \in H^\sigma_1$ and define
		\[
		v_L=\begin{cases}
			v_1,&(x,y)\in\rr^d\times (-1,1)^m,\\
			0,&(x,y) \in\rr^d\times (-L,L)^m\setminus  (-1,1)^m.\\
		\end{cases}
		\]
		By extending periodically to $\rr^d\times \T^m_L$, we observe that $v_L\in H^\sigma_L$ and $\A_L(t_0v_L)<0$. Moreover,
		\begin{equation}\label{unif-bound}
			0<C_0\leq \fm_L=\max_{t\in[0,1]}\A_L(t v_L)=
			\max_{t\in[0,1]}\A_L(t v_1)\leq C_1.
		\end{equation}
		Hence, $\fm_L$ is uniformly bounded in $L$. The mountain-pass lemma (see e.g. \cite{willem}) shows that there is a sequence $\{u_{L,k}\}_k\subset H^\sigma_L$ such that $\A'_L(u_{L,k})\to0$ and $\A_L(u_{L,k})=\fm_L$ as $k\to\infty$. Furthermore,
		\[
		\norm{u_{L,k}}_{H_L^{\sigma}}^2\lesssim
		\A_L(u_{L,k})-\frac{1}{\al+2}\nb_L(u_{L,k})\leq \fm_L+\norm{u_{L,k}}_{H_L^{\sigma}}+o(1).
		\]
		Therefore $\{u_{L,k}\}_k$ is bounded in $H_L^{\sigma}$ which also possesses a weakly convergent subsequence. To rule out the vanishing case, if there is $r>0$ such that
		\[
		\lim_{k\to\infty}\sup_{\zeta\in \Omega_L}\norm{u_{L,k}}_{L^2(K_r(\zeta))}=0,
		\]
		then $\norm{u_{L,k}}_{L^q(\Omega_L)}$ converges to zero as $k\to\infty$ due to \eqref{itemi}.
		Hence
		\[
		\fm_L=\A_L(u_{k,L})-\frac12\nb_L(u_{L,k})+o(1)\leq
		\|u_{L,k}\|_{L^2(\Omega_L)}^2+C \|u_{L,k}\|_{L^{\al+2}(\Omega_L)}^{\al+2}+o(1).
		\]
		This however contradicts \eqref{unif-bound}. Thus there exists $\{\zeta_k\}\subset\Omega_L$ such that, up to a subsequence,
		$\|u_{L,k}\|_{L^2(K_1(\zeta_k))}\geq C>0$ for all $k$. This implies that $u_{L,k}(\cdot+\zeta_k)$ converges to some $u_L\in H^\sigma_L$ weakly in $H^\sigma_L$ and strongly in $L^q_{\rm loc}(\Omega_L)$. In addition $u_L\not\equiv0$ and $\nb_L(u)=0$.
		
		\item\label{item2-3} $c_L=\fm_L$.
		
		The proof of (\ref{item2-3}) follows by the arguments put forward in \cite[Lemma 2.18]{Esfahani2018}. So we omit the details.
		
		\item\label{item3} Let $\{u_L\}_L$ be a sequence of functions of $H_L^{\sigma}$ such that $\|u_L\|_{H^\sigma_L}\leq C$. If there exists $r>0$ such that
		\begin{equation}\label{vanish-i-0}
			\lim_{L\to\infty}\sup_{\zeta\in\rr^d\times\rr^m}\int_{K_r(\zeta)}|u_L|^2\dd x\dd y=0,
		\end{equation}
		then $u_L\to0$ in $L^q(\Omega_L)$ for any $q\in(2,2_\sigma^\ast)$.
		
		To prove (\ref{item3}),	 we assume without loss of generality that $r\in(0,1)$.
		It suffices to show that
		\[
		\sup_\zeta\norm{\Upsilon_L u_L}_{L^2(K_r(\zeta))}^2\leq C\sup_{\zeta}
		\|u_L\|_{L^2(K_r(\zeta))}^2.
		\]
		Indeed, by using \cite[Lemma 2.4]{Secchi2013} we obtain that $\Upsilon_L u_L\to0$ in $L^q(\rr^d\times\rr^m)$, whence the assertion by combining with $u_L=\Upsilon_Lu_L$ on $\Omega_L$.				
		By periodicity, we may also assume that either $K_r(\zeta)\subset\Omega_L $, or $K_r(\zeta)\subset\Omega_{L+1}\setminus\Omega_L $. Since the first case is obvious, we will only consider here the second case. For $\zeta=(\zeta_1,\zeta_2)\in\rr^d\times\rr^m$ let $\zeta_2\in [-L-1,-L]^m$ and $[\zeta_2-r/2,\zeta_2+r/2]^m\subset [-L-1,-L]^m$. Then
		\[
		\norm{\Upsilon_L u_L}_{L^2(K_r(\zeta))}^2\leq
		\norm{  u_L}_{L^2(K_r(\zeta))}^2+
		\norm{  u_L}_{L^2(K_r(\zeta_1,\zeta_2+2L))}^2
		\leq 2	\norm{  u_L}_{L^2(K_r(\zeta))}^2
		\]
		and $K_r(\zeta_1,\zeta_2+2L)\subset\Omega_L$.

		\item \label{itemiv}
		Let $u_L\subset H^\sigma_L$ be a sequence of nontrivial solutions of \eqref{frac-c-scaled} such that $\A_L(u_L)$ is bounded from above. Then there exist  ${\zeta_L}\subset\rr^d\times\rr^m$ and $u\in H^\sigma(\rr^d\times\rr^m)$ such that $\Upsilon_Lu_L(\cdot+\zeta_L)$ converges (up to a subsequence) weakly  to $u$, and $u$ is a nontrivial solution of \eqref{whole-space}.

		To prove Step (\ref{itemiv}), we observe simply from the boundedness of $\A_L(u_L)$ that $\|u_L\|_{H^\sigma_L}\leq C_0$. On the other hand, we obtain from Corollary \ref{cons-embed} that $\|u_L\|_{H^\sigma_L}\leq C_1$ for some positive constants $C_0,C_1$, independent of $L$. The part \eqref{item3} shows that there is $\{\zeta_L\}\subset\rr^d\times\rr^m$ such that $v_L(\cdot)=u_L(\cdot+\zeta_L)$ satisfies
		\[
		\int_{K_r(0)}|v_L|^2\dd x\dd y\geq C_2
		\]
		for some $r,c_2>0$. Moreover, $v_L$ is a critical point of $\A_L$. The boundedness of $\A_L(v_L)$ in $H^\sigma(\rr^d\times\rr^m)$ shows that there exists $u\in H^\sigma(\rr^d\times\rr^m)$ such that $v_L\rightharpoonup u$, up to subsequence, in $H^\sigma(\rr^d\times\rr^m)$. The local compactness of $H^\sigma(\rr^d\times\rr^m)$ into $L^q(\rr^d\times\rr^m)$ simply reveals that $u$ is a nontrivial weak solution of \eqref{whole-space}.

		\item The completion of proof of the theorem.
		
		From item \eqref{itemiv} it follows the existence of a function $u$, where we set the translations $\zeta_L$ according to the sequence $\{v_L=u_L(\cdot+\zeta_L)\}$.
		We next prove that $u$ is a ground state of \eqref{whole-space}. It suffices to show for any $v\in \fs$   and   $\epsilon>0$ there exists $L_\epsilon$ and $v_L\in \fs_L$ such that
		\begin{equation}\label{est-gs}
			\A_L(v_L)\leq\A(v)+\epsilon
		\end{equation} for all $L\geq L_\epsilon$, where
		\[
		\fs_L=\sett{v\in H^\sigma_L\setminus\{0\},\;\nb_L(v)=\scal{\A_L'(v),v}=0},
		\]
		\[
		\fs=\sett{v\in H^\sigma(\rr^d\times\rr^m)\setminus\{0\},\;\nb(v)=\scal{\A'(v),v}=0}.
		\]
		More precisely, this implies for any $v\in\fs$  and $\epsilon>0$ that
		\[
		\limsup_Lc_L\leq\A(v)+\epsilon,
		\]
		whence $\limsup m_L\leq \fm$. The definition of $\fs_L$ shows for any   $L\gg1$ that $m_L\geq\|v_L\|_{L_{\rm loc}^{\al+2}(\rr^d\times\rr^m)}$. This implies $\liminf m_L\geq\A(u)\geq\fm$ and consequently $ m_L\to\fm=\A(u)$.
		
		To prove \eqref{est-gs}, using the continuity of $\A$ we can find for $v\in\fs$ a sequence $w_L\in H^\sigma(\rr^d\times\rr^m)$ such that $w_L\to v$ in $H^\sigma(\rr^d\times\rr^m)$, ${\rm supp} (w_L)\subset\Omega_L$, $\A_L(w_L)\to\A(v)$ and
		\[
		\nb(w_L)\to\nb(v)=0.
		\]
		Moreover, we can find $\tau_L>0$ such that $\nb(\tau_Lw_L)=0$. Since $v\in\fs$, we have $\tau_L\to1$ and $\A(\tau_Lw_L)\to\A(v)$ as $L\to\infty$. Hence, for a sufficiently large $L_\epsilon$ we have $v_L=\tau_Lw_L\in\fs_L$ such that \eqref{est-gs} holds.
		
	\end{enumerate}
\end{proof}

We give a connection between  periodic solutions  of \eqref{frac-c} and solutions of \eqref{whole-space}.
	
\begin{proposition}\label{sum-gs}
	Let $\{u_L\}\subset H^\sigma_L$ be a sequence such that $\A_L'(u_L)\to0$ and $c_L=\A(u_L)\to \fm>0$ as
	$L\to \infty$. Assume that $\norm{u_L}_{H^\sigma_L}$ is uniformly bounded. Then, there exist the critical points $\ff_j\in H^\sigma(\rr^d\times\rr^m)$, $j=1,\cdots,l$,  of
	\[
	\A(\ff)=\frac12\norm{\ff}_{H^\sigma(\rr^d\times\rr^m)}^2-\frac1{\al+2}\|\ff\|_{L^{\al+2}(\rr^d\times\rr^m)}^{\al+2}
	\]
	and $\{\zeta_L\},\{\zeta_{j,L}\}\subset \rr^d\times\rr^m$ such that
	\[
	\norm{u_L(\cdot+\zeta_L)-\sum_{j=1}^l\ff_{j}(\cdot+\zeta_{j,L})}_{H^\sigma_L}\to0\qquad\text{as}\,\;L\to\infty
	\]
	and
	\begin{equation}\label{sum-energ}
		\sum_{j=1}^l\A(\ff_j) 
		=
		\fm.
	\end{equation}
	
\end{proposition}

\begin{proof}
	By Theorem \ref{limit-g-theo}, there exists $\{\zeta_L\}\subset\rr^d\times\rr^m$ such that $\Upsilon_L u_L(\cdot+\zeta_L)$ converges, up to a subsequence, weakly in $H^\sigma(\rr^d\times\rr^m)$ and strongly in $L^q_{\rm loc}(\rr^d\times\rr^m)$, $2<q<2_\sigma^\ast$, to a nontrivial solution $\ff_1$ of \eqref{whole-space}.   
	
	By density, we can choose $\{w_L\}\subset C_0^\infty(\rr^d\times\rr^m)$ such that $w_L\to \ff_1$ strongly in $H^\sigma(\rr^d\times\rr^m)$ and ${\rm supp}( w_L)\subset\Omega_L$. We can also extend $w_L$ to $ H^\sigma_L$. Define $v_L=u_L(\cdot+\zeta_L)$ and $z_L=v_L-w_L\in H^\sigma_L$.
	
	Given $\epsilon>0$, there is $r_\epsilon>0$ and $L_\epsilon$ such that
	\begin{equation}\label{smallness}
		\|w_L\|_{H^\sigma(\rr^d\times\rr^m\setminus \overline{B_{r_\epsilon}(0)})}\leq \epsilon
	\end{equation}
	for all $L\geq L_\epsilon$. Notice that $\rr^d\times\rr^m\setminus B_{r}(0)\subset \rr^d\times\rr^m$ is a regular domain, so it is a $H^\sigma(\rr^d\times\rr^m\setminus\overline{B_{r_\epsilon}(0)})$-extension domain (see \cite{Zhou2015}). Then, it follows from \cite[Theorem 5.4]{DiNezza2012} that $H^\sigma(\rr^d\times\rr^m\setminus\overline{B_{r_\epsilon}(0)})$ is continuously embedded in $H^\sigma(\rr^d\times\rr^m)$, and consequently
	\begin{equation}\label{embd-cyl}
		\|w_L\|_{L^{\al+2}(\rr^d\times\rr^m\setminus \overline{B_{r_\epsilon}(0)})}\leq \|w_L\|_{H^\sigma(\rr^d\times\rr^m\setminus \overline{B_{r_\epsilon}(0)})}\leq\epsilon
	\end{equation}
	by applying Theorem 6.7 in \cite{DiNezza2012} and \eqref{smallness}.
	
	Now, we show that $\A_L'(z_L)\to0$ and $\A_L(z_L)\to\fm-\A(\ff_1)$ as $L\to+\infty$.
	
	Since $\A$ is $C^1$, we have $\A '(w_L)\to0$ and $\A_L(w_L)=\A(w_L)\to\A(\ff_1)$. Then,
	\[
	\scal{\A_L'(w_L),\psi_L}=
	\scal{\A '(w_L),\Upsilon_L\psi_L}
	\leq \epsilon_L  \|\Upsilon_L\psi_L\|_{H^\sigma_L}
	\leq \epsilon_L C\|\psi_L\|_{H^\sigma_L}
	\]
	for any $\psi_L\in H^\sigma_L$, where $\epsilon_L\to0$ and $C>0$ is independent of $L$. This means that $\A_L'(w_L)\to0$ as $L\to\infty$. 
	
	Hence, we have
	\begin{equation}\label{diff-conv}
		\scal{\A'_L(z_L),\psi_L}=\scal{\A_L'(v_L),\psi_L}-\scal{\A_L'(w_L),\psi_L}
		+\int_{\Omega_L}\left(f(v_L)-f(w_L)-f(v_L-w_L)\right)\psi_L\dd x\dd y,
	\end{equation}
	where $f(t)=|t|^{\al}t$. By choosing an appropriate $r_\epsilon$ and $L\gg1$, and splitting $\Omega_L= B_{r_\epsilon}(0)  \cup \Omega_L\setminus \overline{B_{r_\epsilon}(0)}$, we can deduce from \eqref{embd-cyl} and the convergence $v_L$  to $u$ in $L^q_{\rm loc}(\rr^d\times\rr^m)$ combining with the H\"{o}lder inequality to obtain the last term of the right hand side of \eqref{diff-conv} converges to zero as $L\to\infty$. Therefore,  $\A_L'(z_L)\to0$ as $L\to\infty$. 
	
	On the other hand, since
	\[
	\A_L(v_L)=\A_L(z_L)+
	\A_L(w_L)+\scal{z_L,w_L}_{H^\sigma_L}
	-\int_{\Omega_L}\left(F(z_L+w_L)-F(z_L)-F(w_L)\right)\dd x\dd y,
	\]
	where $F$ is the primitive function of $f$, we have, as above, for $L$ large enough that
	\[
	\abso{\scal{z_L,w_L}_{H^\sigma_L}} 
	\leq
	\abso{\scal{\Upsilon_L z_L,w_L}} +C\norm{z_L}_{H^\sigma_L}
	\|w_L\|_{H^\sigma(\rr^d\times\rr^m\setminus \overline{B_{r_\epsilon}(0)})}.
	\]
	As $\Upsilon_Lz_L\rightharpoonup0$  and $w_L\to \ff_1$ in $H^\sigma(\rr^d\times\rr^m)$, $ \abso{\scal{z_L,w_L}_{H^\sigma_L}} $ tends to zero from \eqref{smallness}. Now, by using the mean value theorem and the H\"{o}lder inequality we get
	\[
	\begin{split}
		\left|\int_{\rr^d\times\rr^m\setminus\overline{B_{r_\epsilon}(0)}}F(z_L+w_L)-F(z_L)\dd x\dd y\right|
		&\lesssim
		\int_{\rr^d\times\rr^m\setminus\overline{B_{r_\epsilon}(0)}}(|z_L|+|w_L|)^{\al}|w_L|\dd x\dd y\\
		&
		\lesssim   \left(\|z_L\|_{H^\sigma_L}+\|w_L\|_{H^\sigma_L}\right) ^\al \|w_L\|_{H^\sigma_L(\rr^d\times\rr^m\setminus\overline{B_{r_\epsilon}(0)})}\\&\leq\epsilon.
	\end{split}
	\]
	Therefore it follows from  \eqref{embd-cyl}   that
	\[
	\fm=\lim_{L}\A_L(v_L)=\lim_L\A_L(z_L)+\A(\ff_1).
	\]
	
	Now, if $z_L \to0$ in $H^\sigma_L$, then $\A_L(z_L)\to0$. If $\{z_L\}_L\subset H^\sigma_L$ is far from the origin uniformly, then it follows from $\A_L'(z_L)\to0$ that $\A_L(z_L)$ is also far from zero uniformly. We have in both cases that $\lim_L\A_L(z_L)=\fm-\A(\ff_1)$ and $\A(\ff_1)\leq\fm$. 
	
  Now assume that  $\A(\ff_1)=\fm$.   If $ z_L\not\to0$ in $H_L^\sigma$ as $L\to\infty$, then we can repeat the above argument for $z_L$ instead of $u_L$ to find a nontrivial solution $\tilde\ff_1$ of \eqref{whole-space} such that $\A(\tilde \ff_1)\leq0$; which is a contradiction to the fact that $\A$ is positive on the critical points of $\A$.  Hence, $v_L\to0$ in $H_L^\sigma$ in this case and we take $\zeta_{1,L}=\vec0$.

	Next we assume that $\A(\ff_1)<\fm$. Since $\A(\ff_1)\geq C_1>0$, we can repeat  the above argument for $z_L$ instead of $u_L$ and $\fm-\A(\ff_1)\leq\fm-C_1$ instead of $\fm$, and find $\{\zeta_{2,L}\}\subset\rr^d\times\rr^m$ and a solution $\ff_2$ of \eqref{whole-space} such that $\Upsilon_Lz_L (\cdot+\zeta_{2,L})\rightharpoonup\ff_2$ in $H^\sigma_L$.
	
	Repeating finitely many steps leads us to the proof.
\end{proof}

The weak convergence result of Theorem \ref{limit-g-theo} can be improved to the strong convergence from the proof of Theorem \ref{sum-gs}  $c_L=\fm_L$ and the fact that the ground state of \eqref{whole-space} satisfies
\[
\begin{split}
	\A_L(u_L)\to\A(u)=c&=\inf\sett{\A(v),\;v\in H^\sigma(\rr^d\times\rr^m)\setminus\{0\},\;\scal{\A'(v),v}=0}\\
	&=
	\inf_{\gamma\in\Gamma }\max_{t\in[0,1]}\A (\gamma(t)),
\end{split}
\] 
where $\Gamma =\sett{\gamma\in C^1([0,1],H ^{\sigma}(\rr^d\times\rr^m)),\;\gamma(0)=1,\A(\gamma(1))<0}$.

\begin{corollary}\label{strong-conv-coroll}
	Let $ u_L \in H_L^\sigma$ be a minimizer of $c_L$. Then there exists a nontrivial ground state $u\in H^\sigma(\rr^d\times\rr^m)$ of \eqref{whole-space} and a subsequence $\{\zeta_L\}\subset \rr^d\times\rr^m$ such that along a subsequence
	\[
	\lim_{L\to\infty}\norm{u_L(\cdot+\zeta_L)-u}_{H^\sigma_L}=0.
	\]
\end{corollary}

	\section{Isotropic case: Normalized ground states}\label{sec 2.4}

	In this section, we study the existence of normalized solutions of \eqref{frac-c}. Indeed, we aim to find the minimizers of
	\begin{align}\label{mc def}
	m_c:=
\left\{
\begin{array}{ll}
\inf_{u\in S(c)}\{\E(u): u\in S(c)\},&\alpha\in(0,\min\{  \frac{4\sigma }{d},2_\sigma^\ast \}),\\
\\
\inf_{u\in S(c)}\{\E(u): u\in V(c)\},&\frac{4}{d} < \al < 2_\sigma^\ast,\, m = 1,\, \sigma\in (\frac{d+1}{d+2},1),
\end{array}
\right.
\end{align}
	where
	\begin{gather*}
		S(c):=\sett{u\in H_{x,y}^\sigma:\|u\|_{2}^2=c},\quad
V(c):=\sett{u\in S(c):Q(u)=c},
	\end{gather*}
and $Q$	in the intercritical case is defined by
	\begin{equation}\label{poho-Q}
		Q(u)
		= \sigma \norm{  (-\Delta)^\frac {\sigma-1}2 \nabla_xu}^2_{2}
		-
		\frac{ \al d }{2(\al+2)}\|u\|_{L^{\al+2}}^{\al+2}.
	\end{equation}

We begin with a Pohozaev identity associated with \eqref{frac-c}, which plays a fundamental role in the upcoming analysis. As we shall see, proving such an identity on the waveguide manifolds for the fractional model is very different and more technical in comparison to the integral one. This is the main reason that the proof below is rather lengthy and cumbersome.
	\begin{lemma}
		If $u\in H_{x,y}^\sigma$ is a solution of \eqref{frac-c}, then $Q(u)=0$.
	\end{lemma}
	
	\begin{proof}
By Fourier expansion one easily verifies that
		\[
		\|  (-\Delta)^\frac {\sigma-1}2 \nabla_xu\|_{2}
		\leq \|(-\Delta)^{\frac \sigma 2}u\|_{2},
		\]
		hence $Q(u)$ is well-defined.
Thanks to the higher regularity of $u$ deduced from Proposition \ref{regularity}, the following pointwise identity holds:
		\begin{equation}\label{pointwise-id}
			(-\Delta)^{\sigma} (x\cdot\nabla_x u)+2\sigma (-\Delta)^{\sigma-1}\Delta_x u=x\cdot\nabla_x(-\Delta)^{\sigma} u.
		\end{equation}
Indeed, by considering Proposition \ref{regularity}, the above identity is obtained directly by applying the Fourier transform:
			\[
			\begin{split}\left(x\cdot\nabla_x(-\Delta)^{\sigma} u\right)^	\wedge
				&=-{\rm div}(\xi(|\xi|^2+|\eta|^2)^{\sigma}\hat{u}(\xi,\eta))
				=-
				(|\xi|^2+|\eta|^2)^{\sigma}{\rm div}
				(\xi\hat{u})
				-2\sigma |\xi|^2(|\xi|^2+|\eta|^2)^{\sigma-1}\hat{u}\\
				&=
				\left((-\Delta)^{\sigma} (x\cdot\nabla_x u)\right)^	\wedge
				+2\sigma \left((-\Delta)^{\sigma-1}\Delta_x\right)^	\wedge.
			\end{split}
			\]
Multiplying  $(-\Delta)^{\sigma} \bar{u}$ by $x\cdot\nabla_x u$ and applying  \eqref{pointwise-id} and the Parseval  identity, we obtain

		\[
		\begin{split}
			\Re \scal{(-\Delta )^{\sigma}u,x\cdot \nabla_x u }
			&=
			\Re \scal{ u,(-\Delta )^{\sigma}\left(x\cdot \nabla_x u\right)}
			\\
			&
			=-2\sigma   \scal{ u,(-\Delta )^{\sigma-1}\Delta_xu}
			+\Re \scal{xu, \nabla_x (-\Delta )^{\sigma}u} \\&=:I+II.
		\end{split}
		\]
While the term $I$ shall be directly kept in the Pohozaev identity, we need to take more care on the term $II$. For $II$, we first obtain
		\[
		\begin{split}
			II&=
			\Re\int_{\R^d\times\T^m}{\rm div}_x  \left(x\bar{u} (-\Delta )^{\sigma}u\right)  \dd x\dd y
			-\Re \scal{{\rm div}_x(xu),(-\Delta )^{\sigma}u}
			\\&
			=\Re\int_{\R^d\times\T^m}{\rm div}_x \left(x\bar{u} (-\Delta )^{\sigma}u \right) \dd x\dd y
			-\Re  \scal{du+x\cdot\nabla u,(-\Delta )^{\sigma}u } .
		\end{split}
		\]
		Since $u$ satisfies \eqref{frac-c},
		\[
		\begin{split}
			\Re \scal{x \cdot \nabla_x u,(-\Delta )^{\sigma}u}
			&=2\sigma \norm{  (-\Delta)^\frac {\sigma-1}2 \nabla_xu}_{2} ^2
			+ \Re\int_{\R^d\times\T^m}{\rm div}_x  \left(x\bar{u} \left(|u|^{\al}u-\omega u\right) \right)  \dd x\dd y
			\\
			&\qquad
			-	d \norm{ (-\Delta)^\frac {\sigma}2 u}_{2} 	  -\Re \scal{x\cdot\nabla_x u,(-\Delta )^{\sigma}u},
		\end{split}
		\]
		whence
		\begin{equation}\label{frac-id-poh}
			\begin{split}
				\Re \scal{x \cdot \nabla_x u,(-\Delta )^{\sigma}u}
				&=\sigma\norm{  (-\Delta)^\frac {\sigma-1}2 \nabla_xu}_{2} ^2
				+ \frac12\Re\int_{\R^d\times\T^m}{\rm div}_x  \left(x\bar{u} \left(|u|^{\al}u-\omega u\right) \right)  \dd x\dd y
				\\
				&\qquad
				-	\frac{d}{2} \norm{ (-\Delta)^\frac {\sigma}2 u}_{2} 	  .
			\end{split}
		\end{equation}
		Multiplying \eqref{frac-c} by $ x\cdot\nabla_x \bar{u} $,   taking integral over $\R^d\times\T^m$ and then  the real part of the result, it is derived from the fact $u\to0$ as $|x|\to\infty$ that
		\[\begin{split}
			&\Re \scal{x\cdot \nabla_xu, |u|^{\al}u} =-\frac{d}{\al+2}\|u\|_{L^{\al+2}}^{\al+2},
			\\&
			\Re \scal{x\cdot \nabla_x u,u} =-\frac d2\|u\|_{2}
		\end{split}\]
		and
		\[
		\begin{split}
			\sigma\norm{  (-\Delta)^\frac {\sigma-1}2 \nabla_xu}_{2} ^2
			&+\frac 12\Re\int_{\R^d\times\T^m}{\rm div}_x  \left(x\bar{u} \left(|u|^{\al}u-\omega u\right)\right)   \dd x\dd y
			+\frac{d}{\al+2}\|u\|_{L^{\al+2}}^{\al+2}\\&\qquad=\frac d2\norm{ (-\Delta)^\frac {\sigma}2 u}_{2}+\frac d2\omega\|u\|_{2}.
		\end{split}
		\]
		Multiplying \eqref{frac-c} by $\bar{u}$ and    taking integral over $\R^d\times\T^m$, we get
		\[
		\norm{(-\Delta)^{\frac \sigma 2}u}_{2}+
		\omega\|u\|_{2}=\|u\|_{\alpha+2}^{\al+2}.
		\]
		The last two identities lead us to $Q(u)=0$ as long as we can show that $  x\bar{u} (|u|^{\al}u-\omega u )\to0$ as $|x|\to\infty$. Indeed, it is enough to show from Proposition \ref{regularity} that $|x|u\to0$ as $|x|\to+\infty$. Since $u=K_\omega\ast(|u|^{\al}u)$,
		it suffices to prove that $|x|K_\omega\to 0$ as $|x|\to\infty$,
		where $K_\omega$ is defined as in \eqref{kernel}.
		
		If $|x|\geq1$,  it is straightforward to see that $|x|^{d+2\sigma }K_\omega(x,y)\in L_{x,y}^\infty$, whence $|x|^{d+2\sigma }u\in L_{x,y}^\infty$. This in turn implies that $|x|u\to0$ as $|x|\to\infty$. In the same manner, we may also show that $| x\cdot\nabla_x u |\to0$ as $|x|\to\infty$. This in turn completes the proof.
	\end{proof}	

We will also make use of the theories for variational problems defined on $\R^d$. To this end, we define
$$
\tilde{\E}(u):= \frac12 \|(-\Delta_x)^{\frac \sigma 2} u\|_{L^2(\mathbb{R}^d)}^2 -\frac{1}{\al+2}\|u\|_{L^{\al+2}(\mathbb{R}^d)}^{\al+2},
$$
$$
\tilde{S}(c):= \sett{u\in H^\sigma(\mathbb{R}^d):\|u\|_{L^2(\mathbb{R}^d)}^2=c}, \qquad
\tilde{V}(c):=\sett{u\in \tilde{S}(c):\tilde{Q}(u)=c},
$$
and
\begin{align}
	\tilde{m}_c:=
	\left\{
	\begin{array}{ll}
		\inf_{u\in \tilde{S}(c)}\{\tilde{\E}(u): u\in \tilde{S}(c)\},&\alpha\in(0,\min\{  \frac{4\sigma }{d},2_\sigma^\ast \}) \ \text{or} \ \alpha = \frac{4\sigma }{d}<2_\sigma^\ast,\\
		\\
		\inf_{u\in \tilde{S}(c)}\{\tilde{\E}(u): u\in \tilde{V}(c)\},&\frac{4}{d} < \al < 2_\sigma^\ast,\, m = 1,\, \sigma\in (\frac{d+1}{d+2},1),
	\end{array}
	\right.
\end{align}
where $\tilde{Q}$ in the intercritical case is defined by
\begin{equation}
	\tilde{Q}(u)
	= \sigma \|(-\Delta_x)^\frac {\sigma}2 u\|_{L^{2}(\R^d)}^2
	-
	\frac{ \al d }{2(\al+2)}\|u\|_{L^{\al+2}(\R^d)}^{\al+2}.
\end{equation}
	
Now we state some properties of the minimization problem $\tilde{m}_c$ considered on $\mathbb{R}^d$.
\begin{theorem}[Properties of $\tilde{m}_c$, \cite{Luo2020}]\label{properties}
	The following statements hold true:
	\begin{itemize}
		\item[(i)] If $0 < \al <   \frac{4\sigma }{d}$, then for any $c\in(0,\infty)$ we have $\tilde{m}_c\in(-\infty,0)$ and $\tilde{m}_c$ has a unique (up to a translation) positive radial minimizer $U_c\in \tilde{S}(c)$. Furthermore, $U_c = c^{\frac{2\sigma }{4\sigma - \al d }}U_1(c^{\frac{\al}{4\sigma - \al d }}x)$.
		
		\item[(ii)] If $\al =  \frac{4\sigma }{d}$, there exists $\widehat{c} > 0$ such that
		\begin{itemize}
			\item[(a)]for any $c\in(0,\widehat{c})$ we have $\tilde{m}_c=0$ and $\tilde{m}_c$ has no minimizer;
			\item[(b)]for any $c=\widehat{c}$ we have $\tilde{m}_c = 0$ and $\tilde{m}_c$ has a positive radial minimizer $U_c\in \tilde{S}(c)$;
			\item[(c)]for any $c\in(\widehat{c},\infty)$ we have $\tilde{m}_c = -\infty$ and $\tilde{m}_c$ has no minimizer.
		\end{itemize}
	\end{itemize}
\end{theorem}

Having all the preliminaries, we now classify and state our results according to their different given exponents.
	
	\subsection{The case $  \al < \frac{4\sigma }{d+m}$}
Our first goal is to prove the following existence result of the normalized ground states.
\begin{theorem}[Existence of normalized ground states] \label{exi1}
	Let $ \al \in(0,\frac{4\sigma }{d+m})$. Then for any $c\in(0,\infty)$ we have $m_c\in(-\infty,0)$ and $m_c$ possesses a minimizer.
	\end{theorem}
	\begin{proof}
		The proof is based on standard concentration compactness arguments. We split our proof into four steps.
		
		\subsubsection*{Step 1: Existence of a bounded minimizing sequence and non-vanishing weak limit}
		
		By assuming that a candidate in $S(c)$ is independent of $y$ we have $m_c\leq (2\pi)^m\tilde{m}_{(2\pi)^{-m}c}<0$, where the negativity of $\tilde{m}_{(2\pi)^{-m}c}$ is deduced from Theorem \ref{properties}. Let $(u_k)_k\subset S(c)$ be a minimizing sequence of $m_c$. Using Lemma \ref{localized-GN1} we infer that
		\begin{align} \label{mc larger than minus inf}
			m_c+c+o_k(1)&\geq c+\frac12 \|(-\Delta)^{\frac \sigma 2} u_k\|_{2}^2 -\frac{1}{\al+2}\|u_k\|_{\alpha+2}^{\al+2}\nonumber\\
			&\geq \frac12 \|u_k\|_{H_{x,y}^\sigma}^2 - \frac{C}{\al+2}c^{\frac{4\sigma - \al (d+m-2\sigma )}{4\sigma }}\|u_k\|^{\frac{(d+m) \al }{2\sigma }}_{H_{x,y}^\sigma}\nonumber\\
			&\geq \inf_{t>0}\frac12\left(t^2 - \frac{2C}{\al+2}c^{\frac{4\sigma - \al (d+m-2\sigma )}{4\sigma }}t^{\frac{(d+m) \al }{2\sigma }}\right) >-\infty,
		\end{align}
		where in the last inequality we used the fact that $\al\in(1,1 + \frac{4\sigma }{d+m})$. This in turn implies that $m_c>-\infty$. Finally, if $\|u_k\|_{H_{x,y}^\sigma}\to \infty$, then using the second inequality of \eqref{mc larger than minus inf} we deduce the contradiction $m_c\geq \infty$. We finally prove a non-vanishing weak limit of a minimizing sequence. By definition of $m_c$ we have
		\begin{align}\label{2.7}
			\|u_k\|_{\alpha+2}^{\al+2}=(\al+2)(-m_c+o_k(1) + \frac12 \|(-\Delta)^{\frac \sigma 2} u_k\|_{2}^2) \geq (\al+2)(-m_c+o_k(1)).
		\end{align}
		Thus combining with $m_c<0$ we infer that $\liminf_{n\to\infty}\|u_k\|_{\alpha+2}>0$. By Lemma \ref{localized-GN1} and \ref{localized-GN}, it follows that
		\[
		\sup_{x\in\mathbb{R}^d}\|u_k\|_{L^2(\mathbb{R}_x^d\times\mathbb{T}^m)} \gtrsim 1.
		\]
		Hence, there exists $z_k\in\mathbb{R}^d $, subsequence of $\{u_k\}$ (still denoted by the same letter) and $\varphi\in H_{x,y}^\sigma\setminus\{0\}$ such that $w_k(\cdot,\cdot):=u_k(\cdot+z_k,\cdot)\rightharpoonup\varphi$ in $H_{x,y}^\sigma$ as $k\to\infty$. This completes the proof of Step 1.
		
		\subsubsection*{Step 2: Continuity of the mapping $c\mapsto m_c$ on $(0,\infty)$}
		
		Let $(u_k)_k$ be a bounded minimizing sequence of $m_c$ (whose existence is guaranteed by Step 1). Let $(c_j)_j$ be a positive sequence with $\lim_{j\to\infty}c_j=c$. Then by definition of $m_c$ we infer that
		\begin{align}
			m_{c_j}&\leq \E(\frac{\sqrt{c_j}}{\sqrt{c}}u_k)=\frac{c_j}{c}\left( \frac{1}{2}\|(-\Delta)^{\frac \sigma 2} u_k\|_{2}^2
			-\frac{1}{\al+2}\left(\frac{c_j}{c}\right)^{\frac{\al}{2}}\|u_k\|_{\alpha+2}^{\al+2}\right)\nonumber\\
			&=\frac{c_j}{c}\left(\frac{1}{2}\|(-\Delta)^{\frac \sigma 2} u_k\|_{2}^2
			-\frac{1}{\al+2}\|u_k\|_{\alpha+2}^{\al+2}\right)
			+\frac{c_j}{c}\left(1-\left(\frac{c_j}{c}\right)^{\frac{\al}{2}}\right) \frac{1}{\al+2}\|u_k\|_{\alpha+2}^{\al+2}\nonumber\\
			&=\E(u_k)+\left(\frac{c_j}{c}-1\right)\left(\frac{1}{2}\|(-\Delta)^{\frac \sigma 2} u_k\|_{2}^2
			-\frac{1}{\al+2}\|u_k\|_{\alpha+2}^{\al+2}\right)\nonumber\\
			&\quad+\frac{c_j}{c}\left(1-\left(\frac{c_j}{c}\right)^{\frac{\al}{2}}\right) \frac{1}{\al+2}\|u_k\|_{\alpha+2}^{\al+2}=:\E(u_k)+I+II.\label{2.8}
		\end{align}
		Using $c_j= c+o_j(1)$ and the boundedness of $u_k$ in $H_{x,y}^\sigma$ we obtain that $I,II\to 0$ as $j\to\infty$. Since $(u_k)_k$ is a minimizing sequence for $m_c$, by combining standard diagonal arguments we conclude that
		\begin{align}
			\limsup_{j\to\infty}\,m_{c_j}\leq m_c.
		\end{align}
		Thus it is left to prove the opposite inequality. Let $u_j\in S(c_j)$ satisfy $E(u_j)\leq m_{c_j}+j^{-1}$. Using \eqref{mc larger than minus inf} we know that $(u_j)_j$ is a bounded sequence in $H_{x,y}^\sigma$. Using the definition of $m_c$ and arguing as in \eqref{2.8} we see that
		\[\begin{split}
			m_c&\leq \E\left(\frac{\sqrt{c}}{\sqrt{c_j}}u_j\right) \\
			&\leq m_{c_j}+j^{-1}+\left(\frac{c}{c_j}-1\right)\left(\frac{1}{2}\|(-\Delta)^{\frac \sigma 2} u_k\|_{2}^2
			-\frac{1}{\al+2}\|u_j\|_{\alpha+2}^{\al+2}\right) \\
			&\qquad+\frac{c}{c_j}\left(1-\left(\frac{c}{c_j}\right)^{\frac{\al}{2}}\right) \frac{1}{\al+2}\|u_j\|_{\alpha+2}^{\al+2}.
		\end{split}\]
		Taking $j\to\infty$ we then similarly deduce that
		\begin{align}
			m_c\leq\liminf_{j\to\infty}m_{c_j}.
		\end{align}
		This completes the proof of Step 2.
		
		\subsubsection*{Step 3: The mapping $c\mapsto c^{-1}m_c$ is strictly decreasing on $(0,\infty)$}
		
		Let $0<c_1<c_2<\infty$. Let $(u_k)$ be a minimizing sequence of $m_{c_1}$ with $\liminf_{n\to\infty}\|u_k\|_{\alpha+2}^{\al+2}>0$ (whose existence is guaranteed by Step 1). Then arguing as in \eqref{2.8} we obtain
		\begin{align}
			m_{c_2}\leq \frac{c_2}{c_1}(m_{c_1}+o_k(1)) + \frac{c_2}{c_1}\left(1-\left(\frac{c_2}{c_1}\right)^{\frac{\al}{2}}\right) \frac{1}{\al+2}\|u_k\|_{\alpha+2}^{\al+2}.
		\end{align}
		Thus combining with  $\liminf_{n\to\infty}\|u_k\|_{\alpha+2}^{\al+2}>0$ and taking $n\to\infty$ yields
		\begin{align}
			m_{c_2}<\frac{c_2}{c_1}m_{c_1}.
		\end{align}
		This completes the proof of Step 3.
		
		\subsubsection*{Step 4: Conclusion}
		
		We finish our proof in this final step. Let $(u_k)_k$ be a minimizing sequence of $m_c$, which possesses a non-vanishing weak limit $u\in H_{x,y}^\sigma\setminus\{0\}$ (whose existence is guaranteed by Step 1). Using weakly lower semicontinuity of norms it is necessary that $\|u\|_{2}^2\in(0,c]$. We hence show that $\|u\|_{2}^2=:c_1\in(0,c)$ shall lead to a contradiction, which in turn implies that $\|u\|_{2}^2=c$ and the desired proof follows. Using Br\'{e}zis-Lieb lemma and the fact that $L^2(\Omega)$ is a Hilbert space we have
		\begin{align}
			\E(u_k-u)+\E(u)&=\E(u_k)+o_k(1),\label{214}\\
			\M(u_k-u)+\M(u)&=\M(u_k)+o_k(1).\label{215}
		\end{align}
		By \eqref{215} we know that $\M(u_k-u)=c-c_1+o_k(1)$. Hence by the definition of $m_{c-c_1+o_k(1)}$
		$$\E(u_k-u)\geq m_{c-c_1+o_k(1)}.$$
		Taking $n\to\infty$ in \eqref{214} and using the continuity of the mapping $c\mapsto m_c$ (Step 2) we infer that
		\begin{align}
			m_c=\E(u)+\lim_{n\to\infty}\E(u_k-u)\geq m_{c_1}+m_{c-c_1}.\label{216}
		\end{align}
		Therefore combining this with Step 3 and $c_1,c-c_1<c$ we obtain
		\begin{align}
			m_c>\frac{c_1}{c}m_c+\frac{c-c_1}{c}m_c=m_c,
		\end{align}
		a contradiction. This completes the desired proof.
	\end{proof}

Next, we aim to prove a $y$-dependence result similar to Theorem \ref{thm y depend} but in the context of Theorem \ref{exi1}. To proceed, we first show that any normalized ground state will indeed be a ground state of the variational problem \eqref{minimization-nehari}. This shall help us to find conditions under which the minimizers of $m_c$ are semitrivial or nontrivial.

\begin{lemma}[A different characterization of normalized ground states] \label{relationship}
Let $ \al \in(0,\frac{4\sigma }{d+m})$, $u_c$ be the minimizer of $m_c$ given by Theorem \ref{exi1} and $\omega_c$ be the corresponding Lagrange multiplier. Then $\omega_c > 0$ and $\A_{\omega_c}(u_c)=c_\omega$, i.e. $u_c$ is a ground state solution of \eqref{minimization-nehari} with $\omega = \omega_c$.
\end{lemma}

\begin{proof}
		That $\omega_c > 0$ will be shown in the proof of Theorem \ref{thm iso norm masssup}. Now we prove that $\A_{\omega_c}(u_c)=c_\omega$. Note that
		$$
		\nb_\omega(tu) = (\|\dess u\|_{2}^2+  \omega\|u\|_{2}^2)t^2-\|u\|_{\alpha+2}^{\al+2}t^{\al+2}.
		$$
		Hence, it is not difficult to see that for any $u \in \x\setminus\{0\}$, there exists a unique $t(u) = t(\omega, u) > 0$ such that $\nb_\omega(t(u)u) = 0$ and $\A_\omega(t(u)u) = \max_{t > 0}\A_\omega(tu)$. Arguing as in \cite[Theorem C.3]{Hajaiej-Song-unique} we conclude that $\A_{\omega_c}(u_c)=c_\omega$ and the proof is complete.
	\end{proof}

\begin{lemma}[Asymptotic behaviors of $\omega_c$] \label{asymptotic behaviors}
		Let $ \al \in(0,\frac{4\sigma }{d+m})$, $u_c$ be the minimizer of $m_c$ given by Theorem \ref{exi1} and $\omega_c$ be the corresponding Lagrange multiplier. Then
$$\lim_{c \to 0^+}\omega_c = 0\quad\text{and}\quad\lim_{c \to \infty}\omega_c = \infty.$$
	\end{lemma}
	
	\begin{proof}
		We first prove $\lim_{c \to 0^+}\omega_c$. One easily verifies that $Q(u_c) = 0$ and $\nb_\omega(u_c) = 0$. Direct computation in combination with $\al \in(0,\frac{4\sigma }{d+m})$ yields
		\begin{align} \label{eq sim}
			\|\dess u_c\|_{2}^2 \sim \|u_c\|_{\alpha+2}^{\al+2} \gtrsim \omega_c\|u_c\|_{2}^2 = \omega_cc.
		\end{align}
		Lemma \ref{localized-GN1} shows that
		$$
		\|\dess u_c\|_{2}^2 \sim \|u_c\|_{\alpha+2}^{\al+2} \lesssim c^{\frac{\alpha+2}{2}-\frac{\alpha(d+m)}{4\sigma}}\|\dess u_c\|_{2}^{\frac{\alpha(d+m)}{2\sigma}}.
		$$
		Note that $\frac{\alpha(d+m)}{2\sigma} < 2$. Sending $c$ to $0$ one gets that
		$$
		\|\dess u_c\|_{2}^{2-\frac{\alpha(d+m)}{2\sigma}} \lesssim c^{\frac{\alpha+2}{2}-\frac{\alpha(d+m)}{4\sigma}} \to 0,
		$$
		implying that $\|\dess u_c\|_{2} \to 0$. Hence, $c_{\omega_c} = \A_{\omega_c}(u_c) \to 0$ as $c \to 0^+$. Theorem \ref{general-ground-state} shows that $c_\omega > 0$ and the mapping $\omega \mapsto c_\omega$ is continuous and strictly increasing on $(0,\infty)$. From this we conclude that $\omega_c \to 0$ from $c_{\omega_c}  \to 0$.
		
		Next we prove $\lim_{c \to \infty}\omega_c = \infty$. By Theorem \ref{thm y depend} and Lemma \ref{relationship} we know that when $\omega_c < \omega^*$ we have $\nabla_y u_c=0$. Using fundamental scaling arguments one obtains $u_c = \omega_c^{\frac{1}{\alpha}}Q_1(\omega_c^{\frac{1}{2\sigma}}x)$ (up to symmetries) where $Q_1 \in H^\sigma(\mathbb{R}^d)$ is the unqiue ground state (c.f. \cite{Frank2013,Frank2016}) of
		$$
		(-\Delta_x)^\sigma Q_1 + Q_1 = |Q_1|^\alpha Q_1.
		$$
		Thus, as $\omega_c \to 0$,
		$$
		c = (2\pi)^m\omega_c^{\frac{2}{\alpha}-\frac{d}{2\sigma}}\int_{\R^d}|Q_1|^2dx \to 0.
		$$
		Hence $\liminf_{c \to \infty}\omega_c \geq \delta > 0$ for some $\delta$, implying in turn that $\omega_cc \to \infty$. Then using \eqref{eq sim} we know that $\|\dess u_c\|_{2} \to \infty$ and consequently $c_{\omega_c} = \A_{\omega_c}(u_c) \to \infty$ as $c \to\infty$. Theorem \ref{general-ground-state} shows that $c_\omega > 0$ and the mapping $\omega \mapsto c_\omega$ is continuous and strictly increasing on $(0,\infty)$. Then, we obtain $\omega_c \to \infty$ from $c_{\omega_c}  \to \infty$. The proof is complete.
	\end{proof}

We are now ready to state the $y$-dependence result corresponding to Theorem \ref{exi1}.

\begin{theorem}[$y$-dependence of the normalized ground states]
		Let $ \al \in(0,\frac{4\sigma }{d+m})$, $u_c$ be the minimizer of $m_c$ given by Theorem \ref{exi1}. Then there exists $c^*\in(0,\infty)$ such that
		\begin{itemize}
			\item For all $c\in(0,c^*]$ we have $m_{c}=(2\pi)^m\tilde{m}_{(2\pi)^{-m}c}$. Moreover, for all $c\in(0,c^*)$ we have $\nabla_y u_c=0$.
			
			\item For all $c\in(c^*,\infty)$ we have $m_{c}<(2\pi)^m\tilde{m}_{(2\pi)^{-m}c}$. Moreover, for all $c\in(c^*,\infty)$ we have $\nabla_y u_c\neq 0$.
		\end{itemize}
	\end{theorem}
	
	\begin{proof}
		Define
		$$c^\ast:=\sup\{c>0:\;m_k=(2\pi)^m\tilde{m}_{(2\pi)^{-m}k}  \;\text{for}\;k\in(0,c)\}.$$
		It follows from Lemma \ref{asymptotic behaviors} and Theorem \ref{thm y depend} that $0<c^\ast<\infty$. Moreover, $m_{c}=(2\pi)^m\tilde{m}_{(2\pi)^{-m}c}$ also holds for all $c\in(0,c^*)$ by using the definition of $c^\ast$.
		
		Now we show that if there exists $\tilde{c} > 0$ such that $m_{\tilde{c}}<(2\pi)^m\tilde{m}_{(2\pi)^{-m}\tilde{c}}$, then $m_{c}<(2\pi)^m\tilde{m}_{(2\pi)^{-m}c}$ for all $c \in (\tilde{c},\infty)$. Let $u_{\tilde{c}} \in S(\tilde{c})$ be a minimizer for $m(\tilde{c})$. We use a scaling defined as follows:
		$$
		\bar{u}_{\tilde{c}} = \left( \frac{\tilde{c}}{c}\right) ^ku_{\tilde{c}}\left( \left( \frac{\tilde{c}}{c}\right)^lx,y\right)  \in S(c),
		$$
		where $k = \frac{2\sigma}{\alpha d-4\sigma} < 0, l = \frac{\alpha}{\alpha d-4\sigma} < 0$. Then we have
		\begin{align}
			m_c & \leq \E(\bar{u}_{\tilde{c}})
			\leq \frac12\|\left( -\Delta_x - \left( \frac{\tilde{c}}{c}\right)^{2l}\Delta_y\right) ^{\frac{\sigma}{2}}\bar{u}_{\tilde{c}}\|_2^2 - \frac1{\alpha+2}\|\bar{u}_{\tilde{c}}\|_{\alpha+2}^{\alpha+2} \nonumber \\
			& = \left( \frac{\tilde{c}}{c}\right)^{\frac{2\sigma(\alpha+2)}{\alpha d-4\sigma}}E(u_{\tilde{c}})
			= \left( \frac{\tilde{c}}{c}\right)^{\frac{2\sigma(\alpha+2)}{\alpha d-4\sigma}}m_{\tilde{c}} \nonumber \\
			& < (2\pi)^m\left( \frac{\tilde{c}}{c}\right)^{\frac{2\sigma(\alpha+2)}{\alpha d-4\sigma}}\tilde{m}_{(2\pi)^{-m}\tilde{c}}
			= (2\pi)^m\tilde{m}_{(2\pi)^{-m}c}.
		\end{align}
		This shows that for all $c\in(c^*,\infty)$ we have $m_{c}<(2\pi)^m\tilde{m}_{(2\pi)^{-m}c}$. That $m_{c^*}=(2\pi)^m\tilde{m}_{(2\pi)^{-m}c^*}$ follows from the continuity of the mappings $c\mapsto m_c$ and also $c\mapsto \tilde{m}_c$, where the former is deduced from the proof of Theorem \ref{exi1} and the latter can be proved in a similar way.
		
		Notice that $m_{c}<(2\pi)^m\tilde{m}_{(2\pi)^{-m}c}$ necessarily implies $\nabla_y u_c\neq 0$. It thus remains to prove that $\nabla_y u_c=0$ for all $c\in(0,c^*)$. Suppose on the contrary that there exists a minimizer $u_{\tilde{c}}$ satisfying $\nabla_y u_{\tilde{c}}\not=0$ for some $\tilde{c}\in(0,c^*)$. Take $c \in (\tilde{c},c^\ast)$. Then we obtain a self-contradictory inequality
		\begin{align}
			(2\pi)^m\tilde{m}_{(2\pi)^{-m}c} & = m_c \leq \E(\bar{u}_{\tilde{c}})
			< \frac12\|\left( -\Delta_x - \left( \frac{\tilde{c}}{c}\right)^{2l}\Delta_y\right) ^{\frac{\sigma}{2}}\bar{u}_{\tilde{c}}\|_2^2 - \frac1{\alpha+2}\|\bar{u}_{\tilde{c}}\|_{\alpha+2}^{\alpha+2} \nonumber \\
			& = \left( \frac{\tilde{c}}{c}\right)^{\frac{2\sigma(\alpha+2)}{\alpha d-4\sigma}}E(u_{\tilde{c}})
			= \left( \frac{\tilde{c}}{c}\right)^{\frac{2\sigma(\alpha+2)}{\alpha d-4\sigma}}m_{\tilde{c}} \nonumber \\
			& = (2\pi)^m\left( \frac{\tilde{c}}{c}\right)^{\frac{2\sigma(\alpha+2)}{\alpha d-4\sigma}}\tilde{m}_{(2\pi)^{-m}\tilde{c}}
			= (2\pi)^m\tilde{m}_{(2\pi)^{-m}c}.
		\end{align}
Summing up completes the desired proof.
	\end{proof}
	\subsection{The case $\al =   \frac{4\sigma }{d+m}$}
In the mass-critical case, we also have the following existence result of the ground states for mass $c$ not exceeding a critical number.	
	
	\begin{theorem}[Existence of normalized ground states]\label{exi2}
Let $\al =   \frac{4\sigma }{d+m}$. Then for any $c\in(0,\left(\frac{\al+2}{2C}\right)^{\frac{4\sigma }
			{ 4\sigma - \al  (d+m-2\sigma )} }]$ we have $m_c\in(-\infty,0)$ and $m_c$ has a minimizer, where $C$ is given by Lemma \ref{localized-GN1}.
	\end{theorem}
	
\begin{proof}
In this case, note that \eqref{mc larger than minus inf} yields to the boundedness of $m_c$ from below if $$
	c\in\left(0,\left(\frac{\al+2}{2C}\right)^{\frac{4\sigma }{4\sigma - \al  (d+m-2\sigma )} }\right].
	$$
	Then the existence of optimizers of $m_c$ follows already from the previous proof of the case when $ \al< \frac{4\sigma }{d+m}$.
\end{proof}

\subsection{The case  $  \frac{4\sigma }{d+m} < \al < \min\{  \frac{4\sigma }{d},2_\sigma^\ast \}$}
In this case, we shall rather use a localized argument to prove our main result. Define

	$$
	m_c^\rho:= \inf_{u \in S(c)\cap\mathcal{B}_\rho}\E(u),
	$$
	where	
	$$
	\mathcal{B}_\rho:=  \sett{u \in H_{x,y}^\sigma: \|(-\Delta)^{\frac \sigma 2} u\|_{2}^2 < \rho }.
	$$
Our main result is as follows:
\begin{theorem}[Existence of normalized ground states] \label{exi3}
Let $  \al\in(\frac{4\sigma }{d+m},\min\{  \frac{4\sigma }{d},2_\sigma^\ast \})$. Then for any $c\in(0,M_0\rho_0^l)$ we have $m_c^{\rho_0}\in(-\infty,0)$ and $m_c^{\rho_0}$ has a minimizer in $\mathcal{B}_{\rho_0}$, where $\rho_0 >0$ is the unique positive solution of
		\begin{equation}
			a(\rho)\rho = M_0\rho^l.
		\end{equation}
Here, the numbers $M_0$ and $l > 0$ are the ones given by Lemma \ref{lem-local min1}, and $a(\rho)$ is the one by Lemma \ref{lem-local min2}.
	\end{theorem}
	
\begin{remark}
\normalfont
Lemma \ref{lem-local min2} shows that a minimizer is in $\mathcal{B}_\rho$, not on its boundary $\partial \mathcal{B}_\rho$. Hence a minimizer of $m_c^\rho$ will actually solve the stationary equation \eqref{frac-c}.
\end{remark}

Before we turn to the proof of Theorem \ref{exi3}, we will also need several auxiliary lemmas, which are given in the following.
	\begin{lemma} \label{lem-local min1}
		Assume that $  \frac{4\sigma }{d+m} < \al < \min\{  \frac{4\sigma }{d},2_\sigma^\ast \}$, and
		\begin{equation} \label{est of c 1}
			c < M_0\rho^l,\qquad l = \frac{4\sigma - \al d }{4\sigma -(d-2\sigma )\al} > 0,
		\end{equation}
		where $$M_0 = (2\pi)^{\frac{2\sigma  \al m}{4\sigma -(d-2\sigma )\al}}\|(-\Delta_x)^{\frac{\sigma}{2}} U_1\|_{L^2(\mathbb{R}^d)}^{-\frac{4\sigma - \al d }{4\sigma -(d-2\sigma )\al}}.$$
		 Then $m_c^\rho\in(-\infty,0)$.
	\end{lemma}
	
	\textit{Proof.  } Obviously, $m_c^\rho > -\infty$ for any $\rho < \infty$. It remains to show that $m_c^\rho < 0$. In fact, assuming
	$$
	c < (2\pi)^{\frac{2\sigma  \al m}{4\sigma -(d-2\sigma )\al}}\|(-\Delta_x)^{\frac{\sigma}{2}} U_1\|_{L^2(\mathbb{R}^d)}^{-\frac{4\sigma - \al d }{4\sigma -(d-2\sigma )\al}} \rho^{\frac{4\sigma - \al d }{4\sigma -(d-2\sigma )\al}},
	$$
	direct computations yield that
	$$(2\pi)^m\|(-\Delta_x)^{\frac{\sigma}{2}}U_{(2\pi)^{-m}c}\|_{L^2(\mathbb{R}^d)}^2 < \rho,$$
	where $U_{(2\pi)^{-m}c}$ is given by Theorem \ref{properties}. Then by viewing that $U_{(2\pi)^{-m}c}$ in $S(c)$ is independent of $y$ we have $m_c^\rho\leq (2\pi)^m\tilde{m}_{(2\pi)^{-m}c}<0$.
	
	\begin{lemma} \label{lem-local min2}
		Assume that $  \frac{4\sigma }{d+m} < \al < \min\{  \frac{4\sigma }{d},2_\sigma^\ast \}$, and
		\begin{equation} \label{est of c 2}
			c < a(\rho)\rho,
		\end{equation}
		where $a = a(\rho) > 0$ is the unique positive solution of
		\begin{equation} \label{value of a}
			(a+1)^ka = M_1\rho^{-(k+1)}, k = \frac{(d+m) \al -4\sigma }{4\sigma  -  \al (d+m-2\sigma )} > 0
		\end{equation}
and the number $M_1$ is given by $M_1 = \left(\frac{\al+2}{2C}\right)^{\frac{4\sigma }{4\sigma  -  \al (d+m-2\sigma )}}$ and $C$ is given by Lemma \ref{localized-GN1}. Then there exists $\tau > 0$ small enough such that
		\begin{equation} \label{est of mc 1}
			m_c^\rho < \inf_{u \in S(c)\cap(\mathcal{B}_\rho\backslash\mathcal{B}_{\rho-\tau})}E(u).
		\end{equation}
	\end{lemma}
	
	\textit{Proof.  } For any $u \in \mathcal{B}_\rho\backslash\mathcal{B}_{\rho-\tau}$, Lemma \ref{localized-GN1} shows that
	\begin{equation}
		E(u) > \frac{1}{2}(\rho-\tau) - \frac{C}{\al+2}c^{\frac{4\sigma  -  \al (d+m-2\sigma )}{4\sigma }}(c + \rho)^{\frac{(d+m) \al }{4\sigma }}.
	\end{equation}
	Let $k = \frac{(d+m) \al -4\sigma }{4\sigma  -  \al (d+m-2\sigma )} > 0$. For any $a > 0$, if $c < a\rho$ and $c <M_1(a+1)^{-k}\rho^{-k}$, where $M_1 = \left(\frac{\al+2}{2C}\right)^{\frac{4\sigma }{4\sigma  -  \al (d+m-2\sigma )}}$, direct computations yield that
	\begin{equation} \label{b}
		b = b(c):= \frac{1}{2}\rho - \frac{C}{\al+2}c^{\frac{4\sigma  -  \al (d+m-2\sigma )}{4\sigma }}(c + \rho)^{\frac{(d+m) \al }{4\sigma }} > 0.
	\end{equation}
	Take $v \in S(c)\cap\mathcal{B}_\epsilon$ where $\epsilon > 0$ is small enough. Then we get that
	\begin{equation}
		\E(u) > b - \frac{1}{2}\tau > \epsilon > \|(-\Delta)^{\frac \sigma 2} v\|_{2}^2 - \frac{1}{\al+2}\|v\|_{\alpha+2}^{\al+2} = \E(v),
	\end{equation}
	for small $\tau > 0$ such that $\epsilon + \frac{1}{2}\tau < b$, implying that \eqref{est of mc 1} holds true. Noticing that both $c < a\rho$ and $c <M_1(a+1)^{-k}\rho^{-k}$ hold true if $c < a(\rho)\rho$ where $a = a(\rho) > 0$ is the unique positive solution of \eqref{value of a}, we complete the proof.
	
We are now in a position to prove Theorem \ref{exi3}.


\begin{proof}[Proof of Theorem \ref{exi3}]
Before giving the main proof we shall first establish some useful estimates. Direct computations show that
	$$
	(a(\rho)\rho)' = a'(\rho)\rho + a(\rho) = -\frac{k}{ka(\rho)+ a(\rho) + 1}M_1(a(\rho)+1)^{-k}\rho^{-(k+1)} < 0,
	$$
	where $a(\rho)$ is given by Lemma \ref{lem-local min2}. Hence, the equation
	\begin{equation}
		a(\rho)\rho = M_0\rho^l
	\end{equation}
	has a unique solution $\rho_0$ in $(0,\infty)$, where $M_0$ and $l > 0$ are given by Lemma \ref{lem-local min1}. Let $c < M_0\rho_0^l$. Then both \eqref{est of c 1} and \eqref{est of c 2} hold true. Lemma \ref{lem-local min1} and \ref{lem-local min2} show that $m_c^{\rho_0}\in(-\infty,0)$ and \eqref{est of mc 1} holds for $\rho = \rho_0$.

We now prove the main claim of the theorem. We split our proof into four steps.
	
	\subsubsection*{Step 1: Existence of a bounded minimizing sequence and non-vanishing weak limit}
	
	By \eqref{est of mc 1}, we can take a minimizing sequence $(u_k)_k\subset S(c)\cap\mathcal{B}_{\rho-\tau}$ of $m_c^{\rho_0}$ where $\tau$ depends on $c$. Obviously, $(u_k)_k$ is bounded in $H_{x,y}^\sigma$. Similar to Step 1 in the proof of Theorem \ref{exi1}, we are able to prove a non-vanishing weak limit of a minimizing sequence. This complete the proof of Step 1.
	
	\subsubsection*{Step 2: Continuity of the mapping $c\mapsto m_c^{\rho_0}$ on $(0,M_0\rho_0^l)$}
	
	Let $(c_j)_j$ be a positive sequence with $\lim_{j\to\infty}c_j=c$. From \eqref{b}, we can find $\delta > 0$ independent on $c_j$ such that $b(c_j) > \delta$. Hence, $\tau$ can be chosen independent of $c_j$. Then mimicking Step 2 in the proof of Theorem \ref{exi1}, we complete the proof of Step 2.
	
	\subsubsection*{Step 3: The mapping $c\mapsto c^{-1}m_c^{\rho_0}$ is strictly decreasing on $(0,M_0\rho_0^l)$.}
	
	Let $0<c_0<M_0\rho_0^l$. We then take $\epsilon_0 > 0$ small and $\tau = \tau(c_0,\epsilon_0)$ such that the minimizing sequences of $m_c^{\rho_0}$ for all $c \in (c_0-\epsilon_0,c_0]$ lying in $\mathcal{B}_{\rho_0-\tau}$. Let $\epsilon > 0$ be small enough such that $\epsilon < \epsilon_0$ and $\epsilon<\frac{\tau}{\rho_0}c_0$ (ensuring that $\frac{c_0}{c_0-\epsilon}(\rho_0-\tau) < \rho_0$). Similar to Step 3 in the proof of Theorem \ref{exi1}, we may prove that
	\begin{align}
		m_{c_0}^{\rho_0}<\frac{c_0}{c_0-\epsilon}m_{c_0-\epsilon}^{\rho_0}.
	\end{align}
	This completes the proof of Step 3.
	
	\subsubsection*{Step 4: Conclusion}
	
	We finish our proof in this final step. Let $(u_k)_k$ be a minimizing sequence of $m_c^{\rho_0}$, which possesses a non-vanishing weak limit $u\in H_{x,y}^\sigma\setminus\{0\}$ (whose existence is guaranteed by Step 1). Using weakly lower semicontinuity of norms it is necessary that $\|u\|_{2}^2\in(0,c]$. We hence show that $\|u\|_{2}^2=:c_1\in(0,c)$ shall lead to a contradiction, which in turn implies that $\|u\|_{2}^2=c$ and the desired proof follows. Using the Br\'{e}zis-Lieb lemma and the fact that $L^2(\Omega)$ is a Hilbert space we have
	\begin{align}
		\|(-\Delta)^{\frac \sigma 2} (u_k-u)\|_{2}^2+\|(-\Delta)^{\frac \sigma 2} u\|_{2}^2&=\|(-\Delta)^{\frac \sigma 2} u_k\|_{2}^2+o_k(1),\label{313}\\
		\E(u_k-u)+\E(u)&=\E(u_k)+o_k(1),\label{314}\\
		\M(u_k-u)+\M(u)&=\M(u_k)+o_k(1).\label{315}
	\end{align}
	By \eqref{315} we know that $\M(u_k-u)=c-c_1+o_k(1)$. By \eqref{313} we know that
	$$
	\|(-\Delta)^{\frac \sigma 2} (u_k-u)\|_{2}^2 < \rho_0 - \|(-\Delta)^{\frac \sigma 2} u\|_{2}^2+o_k(1).
	$$
	Hence by the definition of $m_{c-c_1+o_k(1)}^{\rho_0}$
	$$\E(u_k-u)\geq m_{c-c_1+o_k(1)}^{\rho_0}.$$
	Taking $n\to\infty$ in \eqref{314} and using the continuity of the mapping $c\mapsto m_c^{\rho_0}$ (Step 2) we infer that
	\begin{align}
		m_c^{\rho_0}=\E(u)+\lim_{n\to\infty}\E(u_k-u)\geq m_{c_1}^{\rho_0}+m_{c-c_1}^{\rho_0}.
	\end{align}
	Therefore combining this with Step 3 and $c_1,c-c_1<c$ we obtain
	\begin{align}
		m_c^{\rho_0}>\frac{c_1}{c}m_c^{\rho_0}+\frac{c-c_1}{c}m_c^{\rho_0}=m_c^{\rho_0},
	\end{align}
	a contradiction. This completes the desired proof.
\end{proof}
	
	
	\subsection{The case $\frac{4}{d} \leq \al < 2_\sigma^\ast$}\label{sec3.4}
Before we turn to the main results and their proofs, we shall first point out that by straightforward calculation, we immediately see that the restriction $\frac{4}{d} \leq \al < 2_\sigma^\ast$ will indeed imply that $m = 1$ and $\sigma\in (\frac{d+1}{d+2},1)$. Moreover, we would also like to underline that one might think that the restriction $\al\geq\frac{4}{d}$ could be replaced by the more reasonable constraint $\al>\frac{4\sigma}{d} $ as $\frac{4\sigma}{d} $ stands for the mass-critical index. As we shall see, however, such restriction will be a necessary condition for an application of the Pohozaev identity. For details, we refer to Lemma \ref{radial manifold} below.

In what follows, we prove some useful properties of the mappings $t\mapsto \mK(u^t)$ and $m_{c}$, where $u^t$ is the scaling operator defined through
	\begin{align}\label{def of scaling}
		u^t(x,y):=t^{\frac d2}u(tx,y).
	\end{align}
	These properties will play a central role in the proof of our main results.
	\begin{lemma}[Property of the mapping $t\mapsto \mK(u^t)$]\label{unique t}
		Let $c>0$ and $u\in S(c) \cap H^{2\sigma +1}(\Omega)$. Then the following statements hold true:
		\begin{enumerate}
			\item[(i)] $\frac{\dd}{\dd t}\E(u^t)=t^{-1} Q(u^t)$ for all $t>0$.
			\item[(ii)] There exists some $t^*=t^*(u)>0$ such that $u^{t^*}\in V(c)$.
			\item[(iii)] We have $t^*<1$ if and only if $\mK(u)<0$. Moreover, $t^*=1$ if and only if $\mK(u)=0$.
			\item[(iv)] Following inequalities hold:
			\begin{equation*}
				Q(u^t) \left\{
				\begin{array}{lr}
					>0, &t\in(0,t^*) ,\\
					<0, &t\in(t^*,\infty).
				\end{array}
				\right.
			\end{equation*}
			\item[(v)] $\mH(u^t)<\mH(u^{t^*})$ for all $t>0$ with $t\neq t^*$.
		\end{enumerate}
	\end{lemma}
	\begin{proof}
We first prove (i). Note that $(-\Delta_{x})^{\sigma}$ cannot be separated from $(-\Delta)^{\sigma}$ since the latter is nonlocal and isotropic. Hence, $\|(-\Delta)^{\frac{\sigma}{2}} u^t\|_{2}^2$ cannot be written as a form like $t^k\|(-\Delta)^{\frac{\sigma}{2}} u\|_{2}^2$ for some $k$. For this reason, we must compute carefully. For this purpose, we define $\Psi_u(t):= \E(u^t),\,t > 0$. Then
\begin{align*}
			\Psi_u'(t)&=\frac12\frac{\dd}{\dd t}\|(-\Delta)^{\frac{\sigma}{2}} u^t\|_{2}^2 - \frac{ \al d }{2(\al+2)}t^{\frac{ \al d }{2}-1}\|u\|_{\alpha+2}^{\al+2} \\
			&=\int_{\Omega}(-\Delta)^{\frac{\sigma}{2}} u^t (-\Delta)^{\frac{\sigma}{2}} \frac{\dd}{\dd t}u^t\dd x\dd y - \frac1t\frac{ \al d}{2(\al+2)}\|u^t\|_{\alpha+2}^{\al+2}\\
			&=\frac{d}{2t}\|(-\Delta)^{\frac{\sigma}{2}} u^t\|_{2}^2 + \frac1t\int_{\Omega}(-\Delta)^{\frac{\sigma}{2}} u^t(-\Delta)^{\frac{\sigma}{2}}(x\cdot\nabla_x)u^t\dd x\dd y - \frac1t\frac{ \al d}{2(\al+2)}\|u^t\|_{\alpha+2}^{\al+2},
		\end{align*}
		where we used $\frac{\dd}{\dd t}u^t = \frac{d}{2t}u^t + \frac1t(x\cdot\nabla_x)u^t$.


Using \eqref{pointwise-id}\footnote{We may simply assume here that $u$ is smooth. The general case follows from a standard density argument.} we obtain
		\[\begin{split}
			&\,\int_{\Omega}(-\Delta)^{\frac{\sigma}{2}} u^t(-\Delta)^{\frac{\sigma}{2}}(x\cdot\nabla_x)u^t\dd x\dd y\nonumber\\
			=& -\int_{\Omega} u^t(-\Delta)^{\sigma}(x\cdot\nabla_x)u^t\dd x\dd y \\
			=& -\int_{\Omega} u^t(x\cdot\nabla_x)(-\Delta)^{\sigma}u^t\dd x\dd y + 2\sigma \int_{\Omega} u^t(-\Delta)^{\sigma-1}\Delta_xu^t\dd x\dd y \\
			=&\, d\int_{\Omega} u^t(-\Delta)^{\sigma}u^t\dd x\dd y + \int_{\R^d\times\T^m} (x\cdot\nabla_x)u^t(-\Delta)^{\sigma}u^t\dd x\dd y \\& \qquad+ 2\sigma \int_{\Omega} u^t(-\Delta)^{\sigma-1}\Delta_xu^t\dd x\dd y \\
			=& -d\|(-\Delta)^{\frac{\sigma}{2}} u^t\|_{2}^2 -\int_{\Omega}(-\Delta)^{\frac{\sigma}{2}} u^t(-\Delta)^{\frac{\sigma}{2}}(x\cdot\nabla_x)u^t\dd x\dd y\\&\qquad + 2\sigma \|(-\Delta)^{\frac{\sigma-1}{2}}\nabla_x u^t\|_{2}^2,
		\end{split}\]
		implying that
		\begin{align*}
			\int_{\Omega}(-\Delta)^{\frac{\sigma}{2}} u^t(-\Delta)^{\frac{\sigma}{2}}(x\cdot\nabla_x)u^t\dd x\dd y = -\frac d2\|(-\Delta)^{\frac{\sigma}{2}} u^t\|_{2}^2 + \sigma\|(-\Delta)^{\frac{\sigma-1}{2}}\nabla_x u^t\|_{2}^2.
		\end{align*}
		Hence, we get that
		\begin{align*}
			\Psi_u'(t) = \frac1t\left(\sigma\|(-\Delta)^{\frac{\sigma-1}{2}}\nabla_x u^t\|_{2}^2 - \frac{ \al d }{2(\al+2)}\|u^t\|_{\alpha+2}^{\al+2}\right).
		\end{align*}
		This completes the proof of (i).
		
		To prove (ii), we first need to compute $\Psi_u''(t)$. From $$t\Psi_u'(t) = \sigma\|(-\Delta)^{\frac{\sigma-1}{2}}\nabla_x u^t\|_{2}^2 - \frac{ \al d }{2(\al+2)}\|u^t\|_{\alpha+2}^{\al+2},$$ one deduces that
		\begin{align*}
			\Psi_u'(t) + t\Psi_u''(t)& = \sigma\frac{\dd}{\dd t}\|(-\Delta)^{\frac{\sigma-1}{2}}\nabla_x u^t\|_{2}^2 - \frac{ \al ^2d^2}{4(\al+2)}\frac1t\|u^t\|_{\alpha+2}^{\al+2} \\
			&= \frac{d\sigma}{t}\|(-\Delta)^{\frac{\sigma-1}{2}}\nabla_x u^t\|_{2}^2 + \frac{2\sigma }{t} \int_{\Omega}(-\Delta)^{\frac{\sigma-1}{2}}\nabla_x u^t(-\Delta)^{\frac{\sigma-1}{2}}\nabla_x(x\cdot\nabla_x)u^t\dd x\dd y \\
			&\quad- \frac{ \al ^2d^2}{4(\al+2)}\frac1t\|u^t\|_{\alpha+2}^{\al+2}.
		\end{align*}
		Using \eqref{pointwise-id}, we obtain that
		\begin{align*}
			(-\Delta)^{\sigma-1}\Delta_{x} (x\cdot\nabla_x) u &= (-\Delta)^{\sigma-1}(2 + x\cdot\nabla_x)\Delta_x u \\
			&=2(-\Delta)^{\sigma-1}\Delta_{x}u - 2(\sigma-1)(-\Delta)^{\sigma-2}\Delta_{x}^2u + (x\cdot\nabla_x)(-\Delta)^{\sigma-1}\Delta_{x}u.
		\end{align*}
		Hence,
		\[\begin{split}
			\int_{\Omega}(-\Delta)^{\frac{\sigma-1}{2}}\nabla_x u^t&(-\Delta)^{\frac{\sigma-1}{2}}\nabla_x(x\cdot\nabla_x)u^t\dd x\dd y\\& = \int_{\Omega} u^t(-\Delta)^{\sigma-1}\Delta_x(x\cdot\nabla_x)u^t\dd x\dd y \\
			&= \int_{\Omega} u^t(x\cdot\nabla_x)(-\Delta)^{\sigma-1}\Delta_xu^t\dd x\dd y + 2\int_{\Omega} u^t(-\Delta)^{\sigma-1}\Delta_xu^t\dd x\dd y \\&\quad- 2(\sigma-1)\int_{\Omega} u^t(-\Delta)^{\sigma-2}\Delta_{x}^2u^t\dd x\dd y \\
			&= -d\int_{\Omega} u^t(-\Delta)^{\sigma-1}\Delta_xu^t\dd x\dd y - \int_{\R^d\times\T^m} (x\cdot\nabla_x)u^t(-\Delta)^{\sigma-1}\Delta_xu^t\dd x\dd y \\
			&\quad+ 2\int_{\Omega} u^t(-\Delta)^{\sigma-1}\Delta_xu^t\dd x\dd y - 2(\sigma-1)\int_{\Omega} u^t(-\Delta)^{\sigma-2}\Delta_{x}^2u^t\dd x\dd y \\
			&= -d\|(-\Delta)^{\frac{\sigma-1}{2}}\nabla_x u^t\|_{2}^2  -\int_{\Omega}(-\Delta)^{\frac{\sigma-1}{2}}\nabla_x u^t(-\Delta)^{\frac{\sigma-1}{2}}\nabla_x(x\cdot\nabla_x)u^t\dd x\dd y \\
			&\quad+ 2\|(-\Delta)^{\frac{\sigma-1}{2}}\nabla_x u^t\|_{2}^2 + 2(\sigma-1)\|(-\Delta)^{\frac{\sigma-2}{2}}(-\Delta_x) u^t\|_{2}^2,
		\end{split}\]
which in turn implies
		\begin{align*}
			&\int_{\Omega}(-\Delta)^{\frac{\sigma-1}{2}}\nabla_x u^t(-\Delta)^{\frac{\sigma-1}{2}}\nabla_x(x\cdot\nabla_x)u^t\dd x\dd y \\
			&\quad =-\frac d2\|(-\Delta)^{\frac{\sigma-1}{2}}\nabla_x u^t\|_{2}^2 + \|(-\Delta)^{\frac{\sigma-1}{2}}\nabla_x u^t\|_{2}^2 + (\sigma-1)\|(-\Delta)^{\frac{\sigma-2}{2}}(-\Delta_x) u^t\|_{2}^2.
		\end{align*}
		Thus we obtain
		\[	\begin{split}
			&\Psi_u'(t) + t\Psi_u''(t) \\
			&\quad	=\frac1t \left(2\sigma \|(-\Delta)^{\frac{\sigma-1}{2}}\nabla_x u^t\|_{2}^2 + 2\sigma (\sigma-1)\|(-\Delta)^{\frac{\sigma-2}{2}}(-\Delta_x) u^t\|_{2}^2- \frac{ \al ^2d^2}{4(\al+2)}\|u^t\|_{\alpha+2}^{\al+2}\right).
		\end{split}\]
		When $\Psi_u'(t) = 0$, since $\al d\geq 4$, one gets that
		\begin{align*}
			t^2\Psi_u''(t) = \sigma\left(2-\frac{ \al d }{2}\right)\|(-\Delta)^{\frac{\sigma-1}{2}}\nabla_x u^t\|_{2}^2 + 2\sigma (\sigma-1)\|(-\Delta)^{\frac{\sigma-2}{2}}(-\Delta_x) u^t\|_{2}^2 < 0.
		\end{align*}
		This shows that any critical point of $\Psi_u$ must be maxima and hence $\Psi_u$ has at most one critical point when $t > 0$.
		
		Now we show the existence of such a critical point. On the one hand,
		\begin{align*}
			\Psi_u(t) &= \frac12\|(-\Delta)^{\frac{\sigma}{2}} u^t\|_{2}^2 - \frac{1}{\al+2}\|u^t\|_{\alpha+2}^{\al+2} \\
			&\geq \frac12\|(-\Delta_x)^{\frac{\sigma}{2}} u^t\|_{2}^2 - \frac{1}{\al+2}\|u^t\|_{\alpha+2}^{\al+2} \\
			&= t^{2\sigma }\frac12\|(-\Delta_x)^{\frac{\sigma}{2}} u\|_{2}^2 - t^{\frac{ \al d }{2}}\frac{1}{\al+2}\|u\|_{\alpha+2}^{\al+2}.
		\end{align*}
		When $\al >  \frac{4\sigma }{d}$, for $t$ small enough, we have $\Psi_u(t) > 0$. On the other hand, direct calculations yield that as $t \to \infty$,
		$$
		\|(-\Delta)^{\frac{\sigma}{2}} u^t\|_{2}^2 \leq 2\|(-\Delta_x)^{\frac{\sigma}{2}} u^t\|_{2}^2.
		$$
		Hence, for $t$ large enough, we infer that
		\begin{align*}
			\Psi_u(t)
			&\leq \|(-\Delta_x)^{\frac{\sigma}{2}} u^t\|_{2}^2 - \frac{1}{\al+2}\|u^t\|_{\alpha+2}^{\al+2} \\
			&= t^{2\sigma }\|(-\Delta_x)^{\frac{\sigma}{2}} u\|_{2}^2 - t^{\frac{ \al d }{2}}\frac{1}{\al+2}\|u\|_{\alpha+2}^{\al+2},
		\end{align*}
		implying that $\Psi_u(t) \to -\infty$ as $t \to \infty$. Therefore, there exists a unique $t^\ast = t^\ast(u) > 0$ such that $\Psi_u'(t^\ast) = 0$. Furthermore, $\Psi_u'(t) > 0$ for $t \in (0,t^*)$ and $\Psi_u'(t) < 0$ for $t \in (t^*,\infty)$. Now (ii), (iv), and (v) follow immediately.
		
		Finally we prove (iii). Let $\mK(u)<0$. Then
		$$0>\mK(u)=\frac{Q(u^1)}{1}=\Psi_u'(1),$$
		which is only possible as long as $t^*<1$. Conversely, let $t^*<1$. Then using the fact that $\Psi_u'(t) < 0$ for $t \in (t^*,\infty)$ we obtain
		$$\mK(u)=\Psi_u'(1)<0. $$
		This completes the proof of (iii) and in turn the desired proof.
	\end{proof}

	\begin{lemma}[Property of the mapping $c\mapsto m_{c}$]\label{radial monotone lemma}
		
		The mapping $c\mapsto m_{c}$ is monotone decreasing on $(0,\infty)$.
	\end{lemma}
	
	\begin{proof}
		To prove that the mapping $c\mapsto m_{c}$ is monotone decreasing, it is enough to verify that for any $c_1 < c_2$ and $\epsilon > 0$ arbitrary, we have $m_{c_2} \leq m_{c_1} + \epsilon$. From the definition of $m_{c}$, Lemma \ref{unique t} and Remark \ref{dense sunmanifold} below, there exists $u \in V(c_1)\cap H^{2\sigma +1}(\Omega)$ such that
		$$
		\E(u) \leq m_{c_1} + \frac{\epsilon}{2}, \quad\max_{t > 0}\E(u^t)=\E(u).
		$$
		Let $\eta \in C_0^\infty(\mathbb{R}^d)$ be radial and such that
		\begin{equation}
			\eta(x) =
			\left\{
			\begin{array}{cc}
				1, & |x| \leq 1, \\
				\in (0,1), & 1 < |x| < 2, \\
				0, & |x| \geq 2.
			\end{array}
			\right.
		\end{equation}
		For any small $\delta > 0$, let $\tilde{u}_\delta(x,y) = \eta(\delta x)u(x,y)$. It is standard to show that as $\delta \rightarrow 0$,
		$$
		\tilde{u}_\delta \rightarrow u \ \text{in} \ H^{2\sigma +1}(\Omega).
		$$
		Let $v(\cdot,y) \in C_0^\infty(\mathbb{R}^d)$ be such that ${\rm supp}\; v(\cdot,y) \subset B_{4/\delta+1}\backslash B_{4/\delta}$, $\|(-\partial_y^2)^{\frac{\sigma}{2}} v\|_{2}^2 = 0$, and define
		$$
		v_0 = \frac{c_2-\|\tilde{u}_\delta\|_{2}^{2}}{\|v\|_{2}^{2}}v,
		$$
		for which we have $\|v_0\|_{2}^{2} = c_2-\|\tilde{u}_\delta\|_{2}^{2}$. For any $t \in (0, 1)$, we now set $w_\delta^t = \tilde{u}_\delta + v_0^t$. Note that
		$$
		{\rm dist}\left({\rm supp\,} \tilde{u}_\delta(\cdot,y), {\rm supp\,} v_0^t(\cdot,y)\right) > \frac{2}{\delta} > 0,
		$$
		then $w_\delta^t \in S(c_2)$. As $t \to 0$ it is easy to see that $v_0^t \to 0$ in $H^{2\sigma +1}(\Omega)$ and thus as $t, \delta \rightarrow 0$ we infer that
		$$
		w_\delta^t \rightarrow u \ \text{in} \ H^{2\sigma +1}(\Omega).
		$$
		Hence, for any $t, \delta > 0$ small enough,
		$$
		m_{c_2} \leq \max_{l > 0}\E((w_\delta^t)^l) \leq \max_{l > 0}\E(u^l) + \frac{\epsilon}{2} \leq m_{c_1} + \epsilon.
		$$
		The proof is complete.
	\end{proof}

We next state a useful Gagliardo-Nirenberg inequality on $\Omega=\R^d\times\T$, which, on the contrary to the classical ones, satisfies certain scaling-invariant property. See also \cite{Luo_Waveguide_MassCritical,Luo_inter,Luo_energy_crit,luo2022sharp} for its original application for the integral NLS.


\begin{lemma}[Scale-invariant Gagliardo-Nirenberg inequality on $\Omega$] \label{GN}
		There exists some $C > 0$ such that for all $u \in H^\sigma_{x,y}$ we have
		\begin{equation}\label{eq gn}
			\|u\|_{\alpha+2}^{\alpha+2} \leq C\|(-\Delta_x)^{\frac{\sigma}{2}} u\|_2^{\frac{\alpha d}{2\sigma}}\|u\|_2^{\frac{4\sigma-\alpha(d+1-2\sigma)}{2\sigma}}\left(\|u\|_2^{\frac{\alpha}{2\sigma}} + \|(-\partial_y^2)^{\frac{\sigma}{2}} u\|_2^{\frac{\alpha}{2\sigma}}\right).
		\end{equation}
	\end{lemma}
	\begin{proof}
		For a function $u \in H^\sigma_{x,y}$ define $m(u)(x) := (2\pi)^{-1}\int_{\mathbb{T}}u(x,y)dy$. By triangular inequality we obtain
		$$
		\|u\|_{\alpha+2} \leq \|m(u)\|_{\alpha+2} + \|u-m(u)\|_{\alpha+2}.
		$$
		We then estimate $m(u)$ and $u - m(u)$ separately. Since $m(u)$ is independent of $y$, using the standard Gagliardo-Nirenberg inequality on $\mathbb{R}^d$ we conclude that
		\begin{eqnarray}
			\|m(u)\|_{\alpha+2} &\lesssim& \|m(u)\|_{L_x^{\alpha+2}} \nonumber \\
			&\lesssim& \|m(u)\|_{L_x^2}^{1-\frac{\alpha d}{2\sigma(\alpha+2)}}\|(-\Delta_x)^{\frac{\sigma}{2}} m(u)\|_{L_x^2}^{\frac{\alpha d}{2\sigma(\alpha+2)}} \nonumber \\
			&\lesssim& \|u\|_2^{1-\frac{\alpha d}{2\sigma(\alpha+2)}}\|(-\Delta_x)^{\frac{\sigma}{2}} u\|_2^{\frac{\alpha d}{2\sigma(\alpha+2)}}.
		\end{eqnarray}
		For $u - m(u)$, we recall the following well-known Sobolev's inequality on $\mathbb{T}$ for functions with zero mean (see for instance \cite{sobolev_torus}):
		\begin{equation} \label{mean}
			\|u-m(u)\|_{L_y^{\alpha+2}} \leq \|u\|_{\dot{H}_y^{\frac{\alpha}{2(\alpha+2)}}}.
		\end{equation}
		Writing $u$ into the Fourier series $u(x, y) = \Sigma_ke^{iky}u_k(x)$ with respect to $y$ and followed by \eqref{mean}, Minkowski, Gagliardo-Nirenberg on $\mathbb{R}^d$ and H\"{o}lder inequalities (through the identity
		$$
		1 = \frac{\alpha}{2\sigma(\alpha+2)} + \frac{(2\sigma-1)\alpha+4\sigma-\alpha d}{2\sigma(\alpha+2)} + \frac{\alpha d}{2\sigma(\alpha+2)}),
		$$
		we obtain
		\[\begin{split}
			\|u-m(u)\|_{\alpha+2}
			& \lesssim \left(\int_{\mathbb{R}^d}\left(\sum_k|k|^{\frac{\alpha}{\alpha+2}}|u_k(x)|^2 dx \right)^{\frac{\alpha+2}{2}}\right)^{\frac{1}{\alpha+2}} \nonumber \\
			& \lesssim \left(\sum_k|k|^{\frac{\alpha}{\alpha+2}}\|u_k\|_{L_x^{\alpha+2}}^2\right)^{\frac{1}{2}}  \\
			& \lesssim \left(\sum_k|k|^{\frac{\alpha}{\alpha+2}}\|u_k\|_{L_x^2}^{2-\frac{\alpha d}{\sigma(\alpha+2)}}\|(-\Delta_x)^{\frac{\sigma}{2}} u_k\|_{L_x^2}^{\frac{\alpha d}{\sigma(\alpha+2)}}\right)^{\frac{1}{2}}   \\
			& \lesssim \|(\|u_k\|_{L_x^2})_k\|_{\ell_k^2}^{\frac{(2\sigma-1)\alpha+4\sigma-\alpha d}{2\sigma(\alpha+2)}} \|(|k|^\sigma\|u_k\|_{L_x^2})_k\|_{\ell_k^2}^{\frac{\alpha}{2\sigma(\alpha+2)}} \|(\|(-\Delta_x)^{\frac{\sigma}{2}} u_k\|_{L_x^2})_k\|_{\ell_k^2}^{\frac{\alpha d}{2\sigma(\alpha+2)}}   \\
			& = \|u_k\|_2^{\frac{(2\sigma-1)\alpha+4\sigma-\alpha d}{2\sigma(\alpha+2)}} \|(-\partial_y^2)^{\frac{\sigma}{2}} u_k\|_2^{\frac{\alpha}{2\sigma(\alpha+2)}} \|(-\Delta_x)^{\frac{\sigma}{2}} u_k\|_2^{\frac{\alpha d}{2\sigma(\alpha+2)}}
		\end{split}\]
		Then \eqref{GN} follows.
	\end{proof}
The following lemma is a direct consequence of Lemma \ref{GN}.
	\begin{lemma}[Finiteness of $m_{c}$]\label{radial cor lower bound}
		For any $c\in(0,\infty)$ we have $m_{c}\in(0,\infty)$.
	\end{lemma}
	
	\begin{proof}
		Let $(u_k)_k \subset V(c)$ be a minimizing sequence for $m_{c}$. From $u_k \in V(c)$ we deduce that
		$$
		\frac{1}{\al+2}\|u_k\|_{\al+2}^{\al+2} = \frac{2\sigma }{ \al d }\|(-\Delta_x)^{\frac{\sigma-1}{2}}\nabla_x u_k\|_{2}^2 \leq \frac{2\sigma }{ \al d }\|(-\Delta_x)^{\frac{\sigma}{2}} u_k\|_{2}^2.
		$$
		Then using $\E(u_k) \to m_{c}$ and the fact that $\al >   \frac{4\sigma }{d}$ it is not difficult to verify that $(u_k)_k$ is bounded in $H_{x,y}^\sigma$. By Lemma \ref{GN} we have
		\begin{align*}
			\|u_k\|_{\al+2}^{\al+2} &\leq C\|(-\Delta_x)^{\frac{\sigma}{2}} u_k\|_{2}^{\frac{ \al d  }{2\sigma }}\|u_k\|_{2}^{\frac{4\sigma - \al (d+1-2\sigma )}{2\sigma }}\left(\|u_k\|_2^{\frac{\al}{2\sigma }} + \|(-\Delta_y)^{\frac{\sigma}{2}} u_k\|_2^{\frac{\al}{2\sigma }}\right) \\
			&\leq C\|(-\Delta)^{\frac{\sigma-1}{2}}\nabla_x u_k\|_{2}^{\frac{ \al d  }{2\sigma }}\|u_k\|_{2}^{\frac{4\sigma - \al (d+1-2\sigma )}{2\sigma }}\left(\|u_k\|_2^{\frac{\al}{2\sigma }} + \|(-\Delta_y)^{\frac{\sigma}{2}} u_k\|_2^{\frac{\al}{2\sigma }}\right).
		\end{align*}
		Then we obtain
		$$
		\|(-\Delta)^{\frac{\sigma-1}{2}}\nabla_x u_k\|_{2}^2 \sim \|u_k\|_{L^{\al+2}}^{\al+2} \lesssim \|(-\Delta)^{\frac{\sigma-1}{2}}\nabla_x u_k\|_{2}^{\frac{ \al d }{2\sigma }},\qquad \frac{ \al d }{2\sigma } > 2,
		$$
		implying that $\liminf_{k \rightarrow \infty}\|u_k\|_{L^{\al+2}}^{\al+2} > 0$. Then $m_{c} > 0$ follows. Hence, $m_{c} \in (0,\infty)$.
	\end{proof}

	Next we aim to show that any critical point of $E$ on $V(c)$ is a solution of \eqref{frac-c} with $\omega$ being a Lagrange multiplier.
	
	\begin{lemma}[Smoothness of $V(c)$] \label{radial manifold}
		The manifold $V(c)$ is a $C^1$ manifold of codimension one in $S(c)$, hence a $C^1$ manifold of codimension $2$ in $H_{x,y}^\sigma$.
	\end{lemma}
	
	\begin{proof}
		It is enough to prove that $\dd_uQ(u)$ and $\dd_u\M(u)$ are linearly dependent. Arguing by contradictory, we assume that there exists $\theta \in \mathbb{R}$ such that
		$$
		2\sigma \int_{\Omega}(-\Delta)^{\frac{\sigma-1}{2}}\nabla_xu (-\Delta)^{\frac{\sigma-1}{2}}\nabla_x\varphi \dd x\dd y - \frac{ \al d }{2}\int_{\Omega}|u|^{\al} u\varphi \dd x\dd y = \theta\int_{\Omega}u\varphi \dd x\dd y,
		$$
		for any $\varphi \in H_{x,y}^\sigma$. This means that $u$ is a solution of
		$$
		2\sigma (-\Delta)^{\sigma-1}(-\Delta_x)u - \theta u = \frac{ \al d }{2}|u|^{\al} u,
		$$
		implying that
		$$
		2\sigma \|(-\Delta)^{\frac{\sigma-1}{2}}\nabla_x u\|_{2}^2 + 2\sigma (\sigma-1)\|(-\Delta)^{\frac{\sigma-2}{2}}(-\Delta_x) u\|_{2}^2 = \frac{ \al ^2d^2}{4(\al+2)}\|u\|_{\alpha+2}^{\al+2}.
		$$
		Since $u \in V(c)$, one gets that
		$$
		\sigma\left(2 - \frac{ \al d }{2}\right)\|(-\Delta)^{\frac{\sigma-1}{2}}\nabla_x u\|_{2}^2 + 2\sigma (\sigma-1)\|(-\Delta)^{\frac{\sigma-2}{2}}(-\Delta_x) u\|_{2}^2 = 0.
		$$
		Noticing that $ \al d  \geq 4$ and $\sigma - 1 < 0$, we find a contradiction.
	\end{proof}
	
	\begin{lemma}[Smoothness of the mapping $u \mapsto t^\ast(u)$] \label{radial smooth of t}
		The mapping $u \mapsto t^\ast(u)$ is of class $C^1$ from $H^{2\sigma +1}(\Omega)$ to $\rr$.
	\end{lemma}
	
	\begin{proof}
This follows from a fundamental application of the implicit function theorem, we omit the details here.
	\end{proof}
	\begin{remark} \label{dense sunmanifold}
\normalfont
		Let $(u_k)_k \subset H^{2\sigma +1}(\Omega)$ be such that $u_k \to u$ in $H_{x,y}^\sigma$ as $k \to \infty$. For any fixed $t > 0$, direct computations yield to the convergence of $u_k^t$ to $u^t$ both in $H_{x,y}^\sigma$ and $L^{\al+2}$. Hence, $u \mapsto t^\ast(u)$ can be extended continuously to $H^{\sigma}(\Omega) \to \rr$. This shows that $V(c) \cap H^{2\sigma +1}(\Omega)$ is a dense submanifold of $V(c)$.
	\end{remark}
	
	\begin{lemma}[Direct decomposition of $T_uS(c)$] \label{radial oplus}
Suppose that $u \in V(c) \cap H^{2\sigma +1}(\Omega)$ satisfies that $u(\cdot,y) \in C_c^\infty(\mathbb{R}^d)$. Then
		$$
		T_uS(c) = T_uV(c)\oplus\mathbb{R}\frac{\dd}{\dd t}\Big|_{t=1}u^t(x,y).
		$$
	\end{lemma}
	
	\begin{proof}
		First observe that for $u \in V(c) \cap H^{2\sigma +1}(\Omega)$ satisfying that $u(\cdot,y) \in C_c^\infty(\mathbb{R}^d)$, $t \mapsto u^t := t^{\frac{d}{2}}u(tx,y)$ is of class $C^1$ from $\mathbb{R}$ to $S(c) \cap H^{2\sigma +1}(\Omega)$ with derivative given by $\frac{d}{2}t^{-1}u^t + t^{-1}(x\cdot \nabla_x)u^t$. Consequently $\frac{\dd}{\dd t}|_{t=1}u^t \in T_uS(c)$ is well-defined. Thus it suffices to show that $\frac{\dd}{\dd t}|_{t=1}u^t \notin T_uV(c)$, i.e. $\dd_uQ(u)\left[\frac{\dd}{\dd t}\big|_{t=1}u^t\right] \not= 0$. We compute directly that
		\[	\begin{split}
			\dd_uQ(u)\left[\frac{\dd}{\dd t}\Big|_{t=1}u^t\right]
			&=  2\sigma \int_{\Omega}(-\Delta)^{\frac{\sigma-1}{2}}\nabla_x u (-\Delta)^{\frac{\sigma-1}{2}}\nabla_x\left(\frac{d}{2}u + (x\cdot\nabla_x)u\right) \dd x\dd y   \\
			&\quad-  \frac{ \al d }{2}\int_{\Omega}|u|^\alpha u\left(\frac{d}{2}u + (x\cdot\nabla_x)u\right) \dd x\dd y   \\
			&=    2\sigma \|(-\Delta)^{\frac{\sigma-1}{2}}\nabla_x u\|_{2}^2 + 2\sigma (\sigma-1)\|(-\Delta)^{\frac{\sigma-2}{2}}(-\Delta_x) u\|_{2}^2 \\&\quad- \frac{ \al ^2 d^2}{4(\al+2)}\int_{\Omega}|u|^{\al+2}\dd x\dd y    \\
			&=  \sigma\left(2 - \frac{ \al d }{2}\right)\|(-\Delta)^{\frac{\sigma-1}{2}}\nabla_x u\|_{2}^2 + 2\sigma (\sigma-1)\|(-\Delta)^{\frac{\sigma-2}{2}}(-\Delta_x) u\|_{2}^2 < 0,
		\end{split}\]
		which completes the proof.
	\end{proof}
	\begin{lemma} \label{radial 000}
		Suppose that $u \in V(c) \cap H^{2\sigma +1}(\Omega)$ satisfies that $u(\cdot,y) \in C_c^\infty(\mathbb{R}^d)$. Then
		$$
		\dd_u\E(u)\left[\frac{\dd}{\dd t}\Big|_{t=1}u^t\right] = 0.
		$$
	\end{lemma}
	
	\begin{proof}
		If $u \in V(c) \cap H^{2\sigma +1}(\Omega)$ satisfies that $u(\cdot,y) \in C_c^\infty(\mathbb{R}^d)$, then $t^\ast(u) = 1$ and
		$$
		0 = \frac{\dd}{\dd t}\Big|_{t=1}\E(u^t) = \dd_u\E(u^t)\left[\frac{\dd}{\dd t}u^t\right]\Big|_{t=1} = \dd_u\E(u)\left[\frac{\dd}{\dd t}\Big|_{t=1}u^t\right].
		$$
	\end{proof}
	
	\begin{proposition}[Existence of a Palais-Smale sequence] \label{radial natural constraint}
The following statements hold true:
		\begin{itemize}
			\item[(i)] If $(u_k)_k \subset V(c) \cap H^{2\sigma +1}(\Omega)$, satisfying $u_k(\cdot,y) \in C_c^\infty(\mathbb{R}^d)$, is a Palais-Smale sequence for $\E$ restricted to $V(c)$ at a certain level $a \in \mathbb{R}$, then $(u_k)_k$ is a Palais-Smale sequence for $\E$ restricted to $S(c)$.
			
			\item[(ii)] If there exists a Palais-Smale sequence $(\tilde{u}_k)_k \subset V(c) \cap H^{2\sigma +1}(\Omega)$ for $\E$ restricted to $V(c)$ at level $a \in \mathbb{R}$, then there exists a possibly different Palais-Smale sequence $(u_k)_k \subset V(c) \cap H^{2\sigma +1}(\Omega)$ satisfying $u_k(\cdot,y) \in C_c^\infty(\mathbb{R}^d)$ for $\E$ restricted to $V(c)$ at the same level $a \in \mathbb{R}$. Moreover, $$\|\tilde{u}_k - u_k\|_{H^{2\sigma +1}(\Omega)} \to 0$$ as $k \to \infty$.
			
			\item[(iii)] If there exists a Palais-Smale sequence $(\tilde{u}_k)_k \subset V(c) \cap H^{2\sigma +1}(\Omega)$ for $\E$ restricted to $V(c)$ at level $a \in \mathbb{R}$, then there exists a possibly different Palais-Smale sequence $(u_k)_k \subset V(c) \cap H^{2\sigma +1}(\Omega)$ satisfying $u_k(\cdot,y) \in C_c^\infty(\mathbb{R}^d)$ for $\E$ restricted to $S(c)$ at the same level $a \in \mathbb{R}$. Moreover, $$\|\tilde{u}_k - u_k\|_{H^{2\sigma +1}(\Omega)} \to 0$$ as $k \to \infty$.
			
			\item[(iv)] Let $u \in H^{2\sigma +1}(\Omega)$ be a critical point of $\E$ restricted on $V(c)$. Then $u$ is a critical point of $\E$ restricted on $S(c)$, and hence a solution to \eqref{frac-c} with $\omega$ being a Lagrange multiplier.
		\end{itemize}
	\end{proposition}
	
	\begin{proof}
		We first prove (i).  Let $(u_k)_k \subset V(c) \cap H^{2\sigma +1}(\Omega)$, satisfying $u_k(\cdot,y) \in C_c^\infty(\mathbb{R}^d)$, be a Palais-Smale sequence for $\E$ restricted to $V(c)$. We denote by $(T_uS(c))^\ast$ the dual space to $T_uS(c)$. In view of Lemma \ref{radial oplus}, we have
		\begin{eqnarray}
			&& \|\dd_u\E(u_k)\|_{(T_uS(c))^\ast} \nonumber \\
			&=& \sup\{|\dd_u\E(u_k)[\varphi]|: \varphi \in T_uS(c), \|\varphi\|_{H^\sigma} \leq 1\} \nonumber \\
			&=& \sup\left\{\abso{\dd_u\E(u_k)[\phi]+\dd_u\E(u_k)[\psi]}:
			\begin{array}{c}
				\varphi = \phi+\psi, \|\varphi\|_{H^\sigma} \leq 1, \\
				\phi \in T_uV(c), \psi \in \mathbb{R}\frac{\dd}{\dd t}\Big|_{t=1}(u_k)^t
			\end{array}
			\right\}.
		\end{eqnarray}
		By Lemma \ref{radial 000}, $\dd_u\E(u_k)[\psi] = 0$, and hence
	\[	\begin{split}
			  \norm{\dd_u\E(u_k)}_{(T_uS(c))^\ast}  
			&= \sup\sett{|\dd_u\E(u_k)[\phi]|: \phi \in T_uV(c), \|\phi\|_{H^\sigma} \leq 1 }  
			 \\
			&= \norm{\dd_u\E(u_k)}_{(T_uV(c))^\ast} \rightarrow 0.
		\end{split}\]
		This proves (i).
		
		Next we show (ii). Let $u \in V(c)\cap H^{2\sigma +1}(\Omega)$. Since $C_c^\infty(\mathbb{R}^d)$ is dense in $H^{2\sigma +1}_x(\mathbb{R}^d)$, there exists $(u_k) \subset S(c) \cap H^{2\sigma +1}(\Omega)$ satisfying $u_k(\cdot,y) \in C_c^\infty(\mathbb{R}^d)$ such that $u_k \to u$ in $H^{2\sigma +1}(\Omega)$. If $u_k \notin V(c)$, then we replace it with $(u_k)^{t^\ast(u_k)}$, for which $(u_k)^{t^\ast(u_k)} \to u$ in $H^{2\sigma +1}(\Omega)$ holds. Hence we may simply assume that $(u_k) \subset V(c)\cap H^{2\sigma +1}(\Omega)$. The proof of (ii) then follows from standard diagonal arguments. Moreover, it is easy to verify that (iii) follows immediately from (i) and (ii).
		
		To prove (iv), take a Palais-Smale sequence $(\tilde{u}_k)_k$ for $\E$ restricted to $V(c) \cap H^{2\sigma +1}(\Omega)$ where $\tilde{u}_k \equiv u$. By (iii), there exists a possibly different Palais-Smale sequence $(u_k)_k$ satisfying $u_k(\cdot,y) \in C_c^\infty(\mathbb{R}^d)$ for $\E$ restricted to $V(c)$ with $\|u_k-u\|_{H_{x,y}^\sigma} \to 0$ as $k \to \infty$. This in turn proves (iv) and consequently completes the desired proof.
	\end{proof}
Having all the preliminaries we are now able to give the main result for the existence of the ground states in this subsection.
	\begin{theorem}[Existence of normalized ground states]\label{thm iso norm masssup}
		Let $  \frac{4}{d} \leq \al < 2_\sigma^\ast, m = 1$ and $\sigma \in (\frac{d+1}{d+2},1)$. Then for any $c\in(0,\infty)$ we have $m_{c}\in(0,\infty)$ and $m_{c}$ has a minimizer.
	\end{theorem}

	
	\begin{proof}
We split the proof into three steps.
\subsubsection*{Step 1: Existence of a bounded minimizing sequence $(u_k)_k$ for $m_{c}$ with $\liminf_{k \to \infty}\|u_k\|_{L^{\al+2}}^{\al+2} > 0$}
Let $\{u_k\}\subset V(c)$ be a minimizing sequence for $m_c$. Using $u \mapsto |u|$ we may assume that $u_k \geq 0$. From the proof of Lemma \ref{radial cor lower bound} we know that $u_k$ is bounded in $H_{x,y}^\sigma$ and $\liminf_{k \to \infty}\|u_k\|_{L^{\al+2}}^{\al+2} > 0$. By Ekland variational principle and Proposition \ref{radial natural constraint} we assume that $\E'(u_k) + \omega_k u_k \to 0$ for some $\omega_k$. Up to a subsequence, we may find $0 \leq u \in H^\sigma_{x,y}\backslash\{0\}$ such that $u_k(\cdot + z_k,\cdot)$ converges to $u $ weakly in $H_{x,y}^\sigma$ for some $\{z_k\} \subset \mathbb{R}_x^d$.
		
		\subsubsection*{Step 2: $u$ is a solution of \eqref{eq of u} with $\omega > 0$}
		
	    From the boundedness of $u_k$ one gets that $\omega_k$ is bounded. Up to a subsequence, let $\omega_k \to \omega$. Then it can be verified that $\E'(u) - \omega u = 0$, i.e. $u$ is a solution of the following equation
		\begin{align} \label{eq of u}
			(-\Delta)^{\sigma}u + \omega u = |u|^{\al}u.
		\end{align}
        We now show that $\omega > 0$.  
By maximum principle we may assume that $u$ is positive. We will indeed show that \eqref{eq of u} has no positive solution when $\omega \leq 0$, which completes the proof of Step 2.
			
First consider the case when $d\al<2\sigma (\al+1)$. We argue by contradiction and assume that such a positive solution $u$ exists.
			Since $\omega\leq0$, we deduce from \eqref{eq of u} that
			\begin{equation} \label{Liou-1}
				(-\Delta)^{\sigma}u \geq |u|^{\al}u.
			\end{equation}
Next, let $\psi$ be a nod-negative smooth function in $\R^d\times\T^m$.
			If we multiply \eqref{Liou-1} by $\psi$ and integrate over the region $\R^d\times\T^m$, then applying integration by parts leads  to
			\[
			\scal{|u|^{\al}u,\psi}
			\leq\scal{ u^{\al+1},\psi} ^{\frac{1}{\al+1}} \scal{ \psi^{-\frac{1}{\al}},((-\Delta)^{\sigma}\psi)^{\frac{\al+1}{\al}}} ^{1-\frac{1}{\al+1}}.
			\]
If  the last term in the right-hand side of the above inequality is bounded,    we obtain by a similar calculation for $\psi_R(x,y)=\psi_0(R^{-1}x)$ with $R>0$ that
			\begin{equation}\label{fr-1}
				\begin{split}
					\int_{|x|\leq R}
					|u|^{\al}u \dd x
					&\lesssim
					\scal {|u|^{\al}u,\psi_R(x,y)}\\ &\lesssim R^{d -\frac{2\alpha\sigma  }{\al+1}}\int_{\Omega}\psi_0^{-\frac{1}{\al}} ((-\Delta)^{\sigma}\psi_0)^{\frac{\al+1}{\al}}\dd x\dd y\\&
					\cong
					R^{d -\frac{2\alpha\sigma  }{\al+1}}\int_{\rn}\psi_0^{-\frac{1}{\al}} ((-\Delta)_x^{\sigma}\psi_0)^{\frac{\al+1}{\al}}\dd x
				\end{split}
			\end{equation}
 	Hence, we get from the assumption on $\alpha$   when $R\to\infty$ that
			$u\equiv0$.
 To complete the proof, we choose   $\psi_0(x)=(1+|x|)^{-2\sigma -d} $ and show that the right hand side of \eqref{fr-1} is bounded. To see this, we recall from \cite[Corollary 3.1]{DAbbicco2021} that
			\[
			|(-\Delta)^{\sigma}_x\psi_0(x)|\lesssim
			\psi_0(x).
			\]
Note however that there is still a gap when $d\al\geq2\sigma (\al+1)$,
%
which appears in the case $d \geq 3$. This gap can nevertheless be covered by using the regularity of $u$ deduced from Proposition \ref{regularity} and modifying the proof of \cite[Theorem 4 (i)]{Chen2017}, we omit the details of   straightforward modification. This completes the proof of Step 2.
		
		\subsubsection*{Step 3: Conclusion}
		
		By weakly lower semi-continuity of norms we deduce that
		$$
		\M(u) := c_1 \in (0,c].
		$$
		Let $\tilde{u}_k = u_k(\cdot+z_k,\cdot)$. The Brezis-Lieb lemma implies that
		$$
		Q(\tilde{u}_k) = Q(\tilde{u}_k-u) + Q(u) + o_k(1),
		$$
		$$
		\E(\tilde{u}_k) = \E(\tilde{u}_k-u) + \E(u) + o_k(1).
		$$
		Since $u_k \in V(c)$, one knows $Q(\tilde{u}_k) = Q(u_k) = 0$. Since $u$ is a solution of \eqref{eq of u}, we know that $Q(u) = 0$. Hence, $Q(\tilde{u}_k-u) = o_k(1)$.

First, we consider the case that $\|\tilde{u}_k-u\|_{\al+2}^{\al+2} = o_k(1)$ passing to a subsequence if necessary. From Step 1 and 2 we know that $(u_k)_k$ is a bounded sequence in $H_{x,y}^\sigma$ satisfying
		$$
		\mathbb{A}_\omega'(\tilde{u}_k) \to 0.
		$$
		Using the fact that $\mathbb{A}_\omega(u) = 0$ and strong convergence of $\tilde{u}_k$ in $L^{\al+2}$ we deduce that
		\begin{align}
			\lim_{k \to \infty}\left(\|(-\Delta)^{\frac \sigma 2}(\tilde{u}_k-u)\|_{2}^2 + \omega\|\tilde{u}_k-u\|_{2}^2\right)  = \lim_{k \to \infty} \left\langle \mathbb{A}'_\omega(\tilde{u}_k)-\mathbb{A}'_\omega(u),\tilde{u}_k-u\right\rangle = 0.
		\end{align}
		This shows that $\tilde{u}_k$ converges to $u$ strongly in $H_{x,y}^\sigma$. Hence $\E(u) = m_{c}$ and $\M(u) = c$.
		
		Next we consider the case that $\liminf_{k \to \infty}\|\tilde{u}_k-u\|_{\al+2}^{\al+2} > 0$. Hence, $$\liminf_{k \to \infty}\norm{  (-\Delta)^\frac {\sigma}2 (\tilde{u}_k-u)}_{2} \geq \liminf_{k \to \infty}\norm{  (-\Delta)^\frac {\sigma-1}2 \nabla_x(\tilde{u}_k-u)}_{2} > 0.$$ Using $Q(\tilde{u}_k-u) = o_k(1)$ we obtain
		\begin{align}
		\E(\tilde{u}_k-u) = & \frac12\norm{  (-\Delta)^\frac {\sigma}2 (\tilde{u}_k-u)}_{2} - \frac{2\sigma}{\alpha d}\norm{  (-\Delta)^\frac {\sigma-1}2 \nabla_x(\tilde{u}_k-u)}_{2} + o_k(1) \nonumber \\
		\geq & \frac{1}2(1-\frac{4\sigma}{\alpha d})\norm{  (-\Delta)^\frac {\sigma-1}2 \nabla_x(\tilde{u}_k-u)}_{2} + o_k(1) \nonumber \\
		> & 0
		\end{align}
		for $k$ large enough. Then together with the fact that $u \in V(c_1)$ and Lemma \ref{radial monotone lemma}, one gets, for $k$ large enough,
		\begin{align}
		m_c =& \E(\tilde{u}_k) +o_k(1) \nonumber \\
		=& \E(\tilde{u}_k-u) + \E(u) + o_k(1) \nonumber \\
		>& m_(c_1) \nonumber \\
		\geq& m_c,
		\end{align}
	    which is self-contradictory. The proof is complete.
	\end{proof}

Finally, we prove the following $y$-dependence result corresponding to Theorem \ref{thm iso norm masssup}.

\begin{theorem}[$y$-dependence of the normalized ground states] \label{thm y-dependence intercritical case}	
		Let $\frac{4}{d} \leq \al < 2_\sigma^\ast, m = 1$ and $\sigma \in (\frac{d+1}{d+2},1)$, $u_c$ be the minimizer of $m_c$ given by Theorem \ref{thm iso norm masssup}. Then there exists $c^*\in(0,\infty)$ such that
		\begin{itemize}
			\item For all $c\in[c^*,\infty)$ we have $m_{c}=2\pi\tilde{m}_{(2\pi)^{-1}c}$. Moreover, for all $c\in(c^*,\infty)$ we have $\partial_y u_c=0$.
			
			\item For all $c\in(0,c^*)$ we have $m_{c}<2\pi\tilde{m}_{(2\pi)^{-1}c}$. Moreover, for all $c\in(0,c^*)$ we have $\partial_y u_c\neq 0$.
		\end{itemize}
	\end{theorem}

To proceed, we first define the quantities
\begin{align*}
		\nll & = \pax-\lambda\partial_y^2, \\
		\E_\lambda(u) & = \frac12\|\nll^{\frac{\sigma}{2}}u\|_2^2 - \frac{1}{\al+2}\|u\|_{\alpha+2}^{\alpha+2}, \\
		Q_\lambda(u) & = \sigma\|\nll^{\frac{\sigma-1}{2}}\nabla_xu\|_2^2 - \frac{\al d}{2(\al+2)}\|u\|_{\alpha+2}^{\alpha+2}, \\
		V_{1,\lambda} & = \{u \in S(1): Q_\lambda(u) = 0\}, \\
		m_{1,\lambda} & = \inf\sett{\E_\lambda(u), u \in V_{1,\lambda}}.
	\end{align*}
In order to prove Theorem \ref{thm y-dependence intercritical case} we will first consider the auxiliary problem $m_{1,\lambda}$. Proceeding as in the proof of Theorem \ref{thm iso norm masssup} we know that the variational problem $m_{1,\lambda}$ has an optimizer $u_\lambda \in V_{1,\lambda}$ for any $\lambda \in (0,\infty)$. Following the similar lines in \cite{TTVproduct2014}, we prove the following characterization of $m_{1,\lambda}$ for varying $\lambda$. As we shall see, however, the isotropic nature of the underlying model will nevertheless raise some additionally technical difficulties.

Our first step is to prove the following crucial $y$-dependence result for the variational problem $m_{1,\lambda}$.

	\begin{lemma} \label{characterization}
		Let $\frac{4}{d} \leq \al < 2_\sigma^\ast, m = 1$ and $\sigma \in (\frac{d+1}{d+2},1)$, $u_\lambda$ be the minimizer of $m_{1,\lambda}$. Then there exists $\lambda^*\in(0,\infty)$ such that
		\begin{itemize}
			\item For all $\lambda\in[\lambda^*,\infty)$ we have $m_{1,\lambda}=2\pi\tilde{m}_{(2\pi)^{-1}}$. Moreover, for all $\lambda\in(\lambda^*,\infty)$ we have $\partial_y u_\lambda=0$.
			
			\item For all $\lambda\in(0,\lambda^*)$ we have $m_{1,\lambda}<2\pi\tilde{m}_{(2\pi)^{-1}}$. Moreover, for all $\lambda\in(0,\lambda^*)$ we have $\partial_y u_\lambda\neq 0$.
		\end{itemize}
	\end{lemma}

	In order to prove Lemma \ref{characterization}, we firstly collect some useful auxiliary lemmas.
	
	\begin{lemma} \label{characterization 1}
		Let $\frac{4}{d} \leq \al < 2_\sigma^\ast, m = 1$ and $\sigma \in (\frac{d+1}{d+2},1)$. We have
		\begin{align} \label{eq of equal}
			\lim_{\lambda \to \infty}m_{1,\lambda}=2\pi\tilde{m}_{(2\pi)^{-1}}.
		\end{align}
		Additionally, let $u_\lambda \in V_{1,\lambda}$ be a positive optimizer of $m_{1,\lambda}$ which also satisfies
		\begin{align} \label{equa of u}
			\nll^\sigma u_\lambda + \omega_\lambda u_\lambda = |u_\lambda|^{\al}u_\lambda \ \text{on} \ \R^d \times \T
		\end{align}
		for some $\omega_\lambda$. Then
		\begin{align} \label{stronger limit}
			\lim_{\lambda \to \infty}\lambda \norm{(-\partial_y^2)^{\frac{\sigma}{2}}u_\lambda} _2^2 = 0.
		\end{align}
	\end{lemma}
	
	\begin{proof}
		By assuming that a candidate in $V_{1,\lambda}$ is independent of $y$ we already conclude
		\begin{align} \label{leq}
			m_{1,\lambda} \leq 2\pi\tilde{m}_{(2\pi)^{-1}}.
		\end{align}
		Next we prove
		\begin{align} \label{tend to 0}
			\lim_{\lambda \to \infty}\|(-\partial_y^2)^{\frac{\sigma}{2}}u_\lambda\|_2^2 = 0.
		\end{align}
		Suppose that \eqref{tend to 0} does not hold. Then there exists $\lambda_n \to \infty$ such that
		\begin{align*}
			\lim_{n \to \infty}\lambda_n\norm{ (-\partial_y^2)^{\frac{\sigma}{2}}u_{\lambda_n}} _2^2 = \infty.
		\end{align*}
		Then using the Minkowski's inequality we know that $\|\mathscr{L}_{\lambda_n}^\frac \sigma 2u_{\lambda_n}\|_{2}$ blows up as $n\to\infty$. On the other hand, since $Q_\lambda(u_\lambda) = 0$ and $\al \geq \frac{4}{d} > \frac{4\sigma}{d}$,
		\begin{align*}
			m_{1,\lambda} & = \E_\lambda(u_\lambda) - \frac{2}{\al d}Q_\lambda(u_\lambda) \nonumber \\
			& = \frac12\norm{\nll^{\frac{\sigma}{2}}u_\lambda}_2^2 - \frac{2\sigma}{\al d}\norm{\nll^{\frac{\sigma-1}{2}}\nabla_xu_\lambda}_2^2 \nonumber \\
			& \geq \left( \frac12 - \frac{2\sigma}{\al d}\right)  \norm{\nll^{\frac{\sigma}{2}}u_\lambda} _2^2,
		\end{align*}
		using \eqref{leq} one gets $\|\nll^{\frac{\sigma}{2}}u_\lambda\|_2^2$ is bounded, contradicting that $\|\mathscr{L}_{\lambda_n}^\frac \sigma 2u_{\lambda_n}\|_{2}$ blows up as $n\to\infty$. Therefore $(u_\lambda)_\lambda$ is a bounded sequence in $H^\sigma_{x,y}$ as $\lambda \to \infty$, whose weak limit is denoted by $u$ up to a subsequence. Up to translations, we may also assume that $u \not= 0$. By \eqref{tend to 0} we know that $u$ is independent of $y$ and thus $u \in H^\sigma_x(\R^d)$. Moreover, using weakly lower semicontinuity of norms we know that $\|u\|_{L^2(\mathbb{R}^d)}^2 \in (0,(2\pi)^{-1}]$. On the other hand, using $Q_\lambda(u_\lambda) = 0$, \eqref{equa of u} and $\M(u_\lambda) = 1$ we obtain
		\begin{align*}
			\omega_\lambda & = \norm{u_\lambda}_{\al+2}^{\al+2} -  \norm{\nll^{\frac{\sigma}{2}}u_\lambda}_2^2 \nonumber \\
			& \leq \|u_\lambda\|_{\al+2}^{\al+2} -  \norm{\nll^{\frac{\sigma-1}{2}}\nabla_xu_\lambda} _2^2 \nonumber \\
			& = \left( 1-\frac{\al d}{2\sigma(\al +2)}\right)  \norm{u_\lambda} _{\al+2}^{\al+2}.
		\end{align*}
		Thus $(\omega_\lambda)_\lambda$ is a bounded sequence in $(0, \infty)$ as $\lambda \to \infty$, whose limit is denoted by $\omega_\infty$ up to a subsequence. Then we can verify that $u$ satisfies
		\begin{align*}
			(-\Delta_x)^{\sigma}u + \omega_\infty u = |u|^\al u \ \text{in} \ \R_x^d,
		\end{align*}
		implying that $\tilde{Q}(u) = 0$ and $\omega_\infty > 0$. Using techniques in the proof of Theorem \ref{thm iso norm masssup} (Step 3), one gets $u_\lambda \to u$ strongly in $H_{x,y}^\sigma$. From \eqref{leq} and $u \in \tilde{V}$ one deduces that $2\pi\tilde{m}_{(2\pi)^{-1}} \leq 2\pi\tilde{\E}(u) = \E(u) \leq 2\pi\tilde{m}_{(2\pi)^{-1}}$. Thus $u$ is an optimizer of $\tilde{m}_{(2\pi)^{-1}}$. Finally, we conclude that
		\begin{align*}
			m_{1,\lambda} & = \E_\lambda(u_\lambda) - \frac{2}{\al d}Q_\lambda(u_\lambda) = \frac12 \norm{\nll^{\frac{\sigma}{2}}u_\lambda} _2^2 - \frac{2\sigma}{\al d}\norm{\nll^{\frac{\sigma-1}{2}}\nabla_xu_\lambda}_2^2\nonumber \\
			& \geq \left( \frac12 - \frac{2\sigma}{\al d}\right) \norm{\nll^{\frac{\sigma}{2}}u_\lambda}_2^2 \geq \left( \frac12 - \frac{2\sigma}{\al d}\right)\left(\norm{(-\Delta_x)^{\frac{\sigma}{2}}u_\lambda}_2^2 + \lambda \norm{(-\partial_y^2)^{\frac{\sigma}{2}}u_\lambda} _2^2\right)\nonumber \\
			& \geq 2\pi\left( \frac12 - \frac{2\sigma}{\al d}\right) \norm{(-\Delta_x)^{\frac{\sigma}{2}}u}_{L^2_x(\R^d)}^2 + o_\lambda(1) \nonumber \\
			& = 2\pi\tilde{\E}(u) + o_\lambda(1) = 2\pi\tilde{m}_{(2\pi)^{-1}} + o_\lambda(1).
		\end{align*}
		Letting $\lambda \to \infty$ and taking \eqref{leq} into account we conclude \eqref{eq of equal}. Finally, \eqref{stronger limit} follows directly from the previous calculation by not neglecting $\lambda \|(-\partial_y^2)^{\frac{\sigma}{2}}u_\lambda\|_2^2$ therein. This completes the desired proof.
	\end{proof}

\begin{lemma} \label{characterization 2}
		Let $u_\lambda$ and $u$ be the functions given in the proof of Lemma \ref{characterization 1}. Then $u_\lambda \to u$ strongly in $H_{x,y}^\sigma$.
	\end{lemma}
	
	\begin{proof}
		This has been proved in the proof of Lemma \ref{characterization 1}.
	\end{proof}

\begin{lemma} \label{characterization 3}
		Let $\frac{4}{d} \leq \al < 2_\sigma^\ast, m = 1$ and $\sigma \in (\frac{d+1}{d+2},1)$. Then there exists some $\lambda_0$ such that $\partial_yu_\lambda = 0$ and thus $m_{1,\lambda}=2\pi\tilde{m}_{(2\pi)^{-1}}$ for all $\lambda > \lambda_0$.
	\end{lemma}
	
	\begin{proof}
		This can be proved similarly by using the arguments applied in the proofs of Lemma \ref{characterization 1}, \ref{characterization 2} and in Step 3 of the proof of Lemma \ref{small-omega}, we thus omit the details here.
	\end{proof}

\begin{lemma} \label{characterization 4}
		Let $\frac{4}{d} \leq \al < 2_\sigma^\ast, m = 1$ and $\sigma \in (\frac{d+1}{d+2},1)$. Then there exists some $\lambda_1$ such that $m_{1,\lambda}<2\pi\tilde{m}_{(2\pi)^{-1}}$ and thus $\partial_yu_\lambda \not= 0$  for all $\lambda < \lambda_1$.
	\end{lemma}
	
	\begin{proof}
		Let $\rho \in H_y^\sigma$ (see the proof of \cite[Lemma 3.1]{Luo_inter} for a construction) be such that
		$$
		2\pi > \|\rho\|_{L_y^2}^2 = \|\rho\|_{L_y^{\al+2}}^{\al+2}.
		$$
		Next, let  $P \in H^\sigma_x$ be an optimizer of $\tilde{m}_{\|\rho\|_{L_y^2}^{-2}}$. Since the mapping $c \mapsto \tilde{m}_c$ is strictly decreasing on $(0,\infty)$ we infer that $\tilde{m}_{\|\rho\|_{L_y^2}^{-2}} < \tilde{m}_{(2\pi)^{-1}}$. Furthermore, note that $\tilde{Q}(P) = 0$, i.e. $\sigma\|(-\Delta_x)^{\frac{\sigma}{2}}P\|_{L_x^2}^2 = \frac{\al d}{2(\al +2)}\|P\|_{L_x^{\al+2}}^{\al +2}$.
		Now define $\psi(x,y) =\rho(y)P(x) \in H_{x,y}^\sigma$. Then $\M(\psi) = 1$. Moreover, we can take a unique $t_\lambda > 0$ such that $Q_\lambda(\psi^{t_\lambda}) = 0$. From $Q_\lambda(\psi) \to Q_0(\psi) = \|\rho\|_{L_y^2}^2\tilde{Q}(P) = 0$ and that $\|\psi\|_{\al+2}^{\al+2} > 0$, one gets that $t_\lambda \to 1$. Hence,
		\begin{align}
			\lim_{\lambda \to 0}m_{1,\lambda} \leq \lim_{\lambda \to 0}\E_\lambda(\psi^{t_\lambda}) = E_0(\psi) < 2\pi \tilde{m}_{(2\pi)^{-1}}.
		\end{align}
		The proof is complete.
	\end{proof}

We are now ready to give the proof of Lemma \ref{characterization}.

\begin{proof}[Proof of Lemma \ref{characterization}]
		Define
		$$\lambda^\ast:=\inf\sett{\lambda>0:\;m_{1,k}=2\pi\tilde{m}_{(2\pi)^{-1}}  \;\text{for}\;k \geq \lambda }.$$
		It follows from Lemmas\ref{characterization 3} and \ref{characterization 4} that $0<\lambda^\ast<\infty$. That $m_{1,\lambda}<2\pi\tilde{m}_{(2\pi)^{-1}}$ (which also implies that a minimizer of $m_{1,\lambda}$ has non-trivial $y$-dependence) for $\lambda < \lambda^\ast$ follows already from the definition of $\lambda^\ast$ and the fact that $\lambda \mapsto m_{1,\lambda}$ is monotone increasing. That $m_{1,\lambda}=2\pi\tilde{m}_{(2\pi)^{-1}}$ holds for all $\lambda\in(\lambda^*,\infty)$ following the definition of $\lambda^\ast$. From the continuity of the mapping $\lambda \mapsto m_{1,\lambda}$ we also know that $m_{1,\lambda^\ast}=2\pi\tilde{m}_{(2\pi)^{-1}}$.

It thus remains to prove that $\partial_y u_\lambda=0$ for all $\lambda\in(\lambda^*,\infty)$. Suppose on the contrary that there exists a minimizer $u_{\tilde{\lambda}}$ satisfying $\partial_y u_{\tilde{\lambda}}\not=0$ for some $\tilde{\lambda}\in(\lambda^*,\infty)$. Take $\lambda \in (\lambda^\ast,\tilde{\lambda})$. We can find a unique $t_\lambda > 0$ such that $Q_\lambda(u_{\tilde{\lambda}}^{t_\lambda}) = 0$. Since $\partial_y u_{\tilde{\lambda}}\not=0$, we have $t_\lambda \neq 1$, $Q_{\tilde{\lambda}}(u_{\tilde{\lambda}}^{t_\lambda}) \neq 0$, and $\E_{\tilde{\lambda}}(u_{\tilde{\lambda}}^{t_\lambda}) < \E_{\tilde{\lambda}}(u_{\tilde{\lambda}})$. We obtain a self-contradictory inequality
		\begin{align}
			2\pi\tilde{m}_{(2\pi)^{-1}} = m_{1,\lambda} \leq \E_\lambda(u_{\tilde{\lambda}}^{t_\lambda})  < \E_{\tilde{\lambda}}(u_{\tilde{\lambda}}^{t_\lambda}) < \E_{\tilde{\lambda}}(u_{\tilde{\lambda}})  = m_{1,\tilde{\lambda}} = 2\pi\tilde{m}_{(2\pi)^{-1}}.
		\end{align}
		Thus we complete the proof.
	\end{proof}

We are finally in a position to prove Theorem \ref{thm y-dependence intercritical case}

\begin{proof}[Proof of Theorem \ref{thm y-dependence intercritical case}]
		First define the following scaling operator:
		$$
		T_cu = c^ku\left( c^lx,y\right),
		$$
		where $k = \frac{2\sigma}{\alpha d-4\sigma} > 0, l = \frac{\alpha}{\alpha d-4\sigma} > 0$. Direct calculation shows that $Q_{c^{2l}}(T_cu) = c^{\frac{2\sigma(\al+2)-\al d}{\al d - 4\sigma}}Q(u)$, thus $T_c$ is a bijection between $V(c)$ and $V_{1,c^{2l}}$. By a similar scaling argument we deduce that $\E_{c^{2l}}(T_cu) = c^{\frac{2\sigma(\al+2)-\al d}{\al d - 4\sigma}}\E(u)$ and consequently $m_{1,c^{2l}} = c^{\frac{2\sigma(\al+2)-\al d}{\al d - 4\sigma}}m_c$, from which we conclude that the statement $u_c$ is a minimizer for $m_c$ is equivalent to $T_cu_c$ is a minimizer for $m_{1,c^{2l}}$. By same rescaling arguments we also infer that $\tilde{m}_{(2\pi)^{-1}} = c^{\frac{2\sigma(\al+2)-\al d}{\al d - 4\sigma}}\tilde{m}_{(2\pi)^{-1}c}$. Notice also that the mapping $c \mapsto c^{2l}$ is strictly monotone increasing on $(0,\infty)$ since $l > 0$. The desired claim follows then from Lemma \ref{characterization}.
	\end{proof}


\section{Anisotropic case}\label{Sec4}
 In this section, we establish the large data scattering for the anisotropic FNLS \eqref{nls} when $m=1$. First we recall the following quantities such as the energy and mass etc. that shall be re-defined in the anisotropic case:
\begin{align}
		\mM(u)&:=\|u\|^2_{2},\label{def of mass}\\
		\mH(u)&:=\frac{1}{2}\norm{(-\Delta)^{\frac{\sigma}{2}}u}_2^2-\frac{1}{\alpha+2}\|u\|_{\alpha+2}^{\alpha+2},\label{def of mhu}\\
		\mK(u)&:=\sigma\norm{(-\Delta_{x})^{\frac{\sigma}{2}}u}_2^2-\frac{\alpha d}{2(\alpha+2)}\|u\|_{\alpha+2}^{\alpha+2},\label{def of Q}\\
		\mI(u)&:=\frac{1}{2}\norm{(-\pt_y^2)^\frac{\sigma}{2} u}_{2}^2+\bg(\frac{\alpha d}{4\sigma}-1\bg)\frac{1}{\alpha+2}\|u\|^{\alpha+2}_{\alpha+2}=\mH(u)-\frac{1}{2\sigma}
		\mK(u).\label{def of mI}
	\end{align}
To formulate the main result, we also define
	\begin{equation}\label{supercr-normalized}
		\begin{split}
			&	S(c):=\sett{u\in H_{x,y}^\sigma:M(u)=c},\\&
			V(c):=\{u\in S(c):\mK(u)=c\},\\&
			m_c:=\inf\{E(u):u\in V_c\}.
		\end{split}
	\end{equation}
 Our main results in this section are given as follows:
%


 \begin{theorem}[Existence of normalized ground states]\label{thm norm ground}
    	Let $\sigma \in (\frac12,1)$ and $\alpha \in (\frac{4\sigma}{d},2_\sigma^\ast)$. Then for any $c\in(0,\infty)$ we have $m_c\in(0,\infty)$ and $m_c$ has a minimizer.
 \end{theorem}

	\begin{theorem}[$y$-dependence of the normalized ground states] \label{thm y-dependence anianisotropic intercritical case}	
		Let $\sigma \in (\frac12,1)$, $\alpha \in (\frac{4\sigma}{d},2_\sigma^\ast)$, $u_c$ be the minimizer of $m_c$ given by Theorem \ref{thm norm ground}. Then there exists $c^*\in(0,\infty)$ such that
		\begin{itemize}
			\item For all $c\in[c^*,\infty)$ we have $m_{c}=2\pi \tilde{m}_{(2\pi)^{-1}c}$. Moreover, for all $c\in(c^*,\infty)$ we have $\partial_y u_c=0$.
			
			\item For all $c\in(0,c^*)$ we have $m_{c}<2\pi \tilde{m}_{(2\pi)^{-1}c}$. Moreover, for all $c\in(0,c^*)$ we have $\partial_y u_c\neq 0$.
		\end{itemize}
	\end{theorem}

	\begin{theorem}[Large data scattering]\label{thm large data scattering}
		Assume $d\in\{2,3,4\}$, $\sigma\in [\frac{d}{2d-1},1)$ and
		\[
		\alpha\in\left(\max\sett{1,\frac{4\sigma}{d}},\frac{4\sigma}{d+1-2\sigma}\right).
		\]
		 Let $u_0\in S(c)$ such that $u_0(\cdot,y)$ is radial with respect to $x\in\rr^d$ for a.e. $y\in\T$. If
		\begin{align}
			E(u_0)<m_c\quad\text{and}\quad \mK(u_0)>0,
		\end{align}
		then the solution $u$ of \eqref{nls} with $u(0)=u_0$ is global and scatters in time.
	\end{theorem}

As the proofs of Theorems \ref{thm norm ground} and \ref{thm y-dependence anianisotropic intercritical case} are almost identical to the ones of Theorems \ref{thm iso norm masssup} and \ref{thm y-dependence intercritical case}, we will omit their proofs and focus on the one of Theorem \ref{thm large data scattering} in the rest of the paper.

\begin{remark}
\normalfont
By direct calculation one sees that $\sigma \leq \frac12$ implies $2_\sigma^\ast = \frac{4\sigma}{d+m-2\sigma} \leq \frac{4\sigma}d$, which also in turn implies that the interval $(\frac{4\sigma}{d},2_\sigma^\ast)$ is empty. For this reason we shall simply assume that $\sigma \in (\frac12,1)$ in order to deduce reasonable existence results of the ground state solutions.
\end{remark}

\begin{remark}
\normalfont
In comparison to Theorem \ref{thm iso norm masssup} we see that Theorem \ref{thm norm ground} holds for a wider range of the exponents. This is due to the following reason: In the anisotropic case, we indeed have
		$$
		\norm{(-\Delta)^{\frac{\sigma}{2}}u}_2^2 = \norm{(-\Delta_x)^{\frac{\sigma}{2}}u}_2^2 + \norm{(-\pt_y^2)^\frac{\sigma}{2} u}_{2}^2.
		$$
		Thus $\norm{(-\Delta)^{\frac{\sigma}{2}}u^t}_2^2$ can be written as a form like $t^k\norm{(-\Delta)^{\frac{\sigma}{2}}u}_2^2$ for some $k$, while this can not be done in the isotropic case (see Lemma \ref{unique t}). For this reason, the proof of Theorem \ref{thm norm ground} is in fact much easier and in this case we can even establish a result for exponents lying a much wider range.
\end{remark}


\begin{remark}
\normalfont
We note that a small data result will indeed hold for the exponents satisfying a wider range, see Lemma \ref{lemma small data}.
\end{remark}

	In the rest of the paper,  we will always assume that the underlying functions are radially symmetric with respect to $x$-variable. 	
	\subsection{Useful inequalities}
	We state in this subsection some useful inequalities. Before we turn to the main part of this subsection, we first introduce the concept of an \textit{admissible} pair on $\R^d$. A pair $(q,r)$ is said to be $\dot{H}^s$-admissible for $s\in[0,\frac d2)$ if $q,r\in[2,\infty]$, $\frac{2\sigma}{q}+\frac{d}{r}=\frac{d}{2}-s$. For $d\geq 2$, we define the space $S_x$ by
	\begin{align}
		S_x:=L_t^\infty L_x^2\cap L^{2^+}_t L_x^{q}\label{def of sx},
	\end{align}
	where $(2^{+},q)$ is an $L^2$-admissible pair with some $2^+\in(2,\infty)$ sufficiently close to $2$.

We now state the main auxiliary lemmas.
	
	\begin{lemma}[Fractional Strichartz estimates on $\R^d$, \cite{Frac1waveguide,Frac2waveguide}]\label{frac strichartz lem}
		Assume $d\geq 2$ and $\sigma\in[\frac{d}{2d-1},1)$. Assume also that there exist $p,q,\tilde{p},\tilde{q}\in(2,\infty)$ and $s\in\R$ such that
		\begin{gather}
			\tilde{p}'<p,\quad\frac{1}{p}\leq \frac{2d-1}{2}\bg(\frac12-\frac1q\bg),\label{2.3}\\
			(p,q)\neq \bg(2,\frac{4d-2}{2d-3}\bg),\label{2.4}\\
			\frac{2\sigma}{p}+\frac{d}{q}=\frac{d}{2}-s,\label{2.5}\\
			\frac{2\sigma}{p}+\frac{d}{q}+\frac{2\sigma}{\tilde{p}}+\frac{d}{\tilde{q}}=\frac{d}{2}.\label{2.6}
		\end{gather}
		Then for any $I\subset\R$ we have
		\begin{align}
			\norm{\ee^{-\ii   t(-\Delta_{x})^{\sigma}}u_0}_{L_t^p L_x^q (I)}&\lesssim \|u_0\|_{H_x^{s}},\\
			\norm{\int_{t_0}^t \ee^{-\ii (t-s)(-\Delta_{x})^{\sigma}}F(s)\dd s}_{L_t^p L_x^q (I)}&\lesssim \|F\|_{L_t^{\tilde{p}'} L_x^{\tilde{q}'} (I)}.
		\end{align}
	\end{lemma}
	
	Together with \cite[Proposition 2.1]{TzvetkovVisciglia2016}, Lemma \ref{frac strichartz lem} implies the following fractional Strichartz estimates on $\Omega=\R^d\times\T$.
	
	\begin{lemma}[Fractional Strichartz estimates on $\Omega$]\label{strichartz lem}
		Assume $d\geq 2$ and $\sigma\in[\frac{d}{2d-1},1)$. Assume also that there exist $p,q,\tilde{p},\tilde{q}\in(2,\infty)$ and $s\in[0,\frac{d}{2})$ such that
		\eqref{2.3}-\eqref{2.6} hold. Then for any $\gamma\in\R$ and $I\subset\R$ we have
		\begin{align}
			\norm{\ee^{-\ii   t(-\Delta)^{\sigma}}u_0}_{L_t^p L_x^q H^\gamma_y(I)}&\lesssim \|u_0\|_{H_x^{s}H_y^\gamma},\\
			\norm{\int_{t_0}^t \ee^{-\ii (t-s)(-\Delta)^{\sigma}}F(s)\dd s}_{L_t^p L_x^q H^\gamma_y(I)}&\lesssim \|F\|_{L_t^{\tilde{p}'} L_x^{\tilde{q}'} H^\gamma_y(I)}.
		\end{align}
	\end{lemma}
	In the following lemma, we establish the existence of some ``friendly'' admissible pairs which are ``computationally friendly''.
	
	\begin{lemma}[``Friendly'' admissible pairs]\label{friendly pairs}
		Suppose that $d\geq 2$, $\sigma\in [\frac{d}{2d-1},1)$ and $\alpha\in(\frac{4\sigma}{d},\frac{4\sigma}{d+1-2\sigma})$. Then there exist $p,q,\tilde{p},\tilde{q}\in(2,\infty)$ such that
		\begin{gather}
			\tilde{p}'<p,\quad\frac{1}{p}\leq \frac{2d-1}{2}\bg(\frac12-\frac1q\bg),\quad\frac{\alpha}{q}<\frac{2\sigma}{d},\label{2.3+}\\
			(p,q)\neq (2,\frac{4d-2}{2d-3}),\label{2.4+}\\
			\frac{2\sigma}{p}+\frac{d}{q}=\frac{2\sigma}{\alpha},\label{2.5+}\\
			\frac{2\sigma}{p}+\frac{d}{q}+\frac{2\sigma}{\tilde{p}}+\frac{d}{\tilde{q}}=\frac{d}{2},\label{2.6+}\\
			\frac{1}{\tilde{p}'}=\frac{\alpha+1}{p},\quad \frac{1}{\tilde{q}'}=\frac{\alpha+1}{q}.\label{2.7+}
		\end{gather}
		Moreover, $q$ can be chosen arbitrarily close to (and strictly larger than) $\frac{\alpha(\alpha+1)d}{(\alpha+2)\sigma}$.
	\end{lemma}
	
	\begin{proof}
		We formulate sufficient conditions for \eqref{2.3+} to \eqref{2.6+}, given in terms of the number $q$. The first condition of $q$ is simply
		\begin{align}\label{a1}
			2<q<\infty.
		\end{align}
		Direct calculation yields $p=\frac{2\alpha \sigma q}{2\sigma q-\alpha d}$. Then $p<\infty$ implies that the denominator $2\sigma q-\alpha d$ is positive, which is equivalent to $q>\frac{\alpha d}{2\sigma}$. On the other hand, $p>2$ implies
		\begin{align}
			\frac{2\alpha \sigma q}{2\sigma q-\alpha d}>2\Leftrightarrow q\sigma(2-\alpha)<\alpha d.
		\end{align}
		We then conclude the second condition on $q$
		\begin{align}\label{a2}
			\frac{\alpha d}{2\sigma}<q<\left\{
			\begin{array}{ll}
				\infty,&\alpha\geq 2,\\
				\frac{\alpha d}{(2-\alpha)\sigma},&\alpha<2.
			\end{array}
			\right.
		\end{align}
		Now it is a straightforward observation that \eqref{2.7+} implies $\tilde{p}'<p$ in \eqref{2.3+} and \eqref{2.6+}. Moreover, $2<\tilde{p},\tilde{q}<\infty$ yield
		\begin{align}
			p<2(\alpha+1),\quad q<2(\alpha+1).
		\end{align}
		Combining with the definition of $p$ formulated in terms of $q$ we conclude the third condition of $q$
		\begin{align}\label{a3}
			\frac{\alpha(\alpha+1)d}{(\alpha+2)\sigma}<q<2(\alpha+1).
		\end{align}
		Using $\alpha\in(\frac{4\sigma}{d},\frac{4\sigma}{d+1-2\sigma})$ one easily verifies that
		\begin{align}
			\frac{\alpha(\alpha+1)d}{(\alpha+2)\sigma}=\max\bg\{\frac{\alpha(\alpha+1)d}{(\alpha+2)\sigma},2,\frac{\alpha d}{2\sigma}\bg\}
			>\max\bg\{2,\frac{\alpha d}{2\sigma}\bg\}.
		\end{align}
		Next, we obtain by fundamental calculation that if $\alpha<2$
		\begin{align}
			\frac{\alpha(\alpha+1)d}{(\alpha+2)\sigma}<\frac{\alpha d}{(2-\alpha)\sigma}\Leftrightarrow
			\alpha^2>0,
		\end{align}
		which always holds. Moreover,
		\begin{align}
			\frac{\alpha(\alpha+1)d}{(\alpha+2)\sigma}<2(\alpha+1)\Leftrightarrow\alpha<\frac{4\sigma}{d-2\sigma}.\label{2.20}
		\end{align}
		Using $\alpha <\frac{4\sigma}{d+1-2\sigma}$ we know that \eqref{2.20} is satisfied. Next, for the last inequality in \eqref{2.3+} we have for $q>\frac{\alpha(\alpha+1)d}{(\alpha+2)\sigma}$
		\begin{align}
			\frac{\alpha}{q}<\frac{(\alpha+2)\sigma}{(\alpha+1)d}<\frac{2\sigma}{d}.
		\end{align}
		Thus for all $q$ sufficiently close to $\frac{\alpha(\alpha+1)d}{(\alpha+2)\sigma}$ we know that \eqref{2.3+}-\eqref{2.7+} are satisfied with $p,q,\tilde{p},\tilde{q}\in(2,\infty)$.
	\end{proof}
	
	Lemma \ref{strichartz lem} and Lemma \ref{friendly pairs} immediately imply the following exotic Strichartz estimates.
	
	\begin{lemma}\label{exotic strichartz}
		Let $d\geq 3$, $\sigma\in[\frac{d}{2d-1},1)$ and $\alpha\in(\frac{4\sigma}{d},\frac{4\sigma}{d+1-2\sigma})$. Then there exist $\ba,\br,\bb,\bs\in(2,\infty)$ such that
		\begin{gather*}
			(\alpha+1)\bs'=\br,\quad(\alpha+1)\bb'=\ba,\quad\frac{\alpha}{\br}  <\min\left\{1,\frac{2\sigma}{d}\right\},\quad
			\frac{2\sigma}{\ba}+\frac{d}{\br}=\frac{2\sigma}{\alpha}.
		\end{gather*}
		Moreover, for any $\gamma\in\R$ we have the following exotic Strichartz estimate:
		\begin{align}
			\norm{\ee^{-\ii   t(-\Delta)^{\sigma}}u_0} _{L_t^\ba L_x^\br H^\gamma_y(I)}&\lesssim \|u_0\|_{H_x^{s_\alpha}H_y^\gamma},\\
			\norm{\int_{t_0}^t \ee^{-\ii (t-s)(-\Delta)^\sigma}F(s)\dd s}_{L_t^\ba L_x^\br H^\gamma_y(I)}&\lesssim \|F\|_{L_t^{\bb'} L_x^{\bs'} H^\gamma_y(I)}
		\end{align}
		for arbitrary $\gamma\in\R$, where $s_\alpha=\frac{d}{2}-\frac{2\sigma}{\alpha}$. Moreover, we can additionally assume that there exists some $0<\beta\ll 1$ such that $\br$ can be chosen as an arbitrary number from $(\frac{\alpha(\alpha+1)d}{(\alpha+2)\sigma},\frac{\alpha(\alpha+1)d}{(\alpha+2)\sigma}+\beta)$.
	\end{lemma}
	
	We will also need the following useful lemma.
	
	\begin{lemma}[A useful H\"older identity]\label{lemma l2 admissible special}
		Let $(\tilde{\ba},\tilde{\br})=(\frac{4\sigma\br}{\alpha d},\frac{2\br}{\br-\alpha})$. Then $(\tilde{\ba},\tilde{\br})$ is an $L^2$-admissible pair and
		\begin{gather}\label{l2 admissible exotic}
			\tilde{\ba}>\max\{2,\tilde{\ba}'\},\quad\frac{1}{\tilde{\ba}'}=\frac{\alpha}{\ba}+\frac{1}{\tilde{\ba}},\quad\frac{1}{\tilde{\br}'}=\frac{\alpha}{\br}+\frac{1}{\tilde{\br}},\\
			\frac{1}{\tilde{\ba}}\leq\frac{2d-1}{2}\bg(\frac{1}{2}-\frac{1}{\tilde{\br}}\bg),\quad (\tilde{\ba},\tilde{\br})\neq (2,\frac{4d-2}{2d-3}).\label{2.28}
		\end{gather}
		is satisfied.
	\end{lemma}
	\begin{remark}
\normalfont
		As a consequence of Lemma \ref{lemma l2 admissible special} and Lemma \ref{strichartz lem}, we know that
		\begin{align}
			\|u\|_{L_t^{\tilde{\ba}} L_x^{\tilde{\br}}H_y^\gamma}\lesssim \|u\|_{L_x^2H_y^\gamma}
			+\|(\ii\pt_t-(-\Delta)^\sigma)u\|_{L_t^{\tilde{\ba}'} L_x^{\tilde{\br}'}H_y^\gamma},
		\end{align}
		for arbitrary $\gamma\in\R$.
	\end{remark}
	
	\begin{proof}
		It is easy to see that \eqref{l2 admissible exotic} is satisfied for the given $(\ba,\br)$ and $\frac{2\sigma}{\ba}+\frac{d}{\br}=\frac{d}{2}$. Moreover, using $\br>\frac{\alpha(\alpha+1)d}{(\alpha+2)\sigma}$ we see that
		\begin{align}
			\tilde{\ba}>\frac{4\sigma}{\alpha d}\times \frac{\alpha(\alpha+1)d}{(\alpha+2)\sigma}=\frac{4(\alpha+1)}{\alpha+2}.
		\end{align}
		However, $\frac{4(\alpha+1)}{\alpha+2}>2$ is equivalent to $\alpha>0$, which always holds. This implies the first inequality in \eqref{l2 admissible exotic} and that $(\tilde{\ba},\tilde{\br})\neq (2,\frac{4d-2}{2d-3})$. It remains to show the first inequality in \eqref{2.28}. Indeed, this is equivalent to
		\begin{align}
			\frac{\alpha d}{4\sigma \br}\leq \frac{2d-1}{2}\bg(\frac{1}{2}-\frac{\br-\alpha}{2\br}\bg)
		\end{align}
		or equivalently after simplifying $\sigma\geq \frac{d}{2d-1}$, which always holds.
	\end{proof}
	
	\begin{lemma}[Fractional calculus on $\T$, \cite{TzvetkovVisciglia2016}]\label{fractional on t}
		For $s\in(0,1)$ and $\alpha>0$ we have
		\begin{align}
			\|u^{\alpha+1}\|_{\dot{H}_y^s}+\||u|^{\alpha+1}\|_{\dot{H}_y^s}+\||u|^\alpha u\|_{\dot{H}_y^s}\lesssim_{\alpha,s} \|u\|_{\dot{H}_y^s}\|u\|_{L_y^\infty}^\alpha
		\end{align}
		for $u\in H_y^s$.
	\end{lemma}
	
	\begin{lemma}[Fractional dispersive estimates on $\R^d$, \cite{SunFrac}]\label{frac disper lem}
		For $d\geq 2$ and $\sigma\in(\frac12,1)$ we have for all $p\in[2,\infty]$
		\begin{align}
			\norm{\ee^{-\ii    t(-\Delta_x)^\sigma}f} _{L_x^p}\lesssim |t|^{-\frac{d}{2}(1-\frac{2}{p})}
			\norm{(-\Delta_x)^{\frac{d(1-\sigma)}{2}(1-\frac2p)}f}_{L_x^{p'}}.
		\end{align}
	\end{lemma}
	
	\begin{lemma}[Fractional radial Sobolev embedding on $\R^d$, \cite{FracRadEmb}]\label{frac rad emb lem}
		Assume that
		\begin{gather*}
			d\geq 1,\quad p,q\in(1,\infty),\quad0\leq \frac{1}{p}-\frac{1}{q}<s< d,\\
			\beta>-\frac{d}{q},\quad\beta+s=\frac{d}{p}-\frac{d}{q}.
		\end{gather*}
		Then for all $u\in\dot{W}_{\rm rad}^{s,p}(\R^d)$ we have
		\begin{align}
			\||x|^{\beta}u\|_{L_x^q}\lesssim \||\nabla_x|^su\|_{L_x^p}.
		\end{align}
	\end{lemma}
	\subsection{Small data and stability theories}

In the following, we present the proof of some useful lemmas, including the small data and stability theories for \eqref{nls}, which are standard preliminaries in a rigidity proof based on the concentration compactness principle.
	\begin{lemma}[Small data well-posedness]\label{lemma small data}
		Let $I$ be an open interval containing $0$. Define the space $X(I)$ through the norm
		\begin{align}
			\|u\|_{X(I)}:=\|u\|_{S_x L_y^2(I)}+\||\nabla_x|^\sigma u\|_{S_x L_y^2(I)}+\||\nabla_y|^\sigma u\|_{S_x L_y^2(I)},
		\end{align}
		where the space $S_x$ is defined by \eqref{def of sx}. Let also $s\in (\frac12,\sigma-s_\alpha)$ be some given number, where $s_\alpha:=\frac{d}{2}-\frac{2\sigma}{\alpha}\in(0,\frac12)$. Assume now
		\begin{align}
			\|u_0\|_{H_{x,y}^\sigma}\leq A
		\end{align}
		for some $A>0$. Then there exists some $\delta=\delta(A)$ such that if
		\begin{align}
			\norm{\ee^{-\ii    t(-\Delta)^\sigma}u_0}_{L_t^\ba L_x^\br  H_y^s(I) }\leq \delta,
		\end{align}
		then there exists a unique solution $u\in X(I)$ of \eqref{nls} with $u(0)=u_0$ such that
		\begin{gather}
			\|u\|_{X(I)}\lesssim A\quad\text{and}\quad
			\|u\|_{L_t^\ba L_x^\br  H_y^s(I)}\leq 2\norm{\ee^{-\ii    t(-\Delta)^\sigma}u_0}_{L_t^\ba L_x^\br  H_y^s(I)}.
		\end{gather}
	\end{lemma}
	\begin{remark}
\normalfont
		Since $\alpha<\frac{4\sigma}{d+1-2\sigma}$, the interval $(\frac12,\sigma-s_\alpha)$ is not empty.
	\end{remark}
	
	\begin{proof}
		We define the space
		\begin{align*}
			B(I):=\{u\in X(I):\|u\|_{X(I)}\leq 2CA,\|u\|_{L_t^\ba L_x^\br  H_y^s(I)}\leq 2\delta\}.
		\end{align*}
		Then $B(I)$ is a complete metric space with the metric $\rho(u,v):=\|u-v\|_{S_x L_y^2(I)}$. Define also the Duhamel  mapping $\Psi(u)$ by
		\begin{align*}
			\Psi(u):=\ee^{-\ii    t(-\Delta)^\sigma}u_0+\ii\int_0^t \ee^{-\ii  (t-s)(-\Delta)^\sigma}(|u|^\alpha u)(s) \dd s.
		\end{align*}
		We first show that $\Psi(B(I))\subset B(I)$ by choosing $\delta$ small. Let $D\in\{1,|\nabla_y|^\sigma\}$. Let $(\tilde{\ba},\tilde{\br})$ be the $L^2$-admissible pairs given by Lemma \ref{lemma l2 admissible special}. Using Lemma \ref{exotic strichartz}, Lemma \ref{lemma l2 admissible special}, Lemma \ref{fractional on t}, the embedding $H_y^s\hookrightarrow L_y^\infty$ and te H\"older inequality we obtain
		\begin{align}
			\|D(\Psi(u))\|_{S_x L_y^2(I)}&
			\leq CA+C\|\|u\|^{\alpha}_{L_y^{\infty}}\|Du\|_{L_y^2}\|_{L_t^{\tilde{\ba}'}L_x^{\tilde{\br}'}(I)}\nonumber\\
			&\leq CA+C\|\|u\|^{\alpha}_{H_y^s}\|Du\|_{L_y^2}\|_{L_t^{\tilde{\ba}'}L_x^{\tilde{\br}'}(I)}
			\leq CA+C\|u\|^{\alpha}_{L_t^{{\ba}}L_x^{{\br}} H_y^s(I)}\|Du\|_{L_t^{\tilde{\ba}}L_x^{\tilde{\br}}L_y^2(I)}\nonumber\\
			&
			\leq CA+(2\delta)^{\alpha} CA\leq 2CA\label{3.5}
		\end{align}
		by choosing $\delta$ sufficiently small. Moreover, for $D=|\nabla_x|^\sigma$, \eqref{3.5} follows similarly by taking the anisotropic fractional chain rule \cite[Thm. A.6]{Anisotropic} into account. Similarly, using Lemma \ref{exotic strichartz}, the embedding $H_y^s\hookrightarrow L_y^\infty$ and Lemma \ref{fractional on t} we obtain
		\begin{align*}
			\|\Psi(u)\|_{L_t^\ba L_x^\br  H_y^s(I)}
			&\leq \delta+\||u|^{\alpha}u\|_{L_t^{\bb'} L_x^{\bs'} H^s_y(I)}
			\leq \delta+C\left(\||u|^{\alpha}u\|_{L_t^{\bb'} L_x^{\bs'} L^2_y(I)}+\||u|^{\alpha}u\|_{L_t^{\bb'} L_x^{\bs'} \dot{H}^s_y(I)}\right)\nonumber\\
			&\leq \delta+C\left(\|\|u\|^{\alpha+1}_{L_y^\infty}\|_{L_t^{\bb'} L_x^{\bs'}(I)}+
			\|\|u\|^{\alpha}_{L_y^\infty}\| u\|_{\dot{H}_y^s}\|_{L_t^{\bb'} L_x^{\bs'}(I)}\right)\nonumber\\
			&\leq \delta+C\|\|u\|^{\alpha+1}_{H_y^s}\|_{L_t^{\bb'} L_x^{\bs'}(I)}=\delta+C\|u\|^{\alpha+1}_{L_t^{\ba} L_x^{\br}H_y^s(I)}
			\leq \delta(1+C\delta^\alpha)\leq 2\delta.
		\end{align*}
		Finally, we show that $\Psi$ is a contraction on $B(I)$. By direct calculation, we infer that
		\begin{align*}
			\|\Psi(u)-\Psi(v)\|_{S_x L_y^2}&\leq C\||u|^\alpha u-|v|^\alpha v\|_{L_t^{\tilde{\ba}'}L_x^{\tilde{\br}'}L_y^2(I)}
			\leq C\|(|u|^{\alpha}+|v|^\alpha)(u-v)\|_{L_t^{\tilde{\ba}'}L_x^{\tilde{\br}'}L_y^2(I)}\nonumber\\
			&\leq C\|(\|u\|^{\alpha}_{L_y^{\infty}}+\|v\|^{\alpha}_{L_y^{\infty}})\|u-v\|_{L_y^2}\|_{L_t^{\tilde{\ba}'}L_x^{\tilde{\br}'}(I)}\nonumber\\
			&\leq C(\|u\|_{L_t^{{\ba}}L_x^{{\br}} H_y^s(I)}+\|v\|_{L_t^{{\ba}}L_x^{{\br}} H_y^s(I)})^\alpha\|u-v\|_{S_x L_y^2(I)}\leq C\delta^\alpha
			\|u-v\|_{S_x L_y^2(I)}.
		\end{align*}
		The desired claim follows by choosing $\delta$ small and using the Banach  fixed point theorem.
	\end{proof}
	
	\begin{lemma}[Scattering norm]\label{lemma scattering norm}
		Let $u\in X_{\rm loc}(\R)$ be a global solution of \eqref{nls} such that
		\begin{align}
			\|u\|_{L_t^\ba L_x^\br  H_y^s(\R)}+\|u\|_{L_t^\infty H_{x,y}^\sigma(\R)}<\infty.
		\end{align}
		Then $u$ scatters in time. Moreover, for all $s'\in(\frac12,\sigma-s_\alpha)$ we have
		\begin{align}\label{uniform bound}
			\|u\|_{X(\R)}+\|u\|_{L_t^\ba L_x^\br H_y^{s'}(\R)}\lesssim_{\|u\|_{L_t^\ba L_x^\br  H_y^s(\R)},\,\|u\|_{L_t^\infty H_{x,y}^\sigma(\R)}}1.
		\end{align}
	\end{lemma}
	\begin{proof}
		We first show \eqref{uniform bound}. Partition $\R$ into $\R=\cup_{j=1}^m I_j$ such that $\|u\|_{L_t^\ba L_x^\br  H_y^s(I_j)}=O(m^{-1})$. Let $D\in\{1,|\nabla_x|^\sigma,|\nabla_y|^\sigma\}$. Arguing as in the proof of Lemma \ref{lemma small data}, on the interval $I_j=[t_j,t_{j+1}]$ with $s_j\in(t_j,t_{j+1})$ we have
		\begin{align*}
			\|u\|_{X(I_j)}&\lesssim \|\ee^{-\ii (t-s_j)(-\Delta)^\sigma}u(s_j)\|_{X(I_j)}+\sum_{D\in\{1,|\nabla_x|^\sigma,|\nabla_y|^\sigma\}}\|u\|^{\alpha}_{L_t^{{\ba}}L_x^{{\br}} H_y^s(I_j)}\|Du\|_{L_t^{\tilde{\ba}}L_x^{\tilde{\br}}L_y^2(I_j)}\nonumber\\
			&\lesssim \|u(s_j)\|_{ H_{x,y}^\sigma}+O(m^{-\alpha})\|u\|_{X(I_j)}.
		\end{align*}
		Choosing $m$ sufficiently large we can absorb the term $O(m^{-\alpha})\|u\|_{X(I_j)}$ to the left-hand side   and conclude that $\|u\|_{X(I_j)}\lesssim \|u(s_j)\|_{ H_{x,y}^\sigma}$. The upper bound of $\|u\|_{X(\R)}$ given by \eqref{uniform bound} follows then by summing up the partial estimates on $I_j$ over $j=1,\cdots,m$. The upper bound of $\|u\|_{L_t^\ba L_x^\br H_y^{s'}(\R)}$ follows very similarly by also appealing to Lemma \ref{fractional on t}:
		\begin{align*}
			\|u\|_{L_t^\ba L_x^\br H_y^{s'}(I_j)}&\lesssim \|u(s_j)\|_{ H_{x,y}^\sigma}+\|u\|^{\alpha}_{L_t^{{\ba}}L_x^{{\br}} H_y^s(I_j)}\|u\|_{L_t^{\ba}L_x^{\br}H_y^{s'}(I_j)}\lesssim \|u(s_j)\|_{ H_{x,y}^\sigma}+O(m^{-\alpha})\|u\|_{L_t^{{\ba}}L_x^{{\br}} H_y^{s'}(I_j)}.
		\end{align*}
		Next, define
		\begin{align*}
			\phi:=u_0+\ii\int_0^\infty \ee^{\ii  s(-\Delta)^\sigma}(|u|^\alpha u)(s) \dd s.
		\end{align*}
		Then
		\begin{align*}
			u-\ee^{-\ii    t(-\Delta)^\sigma}\phi=-\ii\int_t^\infty \ee^{-\ii    (t-s)(-\Delta)^\sigma}(|u|^\alpha u)(s) \dd s.
		\end{align*}
		Consequently, using the Strichartz estimates and the fractional chain rule we infer that
		\begin{align*}
			\|D(u-\ee^{-\ii    t(-\Delta)^\sigma}\phi)\|_{L_{x,y}^2}&
			\lesssim \|u\|^{\alpha}_{L_t^{{\ba}}L_x^{{\br}} H_y^s(t,\infty)}\|Du\|_{L_t^{\tilde{\ba}}L_x^{\tilde{\br}}L_y^2(t,\infty)}\nonumber\\
			&\lesssim\|u\|^{\alpha}_{L_t^{{\ba}}L_x^{{\br}} H_y^s(t,\infty)}\|u\|_{X(\R)}\lesssim\|u\|^{\alpha}_{L_t^{{\ba}}L_x^{{\br}} H_y^s(t,\infty)}\to 0
		\end{align*}
		as $t\to\infty$, since $\|u\|_{L_t^{{\ba}}L_x^{{\br}} H_y^s(\R)}<\infty$. This shows that $u$ scatters in positive time. That $u$ scatters in negative time follows verbatim, we omit the details here.
	\end{proof}
	
	\begin{lemma}[Criterion for maximality of lifespan]
		Let $u\in X_{\rm loc}(I_{\max})$ be a solution of \eqref{nls} defined on its maximal lifespan $I_{\max}$. Then if $t_{\max}:=\sup I_{\max}<\infty$, we have $\|u\|_{L_t^\ba L_x^\br  H_y^s(t,t_{\max})}=\infty$ for all $t\in I_{\max}$. A similar result holds for $t_{\min}:=\inf I_{\max}>-\infty$.
	\end{lemma}
	
	\begin{proof}
		Assume the contrary that there exists some $t_0\in I_{\max}$ such that
		\begin{align}\label{contradiction criterion}
			\|u\|_{L_t^\ba L_x^\br  H_y^s(t_0,t_{\max})}<\infty.
		\end{align}
		We will show that $\lim_{t\to t_{\max}}u(t)$ exists in $H_{x,y}^\sigma$. By Lemma \ref{lemma small data} this would mean that we can extend $u(t)$ beyond $t_{\max}$ for $t$ sufficiently close to $t_{\max}$, which contradicts the maximality of $I_{\max}$. By the strong continuity of the linear group $(\ee^{-\ii   t(-\Delta)^\sigma})_{t\in\R}$ it suffices to show that $(\ee^{\ii  t(-\Delta_{x})^\sigma+(-\Delta_{y})^\sigma}u(t))_{t\nearrow t_{\max}}$ is a Cauchy sequence. Let $t_m<t_n<t_{\max}$. By the Duhamel  formula and the Strichartz estimate we have
		\begin{align}
			&\,\|\ee^{\ii  t_m(-\Delta)^\sigma}u(t_m)-\ee^{\ii  t_n(-\Delta)^\sigma}u(t_n)\|_{H_{x,y}^\sigma}\nonumber\\
			=&\,\norm{\int_{t_m}^{t_n}\ee^{-\ii  (t_n-s)(-\Delta)^\sigma}(|u|^\alpha u)(s)\dd s}_{H_{x,y}^\sigma}
			\leq \norm{\int_{t_m}^{t}\ee^{-\ii (t-s)(-\Delta)^\sigma}(|u|^\alpha u)(s)\dd s}_{L_t^\infty H_{x,y}^\sigma(t_m,t_{\max})}\nonumber\\
			\lesssim &\,
			\sum_{D\in\{1,|\nabla_x|^\sigma,|\nabla_y|^\sigma\}} \|u\|^{\alpha}_{L_t^{{\ba}}L_x^{{\br}} H_y^s(t_m,t_{\max})}\|Du\|_{L_t^{\tilde{\ba}}L_x^{\tilde{\br}}L_y^2(t_m,t_{\max})}.\label{cauchy seq}
		\end{align}
		From the proof of Lemma \ref{lemma scattering norm} and \eqref{contradiction criterion} we know that
		\begin{gather*}
			\|u\|_{X(t_0,t_{\max})}<\infty,\\
			\lim_{t\nearrow t_{\max}}\|u\|_{L_t^\ba L_x^\br  H_y^s(t,t_{\max})}=0
		\end{gather*}
		Combining with \eqref{cauchy seq} we deduce the desired claim.
	\end{proof}
	\begin{lemma}[Stability theory]\label{lem stability cnls}
		Assume $d\in\{2,3,4\}$, $\sigma\in [\frac{d}{2d-1},1)$ and $\alpha\in(\max\{1,\frac{4\sigma}{d}\},\frac{4\sigma}{d+1-2\sigma})$. Let $u$ be a solution of \eqref{nls} defined on some interval $0\in I\subset\R$ and let $\tilde{u}$ be a solution of the perturbed NLS
		\begin{align}
			i\pt_t\tilde{u}-(-\Delta)^\sigma\tilde{u}+|\tilde{u}|^\alpha \tilde{u}+e=0.
		\end{align}
		Assume that there exists some $A>0$ such that
		\begin{align}
			\|\tilde{u}\|_{L_t^\ba L_x^\br H_y^s(I)}&\leq A.\label{cond 1}
		\end{align}
		Then there exist $\vare(A)>0$ and $C(A)>0$ such that if
		\begin{align}
			\|e\|_{L_t^{\bb'}L_x^{\bs'}H_y^s(I)}&\leq \vare\leq\vare(A),\label{cond 2}\\
			\norm{\ee^{-\ii   t(-\Delta)^\sigma}(u(0)-\tilde{u}(0))}_{L_t^{\ba}L_x^{\br}H_y^s(I)}&\leq \vare\leq\vare(A),\label{cond 3}
		\end{align}
		then $\|u-\tilde{u}\|_{L_t^\ba L_x^\br H_y^s(I)}\leq C(A)\vare$.
	\end{lemma}
	
	\begin{proof}
		Without loss of generality, let $I=(-T,T)$ for some $T\in(0,\infty]$.  Set $w:=u-\tilde{u}$ and $F(z):=|z|^\alpha z$ for $z\in\C$. Then by   \cite[Lemma 4.1]{TzvetkovVisciglia2016},
		\begin{align}\label{long1}
			&\,c\||\tilde{u}+w|^\alpha(\tilde{u}+w)-|\tilde{u}|^\alpha \tilde{u}\|_{\dot{H}_y^s}^2\nonumber\\
			=&\,\int_0^{2\pi}\int_{\R}\frac{|(F((\tilde{u}+w)(x+h))-F((\tilde{u})(x+h)))-(F((\tilde{u}+w)(x))-F((\tilde{u})(x)))|^2}{|h|^{1+2s}}\dd h\dd x.
		\end{align}
		Writing $z=a+bi$ we may identify $F(z)$ with the two-dimensional function $F(a,b)$ through $F(a,b)=F(z)$. Define the usual complex derivatives by
		\begin{align*}
			F_z:=\frac{1}{2}\bg(\frac{\pt F}{\pt a}-i\frac{\pt F}{\pt b}\bg),\quad F_{\bar{z}}:=\frac{1}{2}\bg(\frac{\pt F}{\pt a}+i\frac{\pt F}{\pt b}\bg).
		\end{align*}
		Using the chain rule we obtain that for $z,z'\in\C$
		\begin{align*}
			F(z)-F(z')=(z-z')\int_0^1 F_z(z'+\theta(z-z'))\dd\theta+\overline{(z-z')}\int_0^1 F_{\bar{z}}(z'+\theta(z-z'))\dd\theta.
		\end{align*}
		Combining with the standard telescoping argument and the fact that $\alpha>1$ we see that
		\begin{equation}	\label{long2}
			\begin{split}
			 |(F((\tilde{u}+w)(x+h))&-F((\tilde{u})(x+h)))-(F((\tilde{u}+w)(x))-F((\tilde{u})(x)))| \\
			&\lesssim\,|w(x+h)-w(x)|(|\tilde{u}(x+h)|^\alpha+|w(x+h)|^\alpha) \\
		&\quad	+ |w(x)|\left(|\tilde{u}(x+h)|+|w(x+h)|+|\tilde{u}(x)|+|w(x)|\right)^{\alpha-1}\\&
			\quad
			\times(|\tilde{u}(x+h)-\tilde{u}(x)|+|w(x+h)-w(x)|).
				\end{split}
		\end{equation}
	Equations	\eqref{long1}, \eqref{long2} and the embedding $H_y^s\hookrightarrow L_y^\infty$ now yield
	\[	\begin{split}
			&\,\norm{|\tilde{u}+w|^\alpha(\tilde{u}+w)-|\tilde{u}|^\alpha \tilde{u}}_{\dot{H}_y^s}^2 \\
		&	\lesssim  \int_0^{2\pi}\int_{\R}\frac{|w(x+h)-w(x)|^2
			\left(\|\tilde{u}\|_{L_y^\infty}^{2\alpha}+\|w\|_{L_y^\infty}^{2\alpha}\right)}{|h|^{1+2s}}\dd h\dd x \\
		&\quad	+ \int_0^{2\pi}\int_{\R}
			\frac{\left(|\tilde{u}(x+h)-\tilde{u}(x)|^2+|w(x+h)-w(x)|^2\right)\left(\|\tilde{u}\|_{L_y^\infty}^{2\alpha-2}+\|w\|_{L_y^\infty}^{2\alpha-2}\right)\|w\|^2_{L_y^\infty}}{|h|^{1+2s}}\dd h\dd x
			 \\&
			\lesssim
			\|w\|^2_{\dot{H}_y^s}\left(\|\tilde{u}\|_{L_y^\infty}^{2\alpha}+\|w\|_{L_y^\infty}^{2\alpha}\right)
			+\|w\|^2_{\dot{H}_y^s}\left(\|\tilde{u}\|^2_{\dot{H}_y^s}+\|w\|^2_{\dot{H}_y^s}\right)\left(\|\tilde{u}\|_{L_y^\infty}^{2\alpha-2}+\|w\|_{L_y^\infty}^{2\alpha-2}\right)
			 \\&
			\lesssim  \|\tilde{u}\|^{2\alpha}_{H_y^s}\|w\|^2_{H_y^s}+\|w\|^{2\alpha+2}_{H_y^s}.
		\end{split}\]
		This in turn implies
		\begin{align*}
			\||\tilde{u}+w|^\alpha(\tilde{u}+w)-|\tilde{u}|^\alpha \tilde{u}\|_{\dot{H}_y^s}\lesssim
			\|\tilde{u}\|^{\alpha}_{H_y^s}\|w\|_{H_y^s}+\|w\|^{\alpha+1}_{H_y^s}.
		\end{align*}
		On the other hand, a simple application of the  H\"older  inequality also yields
		\begin{align*}
			\||\tilde{u}+w|^\alpha(\tilde{u}+w)-|\tilde{u}|^\alpha \tilde{u}\|_{L_y^2}\lesssim
			\|\tilde{u}\|^{\alpha}_{H_y^s}\|w\|_{H_y^s}+\|w\|^{\alpha+1}_{H_y^s}
		\end{align*}
		and we conclude
		\begin{align}\label{final long}
			\||\tilde{u}+w|^\alpha(\tilde{u}+w)-|\tilde{u}|^\alpha \tilde{u}\|_{H_y^s}\lesssim
			\|\tilde{u}\|^{\alpha}_{H_y^s}\|w\|_{H_y^s}+\|w\|^{\alpha+1}_{H_y^s}.
		\end{align}
		We also notice that $w$ satisfies the NLS
		\begin{align*}
			\ii\pt_t w-(-\Delta)^\sigma w+|\tilde{u}+w|^\alpha(\tilde{u}+w)-|\tilde{u}|^\alpha\tilde{u}+e=0.
		\end{align*}
		Using the exotic Strichartz estimates given in Lemma \ref{exotic strichartz} and \eqref{final long} we infer that for $t\in(0,T)$
		\begin{align}\label{cazenave}
			\|w\|_{L_t^\ba L_x^\br H_y^s(-t,t)}\leq(M+1)\vare +M\|\|\tilde{u}\|^{\alpha}_{H_y^s}\|w\|_{H_y^s}+\|w\|^{\alpha+1}_{H_y^s}\|_{L_t^{\bb'}L_x^{\bs'}(-t,t)}
		\end{align}
		for some $M>0$. We point out that \eqref{cazenave} is exactly (4.20) in \cite{focusing_sub_2011}. Hence following the proof of \cite[Prop. 4.7]{focusing_sub_2011}, words by words, we deduce the desired claim by setting
		\begin{align*}
			\vare(A)\leq 2^{-\frac{1}{\alpha}}[(2M+1)\Phi(A^\alpha)]^{-\frac{\alpha+1}{\alpha}},
		\end{align*}
		where $\Phi$ is the function given by \cite[Lem. 8.1]{focusing_sub_2011}. Notice that in order to apply \cite[Lem. 8.1]{focusing_sub_2011}, the quantities
		\begin{align*}
			\gamma:=\ba,\quad\rho:=\frac{\ba}{\alpha},\quad\beta=\bb'
		\end{align*}
		should satisfy $1\leq\beta<\gamma\leq \infty$ and $1\leq\rho<\infty$. By direct calculation, this is clearly fulfilled when $\br$ is sufficiently close to $\frac{\alpha(\alpha+1)d}{(\alpha+2)\sigma}$, the latter being guaranteed by Lemma \ref{exotic strichartz}.
	\end{proof}
	\subsection{Profile decomposition}
	We derive in this subsection a profile decomposition for a bounded sequence in $H_{x,y}^\sigma$. To begin with, we first record an inverse Strichartz inequality.
	
	\begin{lemma}[Inverse Strichartz inequality]\label{refined l2 lemma 1}
		Let $(f_n)_n\subset H_{x,y}^\sigma$. Suppose that
		\begin{align}\label{425}
			\lim_{n\to\infty}\|f_n\|_{H_{x,y}^\sigma}=A<\infty\quad\text{and}\quad\lim_{n\to\infty}\norm{\ee^{-\ii    t(-\Delta_{x})^\sigma} f_n}_{L_t^\ba L_{x,y}^\br (\R)}=\vare>0.
		\end{align}
		Then up to a subsequence, there exist $\phi\in H_{x,y}^\sigma$ and $(t_n)_n\subset\R$ such that
		\begin{align}\label{cnls l2 refined strichartz basic 5}
			\ee^{\ii  t_n\Delta_x}f_n(x,y)\rightharpoonup &\,\,\phi(x,y)\text{ weakly in $H_{x,y}^1$}.
		\end{align}
		Set $\phi_n:=\ee^{\ii  t_n(-\Delta_x)^\sigma}\phi(x,y)$. Let $s\in(\frac{1}{2},\sigma-s_\alpha)$ be some given number, where $s_\alpha:=\frac{d}{2}-\frac{2\sigma}{\alpha}\in(0,\frac12)$. Then there exists some positive $\kappa\ll 1$ and $\theta,\beta\in(0,1)$ such that for $D\in\{1,(-\Delta_x)^{\frac{\sigma}{2}},(-\pt_y^2)^{\frac{\sigma}{2}}\}$ we have
		\begin{align}
			\lim_{n\to\infty}&\left(\|f_n\|^2_{H_{x,y}^\sigma}-\|f_n-\phi_n\|^2_{H_{x,y}^\sigma}\right)=\|\phi\|^2_{H_{x,y}^\sigma}\gtrsim
			A^{-\frac{2(\theta+\beta(1-\theta))}{(1-\beta)(1-\theta)}-\frac{d}{\kappa}}\vare^{\frac{2}{(1-\beta)(1-\theta)}+\frac{d}{\kappa}},
			\label{cnls l2 refined strichartz decomp 1}\\
			\lim_{n\to\infty}&\left(\|Df_n\|^2_{L_{x,y}^2}-\|D(f_n-\phi_n)\|^2_{L_{x,y}^2}-\|D\phi_n\|^2_{L_{x,y}^2}\right)=0.
			\label{cnls l2 refined strichartz decomp 3}
		\end{align}
	\end{lemma}
	
	\begin{proof}
		For $R\geq 1$ let $\chi_R:\R^d\to[0,1]$ be the indicator function of the ball $\{\xi\in\R^d:|\xi|\leq R\}$. We then define $f_n^R$ through its symbol $\mathcal{F}_x(f_n^R)=\chi_R \mathcal{F}_x(f_n)$. By the Strichartz estimate and the embedding $H_y^s\hookrightarrow L_y^\br$ we obtain
		\begin{align*}
			&\,\|\ee^{-\ii   t(-\Delta_{x})^\sigma}(f_n-f_n^R)\|_{L_t^\ba L_{x,y}^\br (\R)}\nonumber\\
			\lesssim&\,\|\ee^{-\ii   t(-\Delta)^\sigma}\ee^{\ii  t(-\pt_y^2)^\sigma}(f_n-f_n^R)\|_{L_t^\ba L_{x}^\br H_y^s (\R)}
			\lesssim \|f_n-f_n^R\|_{H_x^{s_\alpha} H_y^s}\nonumber\\
			\lesssim&\, \|f_n-f_n^R\|_{L_{x,y}^2}+\|f_n-f_n^R\|_{L_x^2 \dot{H}_y^s}+\|f_n-f_n^R\|_{\dot{H}_x^{s_\alpha} L_y^2 }+\|f_n-f_n^R\|_{\dot{H}_x^{s_\alpha} \dot{H}_y^{s}}\nonumber\\
			=:&\,I+II+III+IV.
		\end{align*}
		Writing the Hilbert norms via Fourier transform we deduce
	\begin{equation}\label{or1}
			\begin{split}
			 (I+III)^2
			&\lesssim  \sum_{k\in\Z}\int_{|\xi|\geq R}|\mathcal{F}_{x,y}(f_n)(\xi,k)|^2\dd\xi
			+\sum_{k\in\Z}\int_{|\xi|\geq R}|\xi|^{2s_\alpha}|\mathcal{F}_{x,y}(f_n)(\xi,k)|^2\dd\xi
			 \\&
			\lesssim  (R^{-2\sigma}+R^{-2(\sigma-s_\alpha)})\sum_{k\in\Z}\int_{|\xi|\geq R}|\xi|^{2\sigma}|\mathcal{F}_{x,y}(f_n)(\xi,k)|^2\dd\xi
			\\&
			\lesssim (R^{-2}+R^{-2(\sigma-s_\alpha)})\|f_n\|^2_{\dot{H}_x^\sigma L_y^2}\lesssim (R^{-2}+R^{-2(\sigma-s_\alpha)})A^2.
		\end{split}
	\end{equation}
		For $II$ and $IV$, since $s\in(\frac{1}{2},\sigma-s_\alpha)$, we can find some positive $\kappa\ll 1$ such that $s\in(\frac{1}{2},\sigma-(s_\alpha+\kappa))$. Then using the H\"older inequality we infer that
		\begin{equation}
			\label{or2}
			\begin{split}
			&(II+IV)^2
			\lesssim\sum_{k\in\Z}k^{2s}\int_{|\xi|\geq R}|\mathcal{F}_{x,y}(f_n)(\xi,k)|^2\dd\xi
			+\sum_{k\in\Z}k^{2s}\int_{|\xi|\geq R}|\xi|^{2s_\alpha}|\mathcal{F}_{x,y}(f_n)(\xi,k)|^2\dd\xi \\
			&\lesssim  (R^{-2(s_\alpha+\kappa)}+R^{-2\kappa})\sum_{k\in\Z}k^{2s}\int_{|\xi|\geq R}|\xi|^{2(s_\alpha+\kappa)}|\mathcal{F}_{x,y}(f_n)(\xi,k)|^2\dd\xi \\&
			\lesssim  (R^{-2(s_\alpha+\kappa)}+R^{-2\kappa})\sum_{k}(k^\sigma)^{2(1-(s_\alpha+\kappa)/\sigma)}\|\mathcal{F}_{y} (f_n)(k)\|^{2(1-(s_\alpha+\kappa)/\sigma)}_{L_x^2}
			\||\nabla_x|^\sigma\mathcal{F}_{y} (f_n)(k)\|^{2(s_\alpha+\kappa)/\sigma}_{L_x^2}
			 \\&
			\lesssim (R^{-2(s_\alpha+\kappa)}+R^{-2\kappa})\|(k^\sigma\|\mathcal{F}_{y} (f_n)(k)\|_{L_x^2})_k\|^{2(1-(s_\alpha+\kappa)/\sigma)}_{\ell_k^2}
			\times \norm{(\||\nabla_x|^\sigma\mathcal{F}_{y} (f_n)(k)}_{L_x^2})_k\|^{2(s_\alpha+\kappa)/\sigma}_{\ell_k^2} \\&
			\sim (R^{-2(s_\alpha+\kappa)}+R^{-2\kappa})\|f_n\|_{L_x^2 \dot{H}_y^\sigma}^{2(1-(s_\alpha+\kappa)/\sigma)}
			\||\nabla_x|^\sigma f_n\|_{L_{x,y}^2}^{2(s_\alpha+\kappa)/\sigma}\lesssim (R^{-2(s_\alpha+\kappa)}+R^{-2\kappa})A^2.
		\end{split}
		\end{equation}
		Hence there exists some $C>0$ independent of $R$, $\vare$ and $A$ such that for all sufficiently large $n$
		\begin{align}
			\|\ee^{-\ii   t(-\Delta_{x})^\sigma}(f_n-f_n^R)\|_{L_t^\ba L_{x,y}^\br (\R)}\leq CR^{-\kappa}A.
		\end{align}
		Let $R=\bg(\frac{4AC}{\vare}\bg)^{\frac{1}{\kappa}}$. Then by \eqref{425} we obtain
		\begin{align}\label{large R}
			\liminf_{n\to\infty}\|\ee^{-\ii   t(-\Delta_{x})^\sigma}f_n^R\|_{L_t^\ba L_{x,y}^\br (\R)}\geq \frac{\vare}{4}.
		\end{align}
		Let $s_\alpha^+\in(s_\alpha,\infty)$ be some number to be chosen later. Define $\alpha^+$ such that $s_\alpha^+=\frac{d}{2}-\frac{2\sigma}{\alpha^+}$. When $s_\alpha^+$ is close to $s_\alpha$ we will have $\alpha^+>\frac{4\sigma}{d}$. Consequently,
		\begin{align*}
			d>s_\alpha^+=\frac{d}{2}-\frac{d}{(\alpha^+ d/(2\sigma))}=d\bg(\frac{1}{2}-\frac{1}{\alpha^+ d/(2\sigma)}\bg)>
			\frac{1}{2}-\frac{1}{\alpha^+ d/(2\sigma)}>0.
		\end{align*}
		Thus choosing $\alpha^{*}>\alpha^+$ sufficiently close to $\alpha^+$ we infer that
		\begin{align*}
			d>s_\alpha^+>0,\quad s_\alpha^+<\frac{d}{2}-\frac{d}{(\alpha^{*} d/(2\sigma))},\quad s_\alpha^+>
			\frac{1}{2}-\frac{1}{\alpha^{*} d/(2\sigma)}>0
		\end{align*}
		and by Lemma \ref{frac rad emb lem} we know that for $\beta:=\frac{d}{2}-\frac{d}{(\alpha^{*} d/(2\sigma))}-s_\alpha^+>0$ we have
		\begin{align}
			\||x|^\beta u\|_{L_x^{\br^*}}\lesssim \||\nabla_x|^{s_\alpha^+}u\|_{L_x^2},
		\end{align}
		where $\br^*:=\frac{\alpha^* d}{2\sigma}$. For $M>0$, this in turn implies
		\begin{align}
			&\,\| \ee^{-\ii   t(-\Delta_{x})^\sigma} f^R_n\|_{L_t^\infty L_{x,y}^{\br^*}(\{|x|\geq M\}\times\R_t)} \nonumber\\
			\lesssim&\, M^{-\beta}\||x|^\beta \ee^{-\ii   t(-\Delta_{x})^\sigma} f^R_n\|_{L_{t}^\infty L_y^\infty L_{x}^{\br^*}(\R)}
			\lesssim M^{-\beta}\||\nabla_x|^{s_\alpha^+}f^R_n\|_{L_{t}^\infty H_y^s L_{x}^{2}(\R)}\nonumber\\
			\lesssim&\, M^{-\beta}\|f^R_n\|_{L_t^\infty H_{x,y}^\sigma}\lesssim M^{-\beta}A.
		\end{align}
		For $\theta\in(0,1)$, define $(\ba^-,\br^-)$ via
		\begin{align}
			\ba^{-1}=\theta (\ba^{-})^{-1},\quad \br^{-1}=\theta (\br^{-})^{-1}+(1-\theta)(\br^*)^{-1}.
		\end{align}
		Define also $s_\alpha^-:=\frac{d}{2}-\frac{2\sigma}{\ba^-}-\frac{d}{\br^-}$. For $\theta$ close to $1$ we see that Lemma \ref{exotic strichartz} is applicable for $(\ba^-,\br^-)$. We also define $\beta\in(0,1)$ such that $(\br^*)^{-1}=\beta 2^{-1}$. Then using an interpolation and the Strichartz estimates we obtain
		\begin{align*}
			\vare&\lesssim\liminf_{n\to\infty} \norm{\ee^{-\ii   t(-\Delta_{x})^\sigma} f^R_n} _{L_t^\ba L_{x,y}^\br (\R)}\nonumber\\
			&\lesssim
			\liminf_{n\to\infty}\bg( \norm{\ee^{-\ii   t(-\Delta_{x})^\sigma} f^R_n} ^\theta_{L_t^{\ba^-} L_{x,y}^{\br-} (\R)}
			 \norm{\ee^{-\ii   t(-\Delta_{x})^\sigma} f^R_n} ^{1-\theta}_{L_t^\infty L_{x,y}^{\br*} (\R)}\bg)\nonumber\\
			&\lesssim \liminf_{n\to\infty}\bg(\|f^R_n\|^\theta_{H_x^{s_\alpha^-}H_y^s }
			 \norm{\ee^{-\ii   t(-\Delta_{x})^\sigma} f^R_n}^{1-\theta}_{L_t^\infty L_{x,y}^{\br^*} (\R)}\bg)\nonumber\\
			&\lesssim M^{-(1-\theta)\beta}A+A^\theta\liminf_{n\to\infty}\bg(
			\norm{\ee^{-\ii   t(-\Delta_{x})^\sigma} f^R_n}^{\beta(1-\theta)}_{L_t^\infty L_{x,y}^2 (\R)}
			\norm{\ee^{-\ii   t(-\Delta_{x})^\sigma} f^R_n}^{(1-\beta)(1-\theta)}_{L_{t,x,y}^\infty(\{|x|\leq M\}\times\R_t)}\bg)\nonumber\\
			&\lesssim M^{-(1-\theta)\beta}A+A^{\theta+\beta(1-\theta)} \liminf_{n\to\infty} \norm{\ee^{-\ii   t(-\Delta_{x})^\sigma} f^R_n}^{(1-\beta)(1-\theta)}_{L_{t,x,y}^\infty(\{|x|\leq M\}\times\R_t)}.
		\end{align*}
		Choosing $M$ satisfying $M^{-(1-\theta)\beta}A\ll \vare$ and rearranging terms then implies
		\begin{align*}
			\liminf_{n\to\infty}\norm{\ee^{-\ii   t(-\Delta_{x})^\sigma}f^R_n}_{L_{t,x,y}^\infty(\{|x|\leq M\}\times\R_t)}\gtrsim A^{-\frac{\theta+\beta(1-\theta)}{(1-\beta)(1-\theta)}}\vare^{\frac{1}{(1-\beta)(1-\theta)}}.
		\end{align*}
		Hence there exist $(t_n,x_n,y_n)_n\subset\R\times\{|x|\leq M\}\times\T$ such that
		\begin{align*}
			\liminf_{n\to\infty} |\ee^{-\ii   t_n(-\Delta_{x})^\sigma}f_n^R(x_n,y_n)|\gtrsim A^{-\frac{\theta+\beta(1-\theta)}{(1-\beta)(1-\theta)}}\vare^{\frac{1}{(1-\beta)(1-\theta)}},
		\end{align*}
		or equivalently
		\begin{align}\label{loew bound}
			\liminf_{n\to\infty} \abso{\int_{\R^d}(\mathcal{F}_x^{-1}\chi_R)(-z)\ee^{-\ii   t(-\Delta_{z})^\sigma}f_n(x_n+z,y_n)\dd z}\gtrsim A^{-\frac{\theta+\beta(1-\theta)}{(1-\beta)(1-\theta)}}\vare^{\frac{1}{(1-\beta)(1-\theta)}}.
		\end{align}
		Since $\T$ is bounded and $(x_n)_n\subset \{|x|\leq M\}$, we may  assume, without loss of generality, that $x_n\equiv 0$ and $y_n\equiv 0$. Next, define
		\begin{align*}
			h_n(x,y):= \ee^{-\ii    t_n(-\Delta_{x})^\sigma}f_n(x,y)
		\end{align*}
		One easily verifies that $\|h_n\|_{H_{x,y}^\sigma}=\|f_n\|_{H_{x,y}^\sigma}$ and by the $H_{x,y}^\sigma$-boundedness of $(f_n)_n$ we know that there exists some $\phi\in H_{x,y}^\sigma$ such that $h_n\rightharpoonup \phi$ weakly in $H_{x,y}^\sigma$. \eqref{loew bound}, the weak convergence of $h_n$ to $\phi$, the H\"older inequality and the embedding $H_y^\sigma \hookrightarrow L_y^\infty$ yield
		\begin{align*}
			A^{-\frac{\theta+\beta(1-\theta)}{(1-\beta)(1-\theta)}}\vare^{\frac{1}{(1-\beta)(1-\theta)}}&\lesssim
			\abso{\int_{\R^d}(\mathcal{F}_x^{-1}\chi_R)(-z)\phi(z,0)\dd z}\lesssim \|\chi_R\|_{L_x^2}\|\phi(\cdot,0)\|_{L_x^2}\nonumber\\
			&\lesssim A^{\frac{d}{2\kappa}}\vare^{-\frac{d}{2\kappa}}\|\phi\|_{L_y^\infty L_x^2}\leq
			A^{\frac{d}{2\kappa}}\vare^{-\frac{d}{2\kappa}}\|\phi\|_{L_x^2 L_y^\infty } \nonumber\\
			&\lesssim A^{\frac{d}{2\kappa}}\vare^{-\frac{d}{2\kappa}}\|\phi\|_{L_x^2 H_y^\sigma}\leq
			A^{\frac{d}{2\kappa}}\vare^{-\frac{d}{2\kappa}}\|\phi\|_{H_{x,y}^\sigma}
		\end{align*}
		and the lower bound estimate in \eqref{cnls l2 refined strichartz decomp 1} follows. Since $H_{x,y}^1$ is a Hilbert space, we infer that for $D\in\{1,(-\Delta_x)^{\frac{\sigma}{2}},(-\pt_y^2)^{\frac{\sigma}{2}}\}$
		\begin{align*}
			\|D(h_n-\phi)\|_{L_{x,y}^2}+\|D\phi\|_{L_{x,y}^2}=\|Dh_n\|_{L_{x,y}^2}+o_n(1).
		\end{align*}
		The equalities in \eqref{cnls l2 refined strichartz decomp 1} and \eqref{cnls l2 refined strichartz decomp 3} now follow from undoing the transformation form $h_n$ to $f_n$ and $\phi$ to $\phi_n$.
	\end{proof}
	
	\begin{remark}
\normalfont
		By redefining the symmetry parameters suitably we may, without loss of generality,  assume that
		$$t_n\equiv 0\quad\text{or}\quad t_n\to \pm\infty$$
		as $n\to\infty$.
	\end{remark}
	
	\begin{lemma}[Energy Pythagorean expansion]
		Let $(f_n)_n$ and $(\phi_n)_n$ be the functions from Lemma \ref{refined l2 lemma 1}. Then
		\begin{align}
			&\|f_n\|_{\alpha+2}^{\alpha+2}=\|\phi_n\|_{\alpha+2}^{\alpha+2}+\|f_n-\phi_n\|_{\alpha+2}^{\alpha+2}+o_n(1).\label{decomp tas}
		\end{align}
	\end{lemma}
	
	\begin{proof}
		Assume first $t_n\to\pm\infty$. For $\beta>0$ let $\psi\in \mathcal{S}(\R^d)\otimes C_{\mathrm{per}}^\infty(\T)$ such that $\|\phi-\psi\|_{H_{x,y}^\sigma}\leq\beta$, where $\phi$ is the same function given by Lemma \ref{refined l2 lemma 1}. Define also $\psi_n:=\ee^{\ii  t_n(-\Delta_x)^\sigma}\psi(x,y)$. Then by Lemma \ref{frac disper lem} we deduce
		\begin{align*}
			\|\psi_n\|_{L_{x,y}^{\alpha+2}}\lesssim |t_n|^{-\frac{\alpha d}{2(\alpha+2)}}\norm{(-\Delta_x)^{\frac{d(1-\sigma)}{2}(1-\frac2p)}\psi}_{L_{y}^{\alpha+2} L_x^{\frac{\alpha+2}{\alpha+1}}}\to 0.
		\end{align*}
		Now let $\zeta\in C^\infty(\R^{d+1};[0,1])$ be a cut-off function such that $\mathrm{supp}\,\zeta\subset \R^{d}\times [-2\pi,2\pi]$ and $\zeta\equiv 1$ on $\R^{d}\times [-\pi,\pi]$. Then by the Sobolev  inequality on $\R^d\times\T$, product rule, and periodicity along the $y$-direction, we deduce
		\begin{align}
			\|\psi_n-\phi_n\|_{L_{x,y}^{\alpha+2}(\R^d\times\T)}
			&\leq\|\zeta(\psi_n-\phi_n)\|_{L_{x,y}^{\alpha+2}(\R^{d+1})}
			\lesssim \|\zeta(\psi_n-\phi_n)\|_{H_{x,y}^\sigma(\R^{d+1})}\nonumber\\
			&\lesssim \|\psi_n-\phi_n\|_{H_{x,y}^\sigma(\R^d\times\T)}\leq\beta,\label{argument gn}
		\end{align}
		which in turn implies $\|\phi_n\|_{L_{x,y}^{\alpha+2}}=o_n(1)$. Therefore by triangular inequality
		\begin{align*}
			|\|f_n\|_{\alpha+2}-\|f_n-\phi_n\|_{\alpha+2}|\leq \|\phi_n\|_{\alpha+2}=o_n(1)
		\end{align*}
		and \eqref{decomp tas} follows. Assume now $t_n\equiv 0$. Notice that by truncation arguments, the weak $H_{x,y}^1$-convergence also implies almost everywhere convergence. Thus by the Brezis-Lieb lemma, we obtain
		\begin{align*}
			\|h_n\|_{\alpha+2}^{\alpha+2} =\|\phi\|_{\alpha+2}^{\alpha+2}+\|h_n-\phi\|_{\alpha+2}^{\alpha+2}+o_n(1),
		\end{align*}
		where $h_n$ is the function given by Lemma \ref{refined l2 lemma 1}. \eqref{decomp tas} follows then by undoing the transformation.
	\end{proof}
	
	\begin{lemma}[Linear profile decomposition]\label{linear profile}
		Let $(\psi_n)_n$ be a bounded sequence in $H_{x,y}^\sigma$. Then,  there exist, up to a subsequence, nonzero linear profiles
		$(\tdu^j)_j\subset H_{x,y}^\sigma$, remainders $(w_n^k)_{k,n}\subset H_{x,y}^\sigma$, time translation $(t^j_n)_{j,n}\subset\R$ and $K^*\in\N\cup\{\infty\}$, such that
		\begin{itemize}
			\item[(i)] For any finite $1\leq j\leq K^*$ the parameter $t^j_n$ satisfies
			\begin{align}
				t^j_n\equiv 0\quad\text{or}\quad \lim_{n\to\infty}t^j_n= \pm\infty.
			\end{align}
			
			\item[(ii)]For any finite $1\leq k\leq K^*$ we have the decomposition
			\begin{align}\label{cnls decomp lemma}
				\psi_n=\sum_{j=1}^k T_n^j\phi^j(x,y)+w_n^k=:\sum_{j=1}^k \ee^{\ii  t_n(-\Delta_x)^\sigma}\phi^j(x,y)+w_n^k.
			\end{align}
			
			\item[(iii)] The remainders $(w_n^k)_{k,n}$ satisfy
			\begin{align}\label{cnls to zero wnk lemma}
				\lim_{k\to K^*}\lim_{n\to\infty}\|\ee^{-\ii   t(-\Delta_x)^\sigma}w_n^k\|_{L_t^\ba L_{x,y}^\br(\R)}=0.
			\end{align}
			
			\item[(iv)] The time translations are orthogonal in the sense that for any $j\neq k$
			\begin{align}\label{cnls orthog of pairs lemma}
				|t_n^k-t_n^j|\to\infty
			\end{align}
			as $n\to\infty$.
			
			\item[(v)] For any finite $1\leq k\leq K^*$ and $D\in\{1,(-\Delta_x)^{\frac{\sigma}{2}},(-\pt_y^2)^{\frac{\sigma}{2}}\}$ we have the energy decompositions
			\begin{align}
				\|D\psi_n\|_{L_{x,y}^2}^2&=\sum_{j=1}^k\|D(T_n^j\tdu^j)\|_{L_{x,y}^2}^2+\|Dw_n^k\|_{L_{x,y}^2}^2+o_n(1),\label{orthog L2 lemma}\\
				\|\psi_n\|_{\alpha+2}^{\alpha+2}&=\sum_{j=1}^k\|T_n^j\tdu^j\|_{\alpha+2}^{\alpha+2}
				+\|w_n^k\|_{\alpha+2}^{\alpha+2}+o_n(1)\label{cnls conv of h lemma}.
			\end{align}
		\end{itemize}
	\end{lemma}
	
	\begin{proof}
		We construct the linear profiles iteratively and start with $k=0$ and $w_n^0:=\psi_n$. We assume initially that the linear profile decomposition is given and its claimed properties are satisfied for some $k$. Define
		\begin{align*}
			\vare_{k}:=\lim_{n\to\infty}\|\ee^{-\ii   t(-\Delta_x)^\sigma}w_n^k\|_{L_t^\ba L_{x,y}^\br(\R)}.
		\end{align*}
		If $\vare_k=0$, then we stop and set $K^*=k$. Otherwise we apply Lemma \ref{refined l2 lemma 1} to $w_n^k$ to obtain the sequence $(\tdu^{k+1},w_n^{k+1},t_n^{k+1},x_n^{k+1})_{n}.$ We should still need to check that items (iii) and (iv) are satisfied for $k+1$. That the other items are also satisfied for $k+1$ follows directly from the construction of the linear profile decomposition. If $\vare_k=0$, then item (iii) is automatic; otherwise, we have $K^*=\infty$. Using \eqref{cnls l2 refined strichartz decomp 1} and \eqref{orthog L2 lemma} we obtain
		\begin{align*}
			\sum_{j\in \N}A^{-\frac{2(\theta+\beta(1-\theta))}{(1-\beta)(1-\theta)}-\frac{d}{\kappa}}\vare^{\frac{2}{(1-\beta)(1-\theta)}+\frac{d}{\kappa}}
			\lesssim \sum_{j\in \N}\|\tdu^j\|^2_{H_{x,y}^1}
			=\sum_{j\in \N}\lim_{n\to\infty}\|T_n^j \phi^j\|^2_{H_{x,y}^1}
			\leq \lim_{n\to\infty}\|\psi_n\|^2_{H_{x,y}^1}= A_0^2,
		\end{align*}
		where $A_j:=\lim_{n\to\infty}\|w_n^j\|_{H_{x,y}^1}$. Hence
		\begin{align*}
			A^{-\frac{2(\theta+\beta(1-\theta))}{(1-\beta)(1-\theta)}-\frac{d}{\kappa}}\vare^{\frac{2}{(1-\beta)(1-\theta)}+\frac{d}{\kappa}}\to 0\quad\text{as $j\to\infty$}.
		\end{align*}
		By \eqref{orthog L2 lemma} we know that $(A_j)_j$ is monotone decreasing, thus also bounded. The boundedness of $(A_j)_j$ then implies  $\vare_j\to 0$ as $j\to\infty $ and the proof of item (iii) is complete. Finally, we take the item (iv). Assume inductively that item (iv) does not hold for some $j<k$ but holds for all pairs $(i_1,i_2)$ with $i_1<i_2\leq j$. We may, without loss of generality, assume that the inductive basis is satisfied, otherwise from the following contradiction proof (where no inductive assumption is invoked in the base case) the algorithm already stops at $k=0$. By construction of the profile decomposition we have
		\begin{align*}
			w_n^{k-1}=w_n^j-\sum_{l=j+1}^{k-1} \ee^{\ii  t_n(-\Delta_x)^\sigma} \tdu^l(x).
		\end{align*}
		Then by definition of $\tdu^k$ we know that
		\begin{align*}
			\tdu^k&=\wlim_{n\to\infty}\ee^{-\ii   t_n^k(-\Delta_x)^\sigma}w_n^{k-1}(x,y)\nonumber\\
			&=\wlim_{n\to\infty}\ee^{-\ii   t_n^k(-\Delta_x)^\sigma}w_n^{j}(x,y)-\sum_{l=j+1}^{k-1}\wlim_{n\to\infty}e^{
				-i(t_n^k-t_n^l)(-\Delta_x)^\sigma}\tdu^l(x,y),
		\end{align*}
		where the weak limits are taken in the $H_{x,y}^\sigma$-topology. We aim to show $\tdu^k$ is zero, which leads to a contradiction and proves item (iv). For the first summand, we obtain that
		\begin{align*}
			\ee^{-\ii   t_n^k(-\Delta_x)^\sigma}w_n^{j}(x,y)
			=(\ee^{-\ii (t_n^k-t_n^j)(-\Delta_x)^\sigma})[\ee^{-\ii   t_n^j(-\Delta_x)^\sigma}w_n^{j}(x,y)],
		\end{align*}
		Then the failure of the item (iv) will lead to the strong convergence of the adjoint of $\ee^{-\ii (t_n^k-t_n^j)(-\Delta_x)^\sigma}$ in $H_{x,y}^\sigma$. On the other hand, by the construction of the profile decomposition (see the proof of Lemma \ref{refined l2 lemma 1}) we have
		\begin{align*}
			\ee^{-\ii   t_n^j(-\Delta_x)^\sigma}w_n^{j}(x,y)\rightharpoonup 0\quad\text{in $H_{x,y}^\sigma$}
		\end{align*}
		and we conclude that the first summand weakly converges to zero in $H_{x,y}^\sigma$. Now we consider the single terms in the second summand. We can rewrite every single summand to
		\begin{align*}
			\ee^{-\ii (t_n^k-t_n^l)(-\Delta_x)^\sigma}\tdu^l(x,y)
			=(\ee^{-\ii (t_n^k-t_n^j)(-\Delta_x)^\sigma})[\ee^{-\ii (t_n^j-t_n^l)(-\Delta_x)^\sigma}\tdu^l(x,y)].
		\end{align*}
		By the previous arguments, it suffices to show that
		\begin{align*}
			I_n:=\ee^{-\ii (t_n^j-t_n^l)(-\Delta_x)^\sigma}\tdu^l(x,y)\rightharpoonup 0\quad\text{in $H_{x,y}^\sigma$}.
		\end{align*}
		Due to the inductive hypothesis, we know that item (iv) is satisfied for the pair $(j,l)$. Then the weak convergence of $I_n$ to zero in $H_{x,y}^\sigma$ follows immediately from the dispersive estimate. This completes the desired proof of item (iv).
	\end{proof}
	\begin{remark}\label{remark interpolation}
\normalfont
		Let $s\in(\frac12,\sigma-s_\alpha)$ with $s_\alpha=\frac{d}{2}-\frac{2\sigma}{\alpha}\in (0,\frac12)$. Using an interpolation, the Strichartz estimate, the embedding $L_y^\br\hookrightarrow L_y^2$ and arguing as in \eqref{or1} and \eqref{or2} we obtain
		\begin{align}
			&\,\lim_{k\to K^*}\lim_{n\to\infty}\|\ee^{-\ii   t(-\Delta)^\sigma}w_n\|_{\diag H_y^{s}(\R)}\nonumber\\
			\lesssim&\lim_{k\to K^*}\lim_{n\to\infty}\bg(\|\ee^{-\ii   t(-\Delta_{x})^\sigma}w_n\|^{1-\frac{s}{1-s_\alpha}}_{\diag L_y^2(\R)}
			\|w_n\|^{\frac{s}{1-s_\alpha}}_{H_x^{s_\alpha} H_y^{\sigma-s_\alpha}}\bg)\nonumber\\
			\lesssim&\lim_{k\to K^*}\lim_{n\to\infty}\bg(\|\ee^{-\ii   t(-\Delta_{x})^\sigma}w_n\|^{1-\frac{s}{1-s_\alpha}}_{\diag L_y^\br(\R)}
			\|w_n\|^{\frac{s}{1-s_\alpha}}_{H_{x,y}^\sigma}\bg)
			=0.\label{interpolation remainder}
		\end{align}
	\end{remark}
	\subsection{Dynamical properties of   $t\mapsto \mK(u^t)$ and $c\mapsto m_c$}\label{sec dynamic}
	In the following, we prove some useful properties of the mappings $t\mapsto \mK(u^t)$ and $c\mapsto m_c$, where $u^t$ is the scaling operator defined through
	\begin{align}\label{def of scaling op}
		u^t(x,y):=t^{\frac d2}u(tx,y).
	\end{align}


	These properties will play a central role in the proof of our main results. Nevertheless, the proofs given in this subsection are almost identical to the ones given in Section \ref{sec3.4}, we thus omit the details of the proofs here.
	\begin{lemma}[Property of the mapping $t\mapsto \mK(u^t)$]\label{monotoneproperty}
		Let $c>0$ and $u\in S(c)$. Then the following statements hold true:
		\begin{enumerate}
			\item[(i)] $\frac{\dd}{\dd t}\mH(u^t)=t^{-1} Q(u^t)$ for all $t>0$.
			\item[(ii)] There exists some $t^*=t^*(u)>0$ such that $u^{t^*}\in V(c)$.
			\item[(iii)] We have $t^*<1$ if and only if $\mK(u)<0$. Moreover, $t^*=1$ if and only if $\mK(u)=0$.
			\item[(iv)] Following inequalities hold:
			\begin{equation*}
				Q(u^t) \left\{
				\begin{array}{lr}
					>0, &t\in(0,t^*) ,\\
					<0, &t\in(t^*,\infty).
				\end{array}
				\right.
			\end{equation*}
			\item[(v)] $\mH(u^t)<\mH(u^{t^*})$ for all $t>0$ with $t\neq t^*$.
		\end{enumerate}
	\end{lemma}

\begin{lemma}[Property of the mapping $c\mapsto m_c$]\label{monotone lemma}
		
		The mapping $c\mapsto m_c$ is continuous and monotone decreasing on $(0,\infty)$.
	\end{lemma}

\begin{lemma}[Equivalent characterization of $m_c$]\label{step 1}
		Define
		\begin{align}
			\tilde{m}_c:=\inf\{\mI(u):u\in S(c),\mK(u)\leq 0\},\label{mtilde equal m}
		\end{align}
		Then $m_c=\tilde{m}_c$.
	\end{lemma}

	\subsection{The MEI-functional and its properties}
We introduce in the following the mass-energy indicator (MEI) functional which was firstly exploited in \cite{killip_visan_soliton} for the study of the cubic-quintic NLS on $\R^3$. To begin with, we first define the domain $\Omega\subset \R^2$ by
	\begin{align}
		\Omega&:=\bg((-\infty,0]\times \R\bg)\cup\bg\{(c,h)\in\R^2:c\in(0,\infty),h\in(-\infty,m_c)\bg\}.
	\end{align}
	Then we define the MEI-functional $\mD:\R^2\to [0,\infty]$ by
	\begin{align}\label{cnls MEI functional}
		\mD(c,h)=\left\{
		\begin{array}{ll}
			h+\frac{h+c}{\mathrm{dist}((c,h),\Omega^c)},&\text{if $(c,h)\in \Omega$},\\
			\infty,&\text{otherwise}.
		\end{array}
		\right.
	\end{align}
	For $u\in H_{x,y}^1$, define $\mD(u):=\mD(\mM(u),\mH(u))$. We also define the set $\mA$ by
	\begin{align}
		\mA&:=\{u\in H_{x,y}^\sigma:\mH(u)<m_{\mM(u)},\,\mK(u)>0\}.
	\end{align}
	By conservation of mass and energy, we know that if $u$ is a solution of \eqref{nls}, then $\mD(u(t))$ is a conserved quantity. In the following, we hence simply write $\mD(u)=\mD(u(t))$ as long as $u$ is a solution of \eqref{nls}.

We next collect some useful properties of the MEI-functional. For their proofs, we refer to \cite{Luo_inter}.

	
	\begin{lemma}\label{invariance from mA}
		Let $u$ be a solution of \eqref{nls} and assume that there exists some $t$ in the lifespan of $u$ such that $u(t)\in\mA$. Then $u(t)\in\mA$ for all $t$ in the maximal lifespan of $u$.
	\end{lemma}

	\begin{lemma}\label{lemma coercivity}
		Let $u\in\mA$. Then
		\begin{align}
			\bg(\frac12-\frac{2\sigma}{\alpha d}\bg)\|(-\Delta)^\frac{\sigma}{2}u\|_2^2&\leq \mH(u)\leq\frac{1}{2}\|(-\Delta)^\frac{\sigma}{2}u\|_2^2\label{inq coercivity}.
		\end{align}
	\end{lemma}
	
	\begin{proof}
		On the one hand, since the nonlinear potential energy $\|u\|_{\alpha+2}^{\alpha+2}$ has negative sign in $\mH(u)$, we have $\mH(u)\leq \frac{1}{2}\|(-\Delta)^\frac{\sigma}{2} u\|_2^2$. On the other hand, using $\mK(u)>0$ for $u\in\mA$ we deduce
		\begin{align*}
			\mH(u)>\mH(u)-\frac{2}{\alpha d}\mK(u)=\bg(\frac12-\frac{2\sigma}{\alpha d}\bg)\|(-\Delta_{x})^\frac{\sigma}{2} u\|_2^2+\|(-\Delta_{y})^\frac{\sigma}{2} u\|_2^2
			\geq \bg(\frac12-\frac{2\sigma}{\alpha d}\bg)\|(-\Delta)^\frac{\sigma}{2} u\|_2^2.
		\end{align*}
	\end{proof}

\begin{lemma}\label{cnls killip visan curve}
		Let $u,u_1,u_2$ be functions in $H_{x,y}^\sigma$. The following statements hold true:
		\begin{itemize}
			\item[(i)] $u\in\mA \Leftrightarrow\mD(u)\in(0,\infty)$.
			
			\item[(ii)] Let $u_1,u_2\in \mA$ satisfy $\mM(u_1)\leq \mM(u_2)$ and $\mH(u_1)\leq \mH(u_2)$, then $\mD(u_1)\leq \mD(u_2)$. If in addition either $\mM(u_1)<\mM(u_2)$ or $\mH(u_1)<\mH(u_2)$, then $\mD(u_1)<\mD(u_2)$.
			
			\item[(iii)] Let $\mD_0\in(0,\infty)$. Then
			\begin{gather}
				m_{\mM(u)}-\mH(u)\gtrsim_{\mD_0} 1\label{small of unaaa},\\
				\mH(u)+\mM(u)\lesssim_{\mD_0}\mD(u)\label{mei var2}
			\end{gather}
			uniformly for all $u\in \mA$ with $\mD(u)\leq \mD_0$.
		\end{itemize}
	\end{lemma}

	\subsection{Existence of a minimal blow-up solution}
	Having all the preliminaries we are ready to construct a minimal blow-up solution of \eqref{nls}. According to Remark \ref{remark interpolation} we fix some $s\in(\frac12,\sigma-s_\alpha)$ with $s_\alpha=\frac{d}{2}-\frac{2\sigma}{\alpha}\in (0,\frac12)$. Define
	\begin{align*}
		\tau(\mD_0):=\sup\bg\{\|u\|_{\diag H_y^{s}(I_{\max})}:
		\text{ $u$ is solution of \eqref{nls}, }\mD(u)\in (0,\mD_0)\bg\}
	\end{align*}
	and
	\begin{align}\label{introductive hypothesis}
		\mD^*&:=\sup\{\mD_0>0:\tau(\mD_0)<\infty\}.
	\end{align}
	By Lemma \ref{lemma small data}, \ref{lemma scattering norm}, \ref{lemma coercivity}, and \ref{cnls killip visan curve} we know that $\mD^*>0$ and $\tau(\mD_0)<\infty$ for sufficiently small $\mD_0$. Therefore we simply assume $\mD^*<\infty$, relying on which we derive a contradiction. This in turn ultimately implies $\mD^*=\infty$ and the proof of Theorem \ref{thm large data scattering} will be complete in view of Lemma \ref{cnls killip visan curve}. By the inductive hypothesis we can find a sequence $(u_n)_n$ which are solutions of \eqref{nls} with $(u_n(0))_n\subset {\mA}$ and maximal lifespan $(I_{n})_n$ such that
	\begin{gather}
		\lim_{n\to\infty}\|u_n\|_{\diag H_y^{s}((\inf I_n,0])}=\lim_{n\to\infty}\|u_n\|_{\diag H_y^{s}([0, \sup I_n))}=\infty,\label{oo1}\\
		\lim_{n\to\infty}\mD(u_n)=\mD^*.\label{oo2}
	\end{gather}
	Up to a subsequence, we may also assume that
	\begin{align*}
		(\mM(u_n),\mH(u_n))\to(\mM_0,\mH_0)\quad\text{as $n\to\infty$}.
	\end{align*}
	By continuity of $\mD$ and finiteness of $\mD^*$ we know that
	\begin{align*}
		\mD^*=\mD(\mM_0,\mH_0),\quad
		\mM_0\in(0,\infty),\quad
		\mH_0\in[0,m_{\mM_0}).
	\end{align*}
	From Lemma \ref{lemma coercivity} and \ref{cnls killip visan curve} it follows that $(u_n(0))_n$ is a bounded sequence in $H_{x,y}^\sigma$, hence Lemma \ref{linear profile} and Remark \ref{remark interpolation} are applicable for $(u_n(0))_n$: There exist nonzero linear profiles
	$(\tdu^j)_j\subset H_{x,y}^\sigma$, remainders $(w_n^k)_{k,n}\subset H_{x,y}^\sigma$, time translations $(t^j_n)_{j,n}\subset\R$ and $K^*\in\N\cup\{\infty\}$, such that
	\begin{itemize}
		\item[(i)] For any finite $1\leq j\leq K^*$ the parameter $t^j_n$ satisfies
		\begin{align}
			t^j_n\equiv 0\quad\text{or}\quad \lim_{n\to\infty}t^j_n= \pm\infty.
		\end{align}
		
		\item[(ii)]For any finite $1\leq k\leq K^*$ we have the decomposition
		\begin{align}\label{cnls decomp}
			u_n(0)=\sum_{j=1}^k T_n^j\phi^j(x,y)+w_n^k=:\sum_{j=1}^k \ee^{\ii  t_n(-\Delta_x)^\sigma}\phi^j(x,y)+w_n^k.
		\end{align}
		
		\item[(iii)] The remainders $(w_n^k)_{k,n}$ satisfy
		\begin{align}\label{cnls to zero wnk}
			\lim_{k\to K^*}\lim_{n\to\infty}\|\ee^{-\ii   t(-\Delta)^\sigma}w_n^k\|_{L_t^\ba L_{x}^\br H_y^s(\R)}=0.
		\end{align}
		
		\item[(iv)] The time translations are orthogonal in the sense that for any $j\neq k$
		\begin{align}\label{cnls orthog of pairs}
			|t_n^k-t_n^j|\to\infty
		\end{align}
		as $n\to\infty$.
		
		\item[(v)] For any finite $1\leq k\leq K^*$ and $D\in\{1,(-\Delta_x)^{\frac{\sigma}{2}},(-\pt_y^2)^{\frac{\sigma}{2}}\}$ we have the energy decompositions
		\begin{align}
			\|D(u_n(0))\|_{L_{x,y}^2}^2&=\sum_{j=1}^k\|D(T_n^j\tdu^j)\|_{L_{x,y}^2}^2+\|Dw_n^k\|_{L_{x,y}^2}^2+o_n(1),\label{orthog L2}\\
			\|u_n(0)\|_{\alpha+2}^{\alpha+2}&=\sum_{j=1}^k\|T_n^j\tdu^j\|_{\alpha+2}^{\alpha+2}
			+\|w_n^k\|_{\alpha+2}^{\alpha+2}+o_n(1)\label{cnls conv of h}.
		\end{align}
	\end{itemize}
	We now define the nonlinear profiles as follows:
	\begin{itemize}
		\item For $t^k_\infty=0$, we define $u^k$ as the solution of \eqref{nls} with $u^k(0)=\tdu^k$.
		
		\item For $t^k_\infty\to\pm\infty$, we define $u^k$ as the solution of \eqref{nls} that scatters forward (backward) to $\ee^{\ii  t(-\Delta_{x})^\sigma}\tdu^k$ in $H_{x,y}^1$.
	\end{itemize}
	In both cases we define
	\begin{align*}
		u_n^k:=u^k(t-t^k_n,x,y).
	\end{align*}
	Then $u_n^k$ is also a solution of \eqref{nls}. In both cases, we have for each finite $1\leq k \leq K^*$
	\begin{align}\label{conv of nonlinear profiles in h1}
		\lim_{n\to\infty}\|u_n^k(0)-T_n^k \tdu^k\|_{H_{x,y}^\sigma}=0.
	\end{align}
	
	In the following, we establish a Palais-Smale type lemma which is essential for the construction of the minimal blow-up solution.
	
	\begin{lemma}[Palais-Smale-condition]\label{Palais Smale}
		Let $(u_n)_n$ be a sequence of solutions of \eqref{nls} with maximal lifespan $I_n$, $u_n\in\mA$ and $\lim_{n\to\infty}\mD(u_n)=\mD^*$. Assume also that there exists a sequence $(t_n)_n\subset\prod_n I_n$ such that
		\begin{align}\label{precondition}
			\lim_{n\to\infty}\|u_n\|_{\diag H_y^{s}((\inf I_n,\,t_n])}=\lim_{n\to\infty}\|u_n\|_{\diag H_y^{s}([t_n,\,\sup I_n)}=\infty.
		\end{align}
		Then up to a subsequence, the sequence $(u_n(t_n, \cdot))_n$ strongly converges in $H_{x,y}^\sigma$.
	\end{lemma}
	\begin{proof}
		By time translation invariance we may assume that $t_n\equiv 0$. Let $(u_n^j)_{j,n}$ be the nonlinear profiles corresponding to the linear profile decomposition of $(u_n(0))_n$. Define
		\begin{align*}
			\Psi_n^k:=\sum_{j=1}^k u_n^j+\ee^{-\ii   t(-\Delta)^\sigma}w_n^k.
		\end{align*}
		We shall show that there exists exactly one non-trivial ``bad'' linear profile, relying on which the desired claim follows. We divide the remaining proof into four steps.
		\subsubsection*{Step 1: Decomposition of energies of the linear profiles}
		We first show that for a given nonzero linear profile $\phi^j$ we have
		\begin{align}
			\mH(T_n^j\phi^j)&> 0,\label{bd for S}\\
			\mK(T_n^j\phi^j)&> 0\label{pos of K}
		\end{align}
		for all sufficiently large $n=n(j)\in\N$. Since $\phi^j\neq 0$ we know that $T_n^j\phi^j\neq 0$ for all sufficiently large $n$. Suppose now that \eqref{pos of K} does not hold. Up to a subsequence, we may assume that $\mK(T_n^j \phi^j)\leq 0$ for all sufficiently large $n$. Recall the energy functional $\mI$ defined by \eqref{def of mI}. Using \eqref{orthog L2} and \eqref{cnls conv of h} we infer that
		\begin{align}
			\mI(u_n(0))&=\sum_{j=1}^k\mI(T_n^j\tdu^j)+\mI(w_n^k)+o_n(1)\label{conv of i}.
		\end{align}
		By the non-negativity of $\mI$, \eqref{conv of i} and \eqref{small of unaaa} we know that there exists some sufficiently small $\delta>0$ depending on $\mD^*$ and some sufficiently large $N_1$ such that for all $n>N_1$ we have
		\begin{align}\label{contradiction1}
			\tm_{\mM(T_n^j\phi^j)}\leq\mI(T_n^j\phi^j)\leq \mI(u_n(0))+\delta
			\leq\mH(u_n(0))+\delta\leq m_{\mM(u_n(0))}-2\delta,
		\end{align}
		where $\tm$ is the quantity defined by \eqref{mtilde equal m}. By continuity of $c\mapsto m_c$ we also know that for sufficiently large $n$ we have
		\begin{align}\label{contradiction3}
			m_{\mM(u_n(0))}-2\delta\leq m_{\mM_0}-\delta.
		\end{align}
		Using \eqref{orthog L2} we deduce that for any $\vare>0$ there exists some large $N_2$ such that for all $n>N_2$ we have
		\begin{align*}
			\mM(T_n^j\phi^j)\leq \mM_0+\vare.
		\end{align*}
		From the continuity and monotonicity of $c\mapsto m_c$ and Lemma \ref{step 1}, we may choose some sufficiently small $\vare$ to see that
		\begin{align}\label{contradiction2}
			\tm_{\mM(T_n^j\phi^j)}=m_{\mM(T_n^j\phi^j)}\geq m_{\mM_0+\vare}\geq m_{\mM_0}-\frac{\delta}{2}.
		\end{align}
		Now \eqref{contradiction1}, \eqref{contradiction3} and \eqref{contradiction2} yield a contradiction. Thus \eqref{pos of K} holds, which combined with Lemma \ref{lemma coercivity} also yields \eqref{bd for S}. Similarly, for each $j\in\N$ we have
		\begin{align}
			\mH(w_n^j)&> 0,\label{bd for S wnj} \\
			\mK(w_n^j)&> 0\label{pos of K wnj}
		\end{align}
		for sufficiently large $n$.
		\subsubsection*{Step 2: Decoupling of nonlinear profiles}
		In this step, we show that the nonlinear profiles are asymptotically decoupled in the sense that
		\begin{align}\label{smallness aaa}
			\lim_{n\to\infty} \|\la u_n^i,u_n^j\ra_{H_y^s}\|_{L_t^{\frac{\ba}{2}}L_x^{\frac{\br}{2}}(\R)}=0
		\end{align}
		for any fixed $1\leq i,j\leq K^*$ with $i\neq j$, provided that
		$$\limsup_{n\to\infty}\,(\|u_n^i\|_{\diag H_y^{s}(\R)}+\|u_n^j\|_{\diag H_y^{s}(\R)})<\infty.$$
		We claim that for any $\beta>0$ there exists some $\psi^i_\beta,\psi_\beta^j\in C_c^\infty(\R\times\R^d)\otimes C_{\mathrm{per}}^\infty(\T)$ such that
		\begin{align*}
			\bg\|u^i_n-\psi^i_\beta(t-t^i_n,x,y)\bg\|_{\diag H_y^{s}(\R)}\leq \beta,\\
			\bg\|u^j_n-\psi^j_\beta(t-t^j_n,x,y)\bg\|_{\diag H_y^{s}(\R)} \leq \beta.
		\end{align*}
		Indeed, we may simply choose some $\psi^i_\beta,\psi^j_\beta\in C_c^\infty(\R\times\R^d)$ such that
		\begin{align*}
			\|u^i-\psi^i_\beta\|_{\diag H_y^{s}(\R)}\leq \beta,\,\|u^j-\psi^j_\beta\|_{\diag H_y^{s}(\R)}\leq \beta,
		\end{align*}
		where $u^k$ are solutions of \eqref{nls} with $u^k(0)=\phi^k$ ($t^k_\infty=0$) or $u^k$ scatters forward (backward) to $\ee^{\ii  t(-\Delta_{x})^\sigma}\phi^k$ in $H_{x,y}^\sigma$ ($t^k_n\to\pm\infty$). Define now
		$$ \Lambda_n (\psi_\beta^\ell):=\psi^\ell_\beta(t-t^\ell_n,x,y)$$
		for $\ell\in\{i,j\}$. Using the H\"older inequality we infer that
		\begin{align*}
			\|\la u_n^i,u_n^j\ra_{H_y^s}\|_{L_t^{\frac{\ba}{2}}L_x^{\frac{\br}{2}}(\R)}\lesssim \beta+\|\la \Lambda_n (\psi_\beta^i),\Lambda_n (\psi_\beta^j)\ra_{H_y^s}\|_{L_t^{\frac{\ba}{2}}L_x^{\frac{\br}{2}}(\R)}.
		\end{align*}
		Since $\beta$ can be chosen arbitrarily small, it suffices to show
		\begin{align}\label{step2a1}
			\lim_{n\to\infty}\|\la \Lambda_n (\psi_\beta^i),\Lambda_n (\psi_\beta^j)\ra_{H_y^s}\|_{L_t^{\frac{\ba}{2}}L_x^{\frac{\br}{2}}(\R)}=0.
		\end{align}
		However, since $\psi^i_\beta$ and $\psi^j_\beta$ are compactly supported in $\R_t\times\R^d_x$, \eqref{step2a1} follows immediately from the orthogonality of the parameters $(t_n^k)_{k,n}$ given by \eqref{cnls orthog of pairs}.
		\subsubsection*{Step 3: Existence of at least one bad profile}
		First we prove that there exists some $1\leq J\leq K^*$ such that for all $j\geq J+1$ and all sufficiently large $n$, $u_n^j$ is global and
		\begin{align}\label{uniform bound of unj}
			\lim_{n\to\infty}\sum_{J+1\leq j\leq K^*}\|u_n^j\|^2_{\diag H_y^{s}(\R)}\lesssim 1.
		\end{align}
		Indeed, using \eqref{orthog L2} we infer that
		\begin{align}\label{small initial data}
			\lim_{k\to K^*}\lim_{n\to\infty}\sum_{j=1}^k\|T_n^j \tdu^j\|^2_{H_{x,y}^\sigma}<\infty.
		\end{align}
		Then \eqref{uniform bound of unj} follows from Lemma \ref{lemma small data} and \ref{lemma scattering norm}. We now claim that there exists some $1\leq J_0\leq J$ such that
		\begin{align}
			\limsup_{n\to\infty}\|u_n^{J_0}\|_{\diag H_y^{s}(\R)}=\infty.
		\end{align}
		We argue by contradiction and assume that
		\begin{align}\label{uniform bound of unj small}
			\limsup_{n\to\infty}\|u_n^j\|_{\diag H_y^{s}(\R)}<\infty\quad\forall\,1\leq j\leq J.
		\end{align}
		To proceed, we first show
		\begin{align}\label{kkkk uniform bound of unj}
			\sup_{J+1\leq k\leq K^*}\lim_{n\to\infty}\bg\|\sum_{j=J+1}^k u_n^j\bg\|_{\diag H_y^{s}(\R)}
			\lesssim 1.
		\end{align}
		Indeed, using triangular inequality, \eqref{smallness aaa} and \eqref{uniform bound of unj} we immediately obtain
		\begin{align*}
			&\,\sup_{J+1\leq k\leq K^*}\lim_{n\to\infty}\bg\|\sum_{j=J+1}^k u_n^j\bg\|_{\diag H_y^s(\R)}\nonumber\\
			\lesssim&\, \sup_{J+1\leq k\leq K^*}\lim_{n\to\infty}\bg(\bg(\sum_{j=J+1}^k
			\|u_n^j\|^2_{\diag H_y^{s}(\R)}\bg)^{\frac12}
			+\bg(\sum_{i,j=J+1,i\neq j}^k\|\la u_n^i, u_n^j\ra_{H_y^s}\|_{L_t^{\frac{\ba}{2}}L_x^{\frac{\br}{2}}(\R)}\bg)^{\frac12}\bg)\lesssim 1.
		\end{align*}
		Combining \eqref{kkkk uniform bound of unj} with \eqref{uniform bound of unj small} we deduce
		\begin{align}\label{super uniform}
			\sup_{1\leq k\leq K^*}\lim_{n\to\infty}\bg\|\sum_{j=1}^k u_n^j\bg\|_{\diag H_y^{s}(\R)}
			\lesssim  1.
		\end{align}
		Therefore, using \eqref{orthog L2}, \eqref{conv of nonlinear profiles in h1} and the Strichartz estimate we confirm that the conditions \eqref{cond 1} and \eqref{cond 3} are satisfied for sufficiently large $k$ and $n$, where we set $u=u_n$ and $\tilde{u}=\Psi_n^k$ in Lemma \ref{lem stability cnls}. As long as we can show that \eqref{cond 2} is satisfied for the chosen $s\in(\frac{1}{2},\sigma-s_\alpha)$, we are able to apply Lemma \ref{lem stability cnls} to obtain the contradiction
		\begin{align}\label{contradiction 1}
			\limsup_{n\to\infty}\|u_n\|_{\diag H_y^{s}(\R)}<\infty.
		\end{align}
		Direct calculation shows
		\begin{align*}
			e&=\,i\pt_t\Psi_n^k-(-\Delta)^\sigma\Psi_n^k+|\Psi_n^k|^{\alpha}\Psi_n^k\nonumber\\
			&=\bg(\sum_{j=1}^k (i\pt_tu_n^j-(-\Delta)^\sigma u_n^j)+|\sum_{j=1}^ku_n^j|^{\alpha}\sum_{j=1}^k u_n^j\bg)
			+\bg(|\Psi_n^k|^{\alpha}\Psi_n^k-|\Psi_n^k-\ee^{\ii  t\Delta_{x,y}}w_n^k|^{\alpha}(\Psi_n^k-\ee^{\ii  t\Delta_{x,y}}w_n^k)\bg)\nonumber\\
			&=:I_1+I_2.
		\end{align*}
		In the following, we show the asymptotic smallness of $I_1$ and $I_2$. Since $u_n^j$ solves \eqref{nls}, we can rewrite $I_1$ to
		\begin{align*}
			I_1
			=-\bg(\sum_{j=1}^k|u_n^j|^{\alpha}u_n^j-\bg|\sum_{j=1}^ku_n^j\bg|^{\alpha}\sum_{j=1}^ku_n^j\bg)
			\lesssim \sum_{i,j=1,i\neq j}^k|u_n^i|^{\alpha}|u_n^j|=:I_{11}.
		\end{align*}
		By the H\"older inequality we have
		\begin{align*}
			\||u_n^i|^\alpha|u_n^j|\|_{L_t^{\bb'}L_x^{\bs'}L_y^2(\R)}&\lesssim \|u_n^{i}u_n^{j}\|^{\frac12}_{L_t^{\frac{\ba}{2}}L_t^{\frac{\br}{2}}L_y^1(\R)}
			\|u_n^i\|^{\alpha-\frac12}_{\diag L_y^\infty}\|u_n^j\|^{\frac12}_{\diag L_y^\infty}\nonumber\\
			&\lesssim \|u_n^iu_n^j\|^{\frac12}_{L_t^{\frac{\ba}{2}}L_t^{\frac{\br}{2}}L_y^1}
			\|u_n^i\|^{\alpha-\frac12}_{\diag H_y^\sigma}\|u_n^j\|^{\frac12}_{\diag H_y^\sigma}.
		\end{align*}
		Then \eqref{uniform bound of unj}, \eqref{uniform bound of unj small} and \eqref{smallness aaa} imply
		\begin{align*}
			\lim_{k\to K^*}\lim_{n\to\infty}\| I_{11}\|_{L_t^{\bb'}L_x^{\bs'}L_y^2(\R) }=0.
		\end{align*}
		On the other hand, by Lemma \ref{lemma scattering norm} and the telescoping arguments given in the proof of Lemma \ref{lem stability cnls} we obtain for $s'\in(s,\sigma-s_\alpha)$
		\begin{align*}
			\||u_n^i|^\alpha u_n^j\|_{L_t^{\bb'}L_x^{\bs'}H_y^{s'}(\R)}\lesssim \|u_n^i\|_{\diag H_y^{s'}(\R)}^\alpha\|u_n^j\|_{\diag H_y^{s'}(\R)}\lesssim1.
		\end{align*}
		Combining with the inequality $\|f\|_{H_y^{s}}\leq\|f\|_{L_y^2}^{1-s/s'}\|f\|_{H_y^{s'}}^{s/s'}$ we infer that
		\begin{align*}
			\lim_{k\to K^*}\lim_{n\to\infty}\| I_{11}\|_{L_t^{\bb'}L_x^{\bs'}H_y^{s}(\R)}=0
		\end{align*}
		Next, we prove the asymptotic smallness of $I_2$. Direct calculation shows
		\begin{align*}
			I_2=O \left(\Psi_n^k(\ee^{-\ii    t(-\Delta)^\sigma}w_n^k)^\alpha+(\Psi_n^k)^\alpha \ee^{-\ii    t(-\Delta)^\sigma}w_n^k+(\ee^{-\ii    t(-\Delta)^\sigma}w_n^k)^{\alpha+1}  \right).
		\end{align*}
		But then \eqref{super uniform}, \eqref{cnls to zero wnk}, Lemma \ref{fractional on t}, the H\"older inequality and the telescoping arguments given in the proof of Lemma \ref{lem stability cnls} immediately yield
		\begin{align*}
			\lim_{k\to K^*}\lim_{n\to\infty}\| I_2\|_{L_t^{\bb'}L_x^{\bs'}H_y^{s}(\R)}=0
		\end{align*}
		and Step 3 is complete.
		\subsubsection*{Step 4: Reduction to one bad profile and conclusion}
		From Step 3 we conclude that there exists some $1\leq J_1\leq K^*$ such that
		\begin{align}
			\limsup_{n\to\infty}\|u_n^j\|_{\diag H_y^s(\R)}&=\infty\quad \forall \,1\leq j\leq J_1,\label{infinite}\\
			\limsup_{n\to\infty}\|u_n^j\|_{\diag H_y^s(\R)}&<\infty\quad \forall \,J_1+1\leq j\leq K^*.
		\end{align}
		Using \eqref{orthog L2} and \eqref{cnls conv of h} we infer that for any finite $1\leq k\leq K^*$
		\begin{align}
			\mM_0&=\sum_{j=1}^k \mM(T_n^j\tdu^j)+\mM(w_n^k)+o_n(1),\label{mo sum}\\
			\mH_0&=\sum_{j=1}^k \mH(T_n^j\tdu^j)+\mH(w_n^k)+o_n(1)\label{eo sum}.
		\end{align}
		If $J_1>1$, then using \eqref{mo sum}, \eqref{eo sum}, the asymptotic positivity of energies deduced from Step 1 and Lemma \ref{cnls killip visan curve} we know that $\limsup_{n\to\infty}\mD(u_n^1)<\mD^*$, which violates \eqref{infinite} due to the inductive hypothesis. Thus $J_1=1$ and
		$$ u_n(0,x)=\ee^{\ii  t_n^1(-\Delta_x)^\sigma}\tdu^1(x)+w_n^1(x).$$
		Similarly, we must have $\mM(w_n^1)=o_n(1)$ and $\mH(w_n^1)=o_n(1)$, otherwise we could deduce again the contradiction \eqref{contradiction 1} using Lemma \ref{lem stability cnls}. Combining with Lemma \ref{cnls killip visan curve} we conclude that $\|w_n^1\|_{H_{x,y}^\sigma}=o_n(1)$. Finally, we exclude the cases $t^1_n\to\pm \infty$. We only consider the case $t_n^1\to \infty$, the case $t_n^1 \to -\infty$ can be similarly dealt with. Indeed, using Strichartz we obtain
		\begin{align*}
			\|\ee^{\ii  t(-\Delta)^\sigma}u_n(0)\|_{\diag H_y^s([0,\infty))}\lesssim
			\|\ee^{\ii  t(-\Delta_{x})^\sigma}\tdu^1\|_{\diag H_y^s([t^1_n,\infty))}+\|w_n^1\|_{H_{x,y}^\sigma}\to 0.
		\end{align*}
		Using Lemma \ref{lemma small data} we derive the contradiction \eqref{contradiction 1} again. This completes the desired proof.
	\end{proof}
	
	\begin{lemma}[Existence of a minimal blow-up solution]\label{category 0 and 1}
		Suppose that $\mD^*\in(0,\infty)$. Then there exists a global solution $u_c$ of \eqref{nls} such that $\mD(u_c)=\mD^*$ and
		\begin{align}
			\|u_c\|_{\diag H_y^s((-\infty,0])}=\|u_c\|_{\diag H_y^s([0,\infty))}=\infty.
		\end{align}
		Moreover, $u_c$ is almost periodic in $H_{x,y}^\sigma$, i.e. the set $\{u(t):t\in\R\}$ is precompact in $H_{x,y}^\sigma$.
	\end{lemma}
	\begin{proof}
		As discussed at the beginning of this section, under the assumption $\mD^*<\infty$ one can find a sequence $(u_n)_n$ of solutions of \eqref{nls} that satisfies the preconditions of Lemma \ref{Palais Smale}. We apply Lemma \ref{Palais Smale} to infer that $(u_n(0))_n$ (up to modifying time and space translation) is precompact in $H_{x,y}^1$. We denote its strong $H_{x,y}^1$-limit by $\psi$. Let $u_c$ be the solution of \eqref{nls} with $u_c(0)=\psi$. Then $\mD(u_c(t))=\mD(\psi)=\mD^*$ for all $t$ in the maximal lifespan $I_{\max}$ of $u_c$ (recall that $\mD$ is a conserved quantity).
		
		We first show that $u_c$ is a global solution. We only show that $s_0:=\sup I_{\max}=\infty$, the negative direction can be similarly proved. If this does not hold, then by Lemma \ref{lemma small data} there exists a sequence $(s_n)_n\subset \R$ with $s_n\to s_0$ such that
		\begin{align*}
			\lim_{n\to\infty}\|u_c\|_{\diag H_y^s((-\inf I_{\max},s_n])}=\lim_{n\to\infty}\|u_c\|_{\diag H_y^s([s_n,\sup I_{\max}))}=\infty.
		\end{align*}
		Define $v_n(t):=u_c(t+s_n)$. Then \eqref{precondition} is satisfied with $t_n\equiv 0$. We then apply Lemma \ref{Palais Smale} to the sequence $(v_n(0))_n$ to conclude that there exists some $\varphi\in H_{x,y}^1$ such that, up to modifying the space translation, $u_c(s_n)$ strongly converges to $\varphi$ in $H_{x,y}^1$. But then using the Strichartz estimate we obtain for $s\in(\frac12,\sigma-s_\alpha)$
		\begin{align*}
			\|\ee^{-\ii   t(-\Delta)^\sigma}u_c(s_n)\|_{\diag H_y^s([0,s_0-s_n))}=\|\ee^{-\ii   t(-\Delta)^\sigma}\varphi\|_{\diag H_y^s([0,s_0-s_n))}+o_n(1)=o_n(1).
		\end{align*}
		By Lemma \ref{lemma small data} we can extend $u_c$ beyond $s_0$, which contradicts the maximality of $s_0$. Now by \eqref{oo1} and Lemma \ref{lem stability cnls} it is necessary that
		\begin{align}\label{blow up uc}
			\|u_c\|_{\diag H_y^s((-\infty,0])}=\|u_c\|_{\diag H_y^s([0,\infty))}=\infty.
		\end{align}
		
		We finally show that the orbit $\{u_c(t):t\in\R\}$ is precompact in $H_{x,y}^\sigma$. Let $(\tau_n)_n\subset\R$ be an arbitrary time sequence. Then \eqref{blow up uc} implies
		\begin{align*}
			\|u_c\|_{\diag H_{y}^s((-\infty,\tau_n])}=\|u_c\|_{\diag H_{y}^s([\tau_n,\infty))}=\infty.
		\end{align*}
		The claim follows by applying Lemma \ref{Palais Smale} to $(u_c(\tau_n))_n$.
	\end{proof}
	
	We end this section by establishing some useful properties of the minimal blow-up solution $u_c$.
	
	\begin{lemma}\label{lemma property of uc}
		Let $u$ be the minimal blow-up solution given by Lemma \ref{category 0 and 1}. Then
		\begin{itemize}
			\item[(i)] For each $\vare>0$ there exists $R>0$ such that
			\begin{align}
				\int_{|x|\geq R}|(-\Delta_{x})^{\frac{\sigma}{2}} u(t)|^2
				+|(-\pt_y^2)^{\frac{\sigma}{2}} u(t)|^2
				+|u(t)|^2+|u(t)|^{\alpha+2}\dd x\dd y\leq\vare\quad\forall\,t\in\R.\label{7.35}
			\end{align}
			
			\item[(ii)]There exists some $\delta>0$ such that $\inf_{t\in\R}\mK(u(t))=\delta$.
		\end{itemize}
	\end{lemma}
	\begin{proof}
		(i) is an immediate consequence of the Gagliardo-Nirenberg inequality on $\R^d\times\T$ and the almost periodicity of $u_c$ in $H_{x,y}^\sigma$. We follow the same lines in the proof of \cite[Prop. 10.3]{killip_visan_soliton} to show (ii). Assume the contrary that the claim does not hold. Then we can find a sequence $(t_n)_n\subset\R$ such that $\mK(u(t_n))\to 0$. By the almost periodicity of $u$ in $H_{x,y}^\sigma$ we can find some $v_0$ such that $u(t_n)$ converges strongly to $v_0$ in $H_{x,y}^\sigma$. Combining with the continuity of $\mK$ in $H_{x,y}^\sigma$ we infer that
		\begin{align*}
			\mD(v_0)&=\mD(u)=\mD^*\in(0,\infty),\\
			\mK(v_0)&=\lim_{n\to\infty}\mK(u(t_n))=0.
		\end{align*}
		This is however impossible due to Lemma \ref{cnls killip visan curve} (i).
	\end{proof}
	
	\subsection{Extinction of the minimal blow-up solution}
Before we prove the main result we still need the following useful lemma.
	\begin{lemma}[\cite{Blowup4nls}]
		For $\sigma\in(0,1)$ we have
		\begin{align}
			(-\Delta_x)^\sigma=\frac{\sin (\pi\sigma)}{\pi}\int_0^\infty m^{\sigma-1}(-\Delta_x)(-\Delta_x+m)^{-1}\dd m
		\end{align}
		and
		\begin{align}
			\sigma\norm{(-\Delta_x)^{\frac{\sigma}{2}}u}_{L_x^2}^2=\int_0^\infty m^\sigma\int|\nabla_x u_m|^2\dd x\dd m,\label{8.2}
		\end{align}
		where
		\begin{align}
			u_m=\sqrt{\frac{\sin (\pi\sigma)}{\pi}}(m-\Delta_x)^{-1}u.\label{um def}
		\end{align}
	\end{lemma}

	We are now ready to prove Theorem \ref{thm large data scattering}.
	
	\begin{proof}[Proof of Theorem \ref{thm large data scattering}]
		As mentioned previously, we show the contradiction that the minimal blow-up solution $u_c$ given by Lemma \ref{category 0 and 1} is equal to zero. Let $\chi:\R^d\to\R$ be a smooth radial cut-off function satisfying
		\begin{align*}
			\chi=\left\{
			\begin{array}{ll}
				|x|^2,&\text{if $|x|\leq 1$},\\
				0,&\text{if $|x|\geq 2$}.
			\end{array}
			\right.
		\end{align*}
		For $R>0$, we define the local virial action $z_R(t)$ by
		\begin{align}\label{final4}
			z_{R}(t):=2\,\mathrm{Im}\int R\nabla_x \chi\bg(\frac{x}{R}\bg)\cdot\nabla_x u(t)\bar{u}(t)\dd x\dd y.
		\end{align}
		Then using \cite[Lem. 5.1]{4NLS}, \eqref{8.2} and rearranging terms we obtain
		\begin{align}
			\pt_t z_R(t)=8\mK(u(t))+A_R(u(t)),
		\end{align}
		where
		\begin{align}
			A_R(u(t))=&\,4\int_0^\infty m^\sigma\int\bg(\pt^2_{x_j}\chi\bg(\frac{x}{R}\bg)-2\bg)|\pt_{x_j} u_m|^2\dd  x\dd  y\dd m\label{8.4}\\
			+&\,4\sum_{j\neq k}\int_0^\infty m^\sigma\int_{R\leq|x|\leq 2R}\pt_{x_j}\pt_{x_k}\chi\bg(\frac{x}{R}\bg)\pt_{x_j} u_m\pt_{x_k}\bar{u}_m\dd x\dd y\dd m\label{8.5}\\
			-&\,\frac{1}{R^2}\int_0^\infty m^\sigma\int\Delta_x^2\chi\bg(\frac{x}{R}\bg)|u_m|^2\dd x\dd y\dd m\label{8.6}\\
			-&\,\frac{2\alpha}{\alpha+2}\int\bg(\Delta_x\chi\bg(\frac{x}{R}\bg)-2d\bg)|u|^{\alpha+2}\dd x\dd y,\label{8.7}
		\end{align}
		where $u_m$ is defined by \eqref{um def}. Let $\vare>0$ be given. Then using \eqref{7.35}, \eqref{8.2}, Cauchy-Schwarz and the definition of $\chi$ we infer that
		\begin{align}
			\eqref{8.4}+\eqref{8.5}+\eqref{8.7}\lesssim \vare\label{final1a}
		\end{align}
		by choosing $R$ sufficiently large. Now from the proof of \cite[Lem. 4.2]{SunFrac} we know that $\|u_m\|_{L_x^2}\lesssim m^{-1}\|u\|_{L_x^2}$. Thus for \eqref{8.6}, we decompose $\int_0^\infty=\int_0^1+\int_1^\infty$. For $\int_1^\infty$, we estimate \eqref{8.6} using
		\begin{align}
			\lesssim R^{-2}\int_1^\infty m^{\sigma-2}\|u\|_{L_x^2}^2\dd  y\dd m\lesssim R^{-2}M(u)m^{\sigma-1}\Big|_{m=1}^\infty\lesssim R^{-2}.
		\end{align}
		Again, from the proof of \cite[Lem. 4.2]{SunFrac} and the Sobolev embedding we know that $\|u_m\|_{L_x^{\frac{2d}{d-1}}}\lesssim m^{-\frac34}\|u\|_{L_x^2}$. Combining with the H\"older inequality and the support property of $\chi$ we obtain the following estimate for the part $\int_0^1$:
		\begin{align}
			&\lesssim R^{-2}\int_0^1 m^\sigma\|u_m\|^2_{L_x^{\frac{2d}{d-1}}}\norm{\Delta_x\chi\bg(\frac{x}{R}\bg)}_{L_x^\infty}
			\norm{\Delta_x\chi\bg(\frac{x}{R}\bg)}_{L_x^d}\dd m\nonumber\\
			&\lesssim R^{-1}\int_0^1 m^{\sigma-\frac{3}{2}}\dd m\lesssim R^{-1}m^{\sigma-\frac12}\Big|_{m=0}^1\lesssim R^{-1}.
		\end{align}
		Here we also used the fact that $\sigma>\frac12$. Thus by choosing $R\gg 1$ we deduce that
		\begin{align}
			\eqref{8.6}\lesssim \vare.\label{final1b}
		\end{align}
		Now using the Cauchy-Schwarz inequality one easily sees that $\sup_{t\in\R}|z_R(t)|\lesssim R$. On the other hand, choosing $\vare\ll\delta$, where $\delta$ is given by Lemma \ref{lemma property of uc} and using \eqref{final1a} and \eqref{final1b} we obtain
		\begin{align}
			R\gtrsim z_R(t)-z_R(0)=\int_0^t \pt_t z_R(s)\dd s\gtrsim t\delta.
		\end{align}
		Sending $t\to\infty$ yields a contradiction.
	\end{proof}

	%
	%
	%
	%
	%
	\section*{Acknowledgments}
	A. Esfahani is supported by Nazarbayev University under Faculty Development Competitive Research Grants Program  for 2023-2025 (grant number 20122022FD4121).
	%
	%
	%
	%
	%
	%
	%



	\bibliographystyle{acm}
\bibliography{MainAHYL}

\end{document}